\documentclass{amsart}
\usepackage{amssymb}
\usepackage{graphics}
\usepackage{amscd}
\usepackage{enumerate}
\usepackage{comment}
\usepackage[all]{xy}
\newtheorem{thm}{Theorem}[section]
\newtheorem{cor}[thm]{Corollary}
\newtheorem{lem}[thm]{Lemma}
\newtheorem{prop}[thm]{Proposition}
\newtheorem{conj}[thm]{Conjecture}

\theoremstyle{definition}
\newtheorem{defn}[thm]{Definition}

\newtheorem{qn}[thm]{Question}
\theoremstyle{remark}
\newtheorem{rem}[thm]{Remark}
\newtheorem{ex}[thm]{Example}
\numberwithin{equation}{section}

\newcommand{\ol}{\overline}
\newcommand{\Bl}{\text{Bl}}
\newcommand{\SFH}{\mathit{SFH}}
\newcommand{\HFh}{\widehat{\mathit{HF}}}
\newcommand{\HFLh}{\widehat{\mathit{HFL}}}
\newcommand{\HFK}{\mathit{HFK}}
\newcommand{\EH}{\mathit{EH}}
\def\a{\alpha}
\def\b{\beta}
\def\g{\gamma}
\def\d{\delta}
\def\S{\Sigma}
\def\z{\zeta}
\def\t{\Theta}
\def\s{\mathfrak{s}}
\def\su{\underline{\s}}
\def\bolda{\boldsymbol{\alpha}}
\def\boldb{\boldsymbol{\beta}}

\def\bolde{\boldsymbol{\eta}}
\def\boldd{\boldsymbol{\delta}}
\def\W{\mathcal{W}}
\def\X{\mathcal{X}}
\def\Int{\text{Int}}
\def\Z{\mathbb{Z}}

\def\R{\mathbb{R}}
\def\C{\mathbb{C}}
\def\T{\mathbb{T}}
\def\F{\mathbb{F}}
\def\M{\mathcal{M}}
\def\D{\mathcal{D}}
\def\P{\mathcal{P}}

\def\V{\mathcal{V}}
\def\cT{\mathcal{T}}
\def\I{\mathcal{I}}
\def\L{\mathbb{L}}
\def\H{\mathcal{H}}
\def\ft{\mathfrak{t}}
\def\Sut{\mathbf{Sut}}
\def\BSut{\mathbf{BSut}}

\def\Vect{\mathbf{Vect}_{\mathbb{Z}_2}}
\def\link{\mathbf{DLink}}
\def\sym{\text{Sym}}
\def\x{\mathbf{x}}
\def\y{\mathbf{y}}
\def\spinc{\text{Spin}^c}
\def\spincu{\underline{\text{Spin}}^c}

\def\j{\mathfrak{j}}
\def\bP{\mathbb{P}}
\def\Diff{\text{Diff}}
\def\SS{\mathbb{S}}

\def\Id{\text{Id}}

\begin{document}

\title[Cobordisms of sutured manifolds]{Cobordisms of sutured manifolds and the functoriality of link Floer homology}%
\author{Andr\'as Juh\'asz}%
\address{Mathematical Institute, University of Oxford, Andrew Wiles Building,
Radcliffe Observatory Quarter, Woodstock Road, Oxford, OX2 6GG, UK}%
\email{juhasza@maths.ox.ac.uk}%

\thanks{Research supported by a Royal Society Research Fellowship.}

% ----------------------------------------------------------------
\begin{abstract}
It has been a central open problem in Heegaard Floer theory
whether cobordisms of links induce homomorphisms on the associated link Floer homology groups.
We provide an affirmative answer by introducing a natural notion of cobordism between sutured manifolds,
and showing that such a cobordism induces a map on sutured Floer homology.
This map is a common generalization
of the hat version of the closed 3-manifold cobordism map in Heegaard Floer theory,
and the contact gluing map defined by
Honda, Kazez, and Mati\'c. We show that sutured Floer homology,
together with the above cobordism maps, forms a type of TQFT in the
sense of Atiyah. Applied to the sutured manifold cobordism complementary to a decorated link cobordism,
our theory gives rise to the desired map on link Floer homology.
Hence, link Floer homology is a categorification of the multi-variable Alexander polynomial.
We outline an alternative definition of the contact gluing map using only the contact
element and handle maps.
Finally, we show that a Weinstein sutured manifold cobordism preserves the contact element.
\end{abstract}

\maketitle
% ----------------------------------------------------------------
\section{Introduction}

Knot Floer homology, introduced independently by Ozsv\'ath and Szab\'o~\cite{OSz3}, and Rasmussen~\cite{Ras},
and later generalized to links by Ozsv\'ath and Szab\'o~\cite{OSz2}, has proven to be a very sensitive
invariant of knots and links. The graded Euler characteristic of link Floer homology is the  multi-variable
Alexander polynomial, and it completely determines the Thurston norm of the link complement~\cite{OSz7}.
Furthermore, it determines whether a knot is fibered according to the work of Ghiggini~\cite{Ghiggini}, Ni~\cite{fibred, corrigendum},
and the author~\cite{sutured, decomposition}. The simplest proofs of these
results uses sutured Floer homology, an invariant of sutured manifolds defined by the author.

Since the introduction of knot Floer homology, it has been a natural question whether
knot cobordisms induce maps on knot Floer homology, exhibiting it as
a categorification of the Alexander polynomial. Compare this with the
results of Jacobsson~\cite{Jacobsson} and Khovanov~\cite{Khovanov} that Khovanov
homology is a categorification of the Jones polynomial, also see~\cite{CMW}.
As link Floer homology turns out to be an invariant of based links~\cite{naturality},
and in many cases moving a basepoint around the link induces a non-trivial
automorphism of link Floer homology due to the work of Sarkar~\cite{basepoint},
it is necessary to endow the link cobordisms with some decorations.
One of the main results of this paper is that decorated link cobordisms
induce functorial maps on link Floer homology. We are going to study the properties
of these maps in a forthcoming paper.

Note that, via grid diagrams and grid movies, Sarkar~\cite{tau} has presented a candidate for link cobordism maps
induced on $\HFK^-$ by surfaces in~$S^3 \times I$.
However, to date, we only know that this map is invariant under eight of the fifteen
marked movie moves by the work of Graham~\cite{Graham}.
So, while it is relatively simple to define link cobordism maps via a Morse-theoretic approach,
showing independence of the Morse function seems to be impractical.
Instead, we approach the problem via sutured manifold theory,
as this is much more general and also provides maps induced by cobordisms
of 3-manifolds with boundary, not just link complements.

Hence, we show that cobordisms of sutured manifolds induce maps on sutured Floer homology,
a Heegaard Floer type invariant of $3$-manifolds with boundary introduced by the author~\cite{sutured}.
Cobordism maps in Heegaard Floer homology were first outlined by Ozsv\'ath and Szab\'o~\cite{OSz10}
for cobordisms between closed $3$-manifolds, but their work did not address two fundamental questions.
The first was the issue of assigning a well-defined Heegaard Floer group -- not just an isomorphism class --
to a $3$-manifold, and the functoriality of the construction under diffeomorphisms. We addressed this
with Dylan Thurston~\cite{naturality}.
The second issue was exhibiting the independence of their cobordism maps of the surgery description of the cobordism. They did
check invariance under Kirby moves, but did not address how this gives rise to a well-defined map without
running into naturality issues. I gave a general framework for constructing cobordism maps -- and TQFTs in particular --
via surgery in~\cite{surgery}, and the present work is the first application of that framework in the Heegaard Floer setting.
The first version of this paper was posted online in 2009. Shortly  thereafter, I discovered
the above-mentioned naturality issues that we fixed in~\cite{naturality} and~\cite{surgery}.
Hence, the completion of this work has been considerably delayed, and should be viewed as
the culmination of all that foundational work of the past six years.

Sutured manifolds, introduced by Gabai~\cite{Gabai}, have been of great use in $3$-manifold topology, and especially
in knot theory. A sutured manifold~$(M,\g)$ is a compact oriented $3$-manifold~$M$ with boundary, together with a decomposition of the boundary~$\partial M$
into a positive part~$R_+(\g)$ and a negative part~$R_-(\g)$ that meet along a ``thickened'' oriented $1$-manifold $\g \subset \partial M$
called the suture.
Honda, Kazez, and Mati\'c~\cite{tight} drew a parallel between convex surface theory and sutured manifold theory, making apparent
the usefulness of sutured manifolds in contact topology. The author~\cite{sutured, survey} defined an invariant called sutured Floer homology, in short $\SFH$, for balanced sutured manifolds. $\SFH$ can be viewed as a common generalization of the hat version of Heegaard Floer homology and link Floer homology, both defined by Ozsv\'ath and Szab\'o~\cite{OSz, OSz2}. The author~\cite{decomposition} showed that $\SFH$ behaves nicely under sutured manifold decompositions, which has several important consequences, such as the above-mentioned detection of the genus and fibredness
by knot Floer homology.

In the present paper, we define a notion of cobordism between sutured manifolds $(M_0,\g_0)$ and $(M_1,\g_1)$.
It consists of a triple $\W = (W,Z,[\xi])$,
where~$W$ is a $4$-manifold with boundary and corners. The horizontal part of~$\partial W$ is $-M_0 \cup M_1$.
The vertical part~$Z \subset \partial W$ is a cobordism from~$-\partial M_0$ to~$-\partial M_1$, and carries a
positive cooriented contact structure~$\xi$ such that~$\partial M_i$ is a convex surface
with dividing set~$\g_i$ for $i \in \{0,1\}$.
We say that two such contact structures on~$Z$ are equivalent if they are homotopic through such contact structures,
and denote the equivalence class of~$\xi$ by~$[\xi]$.
Throughout this paper, all contact structures are considered to be cooriented.
Balanced sutured manifolds, together with certain equivalence classes of cobordisms between them, form a category.
We extend $\SFH$ to a functor from this category to finite dimensional $\Z_2$-vector spaces, giving a type of TQFT.

There might seem to be an alternative definition for cobordisms between sutured manifolds. One could consider triples
$(W,Z,F)$, where~$W$ is a $4$-manifold with corners, and has horizontal boundary $-M_0 \cup M_1$. Furthermore, $Z$
is a cobordism between~$-\partial M_0$ and~$-\partial M_1$, and $F \subset Z$ is a cobordism between~$\g_0$ and~$\g_1$.
A Morse-theoretic approach to define cobordism maps for such objects would require that every non-singular level set is a balanced
sutured manifold. For this, the pair~$(Z,F)$ has to be built up from pairs of 3-dimensional and 2-dimensional handles that are cut
into two equal halves by the 2-dimensional handle, or equivalently, contact handles.
A contact handle decomposition of~$Z$ gives rise to a contact structure~$\xi$ up to equivalence, and we arrive at the previous definition.

The construction of the map~$F_{\W}$, assigned to a cobordism~$\W$, goes as follows. The sutured manifold~$(-M_0,-\g_0)$
is a sutured submanifold of~$(-N,-\g_1)$ in the sense of Honda, Kazez, and Mati\'c~\cite{TQFT}, where $N = M_0 \cup -Z$.
The contact structure~$-\xi$ on~$Z$ -- which is~$\xi$ with the opposite coorientation -- has dividing set~$-\g_0$
on~$\partial M_0$ and~$-\g_1$ on~$\partial M_1$. It induces a gluing map
\[
\Phi_{-\xi} \colon \SFH(M_0,\g_0) \to \SFH(N,\g_1),
\]
as described in~\cite{TQFT}.
Then one can view~$W$ as a cobordism $\W_1 = (W,Z_1,[\xi_1])$
from~$(N,\g_1)$ to~$(M_1,\g_1)$ such that $Z_1 = \partial M_1 \times I$, and~$\xi_1$ is an
$I$-invariant contact structure such that $\partial M_1 \times \{t\}$ is a convex surface with dividing
set~$\g_1 \times \{t\}$ for every~$t \in I$. We call such a cobordism special, and it can be described using 1-, 2-, and 3-handle
attachments along the interior of~$N$. Generalizing the hat version of the cobordism maps on Heegaard-Floer homology, one gets
a map~$F_{\W_1}$ induced by a special cobordisms~$\W_1$. Finally, we set $F_{\W} = F_{\W_1} \circ \Phi_{-\xi}$.
This map is functorial; i.e., $F_{\W \circ \W'} = F_{\W} \circ F_{\W'}$.
Note that, in Remark~\ref{rem:gluing}, we outline a definition of the contact gluing maps,
and more generally, the sutured manifold cobordism maps, purely in terms of
special cobordism maps and the~$\EH$ class in sutured Floer homology~\cite{contact}.

We showed in~\cite{naturality} that~$\HFh$ is an invariant of based 3-manifolds. Given a based 3-manifold~$(Y,p)$,
moving~$p$ around a loop induces an automorphism of~$\HFh$, giving rise to an action of~$\pi_1(Y,p)$
on~$\HFh(Y,p)$. There are examples when this action is non-trivial, but we conjectured that it always factors through~$H_1(Y)$.
For a more precise formulation of this conjecture, see page~4 of~\cite{survey}.
This has recently been settled by Zemke~\cite{Zemke}, building on methods of this paper.
Consequently, given a connected cobordism~$X$ between the closed connected 3-manifolds~$Y_0$ and~$Y_1$,
the construction of cobordism maps~$\widehat{F}_X$ have to take into consideration the choice of basepoints.
Given basepoints~$p_0 \in Y_0$ and~$p_1 \in Y_1$, we have to fix an embedded arc~$a \colon I \to X$ from~$p_0$
to~$p_1$. Then we define the cobordism map
\[
\widehat{F}_{X,a} \colon \HFh(Y_0,p_0) \to \HFh(Y_1,p_1)
\]
to be~$F_\W$ for~$\W = (W,Z,\xi)$, where~$W = X \setminus N(a)$, the vertical boundary $Z = \ol{\partial N(a) \setminus \partial X}$,
and~$\xi$ is obtained from the unique tight contact structure on~$\partial N \approx S^3$ by removing two standard contact balls.
Note that~$\W$ is a cobordism between the sutured manifolds $Y_0(p_0)$ and~$Y_1(p_1)$ for some framings of~$p_0$ and~$p_1$.
Hence~$F_\W$ maps~$\SFH(Y_0(p_0)) = \HFh(Y_0,p_0)$ to~$\SFH(Y_1(p_1)) = \HFh(Y_1,p_1)$.

Suppose that in the based 3-manifold~$(Y,p)$, moving~$p$ around the loop~$\eta(t)$ induces a non-trivial automorphism~$\eta_*$ of~$\HFh(Y,p)$.
Consider the product cobordism~$X = Y \times I$ from~$Y$ to itself, together with the arc~$a(t) = (\eta(t),t)$ for~$t \in I$.
Then the cobordism map
\[
\widehat{F}_{X,a} \colon \HFh(Y,p) \to \HFh(Y,p)
\]
agrees with~$\eta_*$ and hence is not the identity. On the other hand, for the arc $a(t) = (p,t)$ for~$t \in I$,
the map~$\widehat{F}_{X,a}$ is the identity.

As a special case of cobordism maps on~$\SFH$, we also get maps on link Floer homology, induced by decorated link cobordisms. More precisely, we consider
decorated links~$(Y,L,P)$, where~$P$ consists of a positive even number of points on each component of the link~$L$,
together with a decomposition of~$L$ into compact $1$-dimensional submanifolds~$R_+(P)$ and~$R_-(P)$ such that
\[
R_+(P) \cap R_-(P) = P.
\]
Sarkar~\cite{basepoint} showed that knot Floer homology is only an invariant of based knots, as moving the basepoint around
the knot induces a non-trivial automorphism of knot Floer homology for most knots.

A cobordism from the decorated link $(Y_0,L_0,P_0)$ to~$(Y_1,L_1,P_1)$ consists of a triple
$(X,F,\sigma)$, where~$X$ is an oriented cobordism from~$Y_0$ to~$Y_1$,
the surface $F \subset X$ is orientable with boundary~$\partial F = L_0 \cup L_1$,
and~$\sigma \subset F$ is a properly embedded $1$-manifold  such that the map
\[
\pi_0(\partial\sigma) \to \pi_0((L_0 \setminus P_0) \cup (L_1 \setminus P_1))
\]
is a bijection. Furthermore, $\sigma$ divides~$F$ into two compact subsurfaces that meet along~$\sigma$,
and we can orient each component~$R$ of~$F \setminus \sigma$ such that whenever~$\partial\ol{R}$
crosses a point of~$P_0$, it goes from~$R_+(P_0)$ to~$R_-(P_0)$,
and whenever it crosses a point of~$P_1$, it goes from~$R_-(P_1)$ to~$R_+(P_1)$.
Finally, for every closed component~$F_0$ of~$F$, we have~$\sigma \cap F_0 \neq \emptyset$.

According to Lutz~\cite{Lutz}, the decoration~$\sigma$ uniquely defines an~$S^1$-invariant contact structure~$\xi$ on
the total space~$Z = \overline{\partial N(F) \setminus (Y_0 \cup Y_1)}$ of the normal $S^1$-bundle of~$F$ in~$X$
up to equivalence, making $(X \setminus N(F),Z,[\xi])$ into a cobordism~$\W$ between the sutured manifolds
$(Y_i \setminus N(L_i),P_i \times S^1)$ for~$i  \in \{0,1\}$ complementary to the decorated links.
By~\cite[Proposition~9.2]{polytope},
\[
\SFH(Y_i \setminus N(L_i),P_i \times S^1) \cong \HFLh(Y,L) \otimes V^{d_i},
\]
where~$V \cong \Z_2^2$, and~$d_i$ depends on the distribution of the marked points on~$L_i$. Hence, the cobordism
map~$F_{\W}$ maps between certain link Floer homology groups.
For computations of some elementary link cobordism maps, and a relationship
of these with the reduced Khovanov TQFT, see our paper with Marengon~\cite{JM}.
Also see the work of Kronheimer and Mrowka~\cite{unknot}, where they define maps
induced by knot cobordisms on their singular instanton knot homology, which is then
used to prove that Khovanov homology detects the unknot.

Finally, we extend the notion of Weinstein cobordisms to cobordisms between contact manifolds with convex boundary.
If~$\W$ is a Weinstein cobordism from the contact manifold $(M_0,\g_0,\z_0)$ to $(M_1,\g_1,\z_1)$, then we can view~$\W$ as a
cobordism~$\ol{\W}$ from~$(-M_1,-\g_1)$ to~$(-M_0,-\g_0)$. We prove that
\[
F_{\ol{\W}}(\EH(M_1,\g_1,\z_1)) = \EH(M_0,\g_0,\z_0),
\]
where~$\EH$ is the contact element in sutured Floer homology introduced by Honda, Kazez, and Mati\'c~\cite{contact}.

\section*{Acknowledgements}

I would like to thank Martin Hyland, Istv\'an Juh\'asz, Peter Kronheimer, Robert Lipshitz, Ciprian Manolescu, Marco Marengon,
Patrick Massot, Tom Mrowka, and Peter Ozsv\'ath for helpful conversations, and Jacob Rasmussen for his valuable comments on an early version of this paper.
I would also like to thank the referee for the constructive suggestions.

\section{The cobordism category of sutured manifolds}

Sutured manifolds were introduced by Gabai~\cite{Gabai}. The following definition
is slightly less general, in that it excludes toroidal sutures.

\begin{defn}
A \emph{sutured manifold} $(M,\g)$ is a compact oriented
3-manifold $M$ with boundary, together with a set $\g \subset
\partial M$ of pairwise disjoint annuli. Furthermore, the interior of each component of
$\g$ contains a \emph{suture}; i.e., a homologically
nontrivial oriented simple closed curve. We denote the union of the
sutures by $s(\g)$. Finally, every component of
\[
R(\g)=\partial M \setminus
\Int(\g)
\]
is oriented. Define $R_+(\g)$ (or
$R_-(\g)$) to be those components of $\partial M \setminus
\Int(\g)$ whose normal vectors point out of (into) $M$.
The orientation on $R(\g)$ must be coherent with respect to
$s(\g)$; i.e., if $\delta$ is a component of $\partial
R(\g)$ and is given the boundary orientation, then $\delta$ must
represent the same homology class in $H_1(\g)$ as some suture.
\end{defn}

\begin{rem}
In this paper, we are not going to make a distinction between $\g$ and~$s(\g)$, as it is usually clear from
the context which one we mean. One can think of $\g$ as a thickened oriented $1$-manifold.
\end{rem}

We now review some fundamental notions and results about contact structures and set our orientation conventions;
for more details see the notes of Etnyre~\cite{Etnyre} and the book of Geiges~\cite{Geiges}.
Let~$M$ be an oriented $3$-manifold. A \emph{contact structure}~$\xi$ on~$M$ is a nowhere integrable $2$-plane field.
This is equivalent to the condition that each point of~$M$ has a neighborhood~$U$ together with a
$1$-form~$\a$ such that~$\xi|_U = \ker(\a)$ and~$\a \wedge d\a$ is nowhere zero. If~$\xi$
is \emph{coorientable}, then one can choose~$\a$ globally. We say that~$\xi$ is \emph{positive} if
the $3$-form~$\a \wedge d\a$ is coherent with the orientation of~$M$. This is independent of the choice of
contact form~$\a$. In this paper, all contact structures are positive and cooriented, unless otherwise stated.
Given such a contact structure~$\xi$, we write~$-\xi$ for the same 2-plane field with the opposite coorientation;
this is also a positive contact structure.

A vector field~$v$ on~$M$ is a \emph{contact vector field} if its flow preserves~$\xi$. In terms of the
contact form~$\a$, this means that $\mathcal{L}_v \a = f\a$ for some function~$f \colon M \to \R$.
If~$F$ is a properly embedded surface in~$M$, then~$F$ is \emph{convex} if there exists a contact vector
field~$v$ transverse to~$F$. Every surface~$F$ is $C^\infty$-close to a convex surface, and
every convex surface has a product neighborhood in which the contact structure is invariant in
the normal direction.

The contact structure~$\xi$ defines a singular foliation~$\xi F$ on~$F$ called
the \emph{characteristic foliation} of~$\xi$ on~$F$.
Given an orientation of~$F$, we can orient the leaves of~$\xi F$ as follows.
Pick a point~$p \in F$ where~$\xi F$ is non-singular, and let~$l_p$ be the line tangent to the leaf
of~$\xi F$ through~$p$. Then $l_p$ is oriented as~$\xi_p \cap T_pF$. More concretely,
a vector~$v \in l_p$ defines the positive orientation of~$l_p$ if when we choose
vectors~$v_\xi \in \xi_p$ and~$v_F \in T_pF$ such that $(v,v_\xi)$ orients~$\xi$
and~$(v,v_F)$ orients~$T_pF$, then the triple~$(v,v_\xi,v_F)$ orients~$T_pM$.
Given a singular point~$p \in F$ of the foliation~$\xi F$, we can associate to it a sign
depending on whether~$\xi_p$ is positively or negatively tangent to~$F$.

Given a convex surface~$F$ and a contact vector field~$v$ transverse to it, we can define the
\emph{dividing set}
\[
\Gamma = \{\, p \in F \colon v(p) \in \xi_p \,\}.
\]
This is always a $1$-manifold transverse to the characteristic foliation~$ \xi F$,
and given another contact vector field, the resulting dividing set will be isotopic.
We orient~$\Gamma$ so that it is positively transverse to~$\xi$.
If~$F$ is oriented, then~$\Gamma$ splits~$F$ into two subsurfaces~$F_+$ and~$F_-$
that meet along~$\Gamma$. The leaves of~$\xi F$ go from~$F_+$ to~$F_-$, and~$F_+$
contains the positive singularities of~$\xi F$, while~$F_-$ contains the negative ones.
Observe that no matter how we orient~$F$, the dividing set~$\Gamma$ is always oriented
as the boundary of~$F_+$. If~$v$ is positively transverse to~$F$, then~$F_+$ is the
set of points~$p$ of~$F$ where~$v_p$ is positively transverse or tangent to~$\xi_p$,
while~$F_-$ is where~$v_p$ is negatively transverse or tangent to~$\xi$.
A surprising result of Giroux states that the dividing set determines~$\xi$
uniquely up to isotopy in a neighborhood of~$F$.

\begin{defn}
Let $(M,\g)$ be a sutured manifold, and suppose that $\xi_0$ and $\xi_1$ are contact
structures on $M$ such that $\partial M$ is a convex surface with dividing set $\g$ with
respect to both $\xi_0$ and $\xi_1$. Then we say that $\xi_0$ and $\xi_1$ are \emph{equivalent} if there is a
one-parameter family $\{\,\xi_t \,\colon\, t \in I\,\}$ of contact structures such that
$\partial M$ is convex with dividing set $\g$ with respect to $\xi_t$ for every $t \in I$.
In this case, we write $\xi_0 \sim \xi_1$, and we denote by $[\xi]$ the equivalence class of
a contact structure~$\xi$.
\end{defn}

\begin{defn} \label{defn:cob}
Let $(M_0,\g_0)$ and $(M_1,\g_1)$ be sutured manifolds. A \emph{cobordism from $(M_0,\g_0)$ to $(M_1,\g_1)$} is a triple
$\W = (W,Z,[\xi])$, where
\begin{enumerate}
\item $W$ is a compact, oriented $4$-manifold with boundary,
\item $Z$ is a compact, codimension-0 submanifold with boundary of $\partial W$, and $\partial W \setminus \Int(Z) = -M_0 \sqcup M_1$,
\item $\xi$ is a positive contact structure on $Z$, such that $\partial Z$ is a convex surface with dividing
set~$\g_i$ on~$\partial M_i$ for $i \in \{0,1\}$.
\end{enumerate}
\end{defn}

\begin{rem} \label{rem:cob}
For orienting the boundary of a manifold, we always use the ``outward normal first'' convention.
We think of $W$ as a 4-manifold with corners along $\partial Z$.
Furthermore, we say that $-M_0 \sqcup M_1$ is the horizontal and $Z$ is the  vertical part of the boundary of~$W$.
Note that the orientation of the dividing set of~$\xi$ on~$\partial Z$
only depends on the coorientation of~$\xi$,
and is independent of how we orient the convex surface~$\partial Z$.
\end{rem}

\begin{lem} \label{lem:chi}
If the sutured manifolds $(M_0,\g_0)$ and $(M_1,\g_1)$ are cobordant, then
\[
\chi(R_+(\g_0)) - \chi(R_-(\g_0)) =  \chi(R_+(\g_1)) - \chi(R_-(\g_1)).
\]
Furthermore, the map $\pi_0(\g_i) \to \pi_0(\partial M_i)$ is surjective for $i \in \{0, 1\}$.
\end{lem}

\begin{proof}
Recall that $(Z,\xi)$ is a contact manifold with convex boundary $\partial M_0 \cup -\partial M_1$ and
dividing set $\g = \g_0 \cup \g_1$.
We denote by~$R_+(\g)$ and~$R_-(\g)$, respectively,
the positive and negative subsurfaces of~$Z$ induced by~$\xi$. Then
\[
\chi(R_+(\g)) - \chi(R_-(\g)) =  \langle\,e(\xi),[\partial Z] \,\rangle = 0.
\]
Moreover,
\[
\chi(R_+(\g)) = \chi(R_+(\g_0)) + \chi(R_-(\g_1)) \text{ and }
\chi(R_-(\g)) = \chi(R_-(\g_0)) + \chi(R_+(\g_1)),
\]
as on~$-\partial M_1$ the positive subsurface induced by~$\xi$ is~$R_-(\g_1)$ and the negative
subsurface is~$R_+(\g_1)$, while on~$\partial M_0$ the positive subsurface is~$R_+(\g_0)$
and the negative subsurface is~$R_-(\g_0)$.
The second claim follows from the fact that the dividing set
on a closed convex surface is never empty, see~\cite[p.230]{Geiges}.
\end{proof}

\begin{defn} \label{defn:cobeq}
The cobordisms $\W = (W,Z,[\xi])$ and $\W' = (W',Z',[\xi'])$ between the same sutured manifolds $(M_0,\g_0)$ and $(M_1,\g_1)$
are called \emph{equivalent} if there is an orientation preserving diffeomorphism
$d \colon W \to W'$ such that $d(Z) = Z'$ and~$d_*(\xi) \sim \xi'$; furthermore,
$d(x) = x$ for every $x \in M_0 \cup M_1$. Such a map $d$ is called an \emph{equivalence}.

If $\W = (W,Z,[\xi])$ is a cobordism from the sutured manifold $(M_0,\g_0)$ to $(M_1,\g_1)$ and $\W' = (W',Z',[\xi'])$ is a cobordism from
$(M_0',\g_0')$ to $(M_1',\g_1')$, then $\W$ and~$\W'$ are called \emph{diffeomorphic}
if there is an orientation preserving diffeomorphism
$d \colon W \to W'$ such that $d(Z) = Z'$, $d(\g_0) = \g_0'$, $d(\g_1) = \g_1'$, and $d_*(\xi) \sim \xi'$.
Such a map $d$ is called a \emph{diffeomorphism} from~$\W$ to~$\W'$.
\end{defn}

\begin{defn} \label{defn:triv}
Let $(M,\g)$ be a sutured manifold such that there is at least one suture on each component of~$\partial M$.
The \emph{trivial cobordism from $(M,\g)$ to $(M,\g)$} is the triple $\W = (W,Z,[\xi])$, where
\begin{enumerate}
\item $W = M \times I$,
\item $Z = \partial M \times I$,
\item $\xi$ is an $I$-invariant contact structure on $Z$ such that $\partial M \times \{t\}$ is a convex surface with dividing
set $\gamma \times \{t\}$ for every $t \in I$. Note that such a $\xi$ is well-defined up to equivalence.
\end{enumerate}
\end{defn}

\begin{rem}
To be completely precise, just as in Milnor~\cite{Milnor}, one should define a cobordism from $(N_0,\nu_0)$ to $(N_1,\nu_1)$ as a $5$-tuple
\[
\W = ((W,Z,[\xi]),(M_0,\g_0),(M_1,\g_1),h_0,h_1),
\]
where $(W,Z,[\xi])$ is a cobordism from $(M_0,\g_0)$ to $(M_1,\g_1)$ in the sense of Definition~\ref{defn:cob},
and for $i \in \{0,1\}$, the map $h_i \colon M_i \to N_i$ is an orientation preserving diffeomorphism such that $h_i(\g_i) = \nu_i$.
If we have two such cobordisms~$\W$ and~$\W'$ from $(N_0,\nu_0)$ to $(N_1,\nu_1)$,
then an equivalence between them is a diffeomorphism~$g$ from~$\W$ to~$\W'$ such that $g|_{M_i} = (h_i')^{-1} \circ h_i$ for $i \in \{0,1\}$.

If $(N_0,\nu_0)$ and $(N_1,\nu_1)$ are disjoint, then we can safely restrict ourselves to cobordisms between them where
$M_i = N_i$ and $h_i = \text{Id}_{N_i}$ for $i \in \{0,1\}$, in which case the above notion of equivalence
coincides with the one in Definition~\ref{defn:cobeq}. However, to define the
identity morphism from~$(N,\nu)$ to itself, one does need the above more precise approach to cobordisms.
To keep the notation simple, we will use our previous less rigorous terminology, which should not cause much confusion.
\end{rem}

\begin{defn} \label{defn:comp}
Suppose that $\W_0 = (W_0,Z_0,[\xi_0])$ is a cobordism from $(M_0,\g_0)$ to $(M_1,\g_1)$ and $\W_1 = (W_1,Z_1,[\xi_1])$
is a cobordism from $(M_1,\g_1)$ to $(M_2,\g_2)$.
Since~$\partial M_1$ is a convex surface with dividing set~$\g_1$ in both $(Z_0,\xi_0)$ and $(Z_1,\xi_1)$,
we can glue the contact structures~$\xi_0$ and~$\xi_1$ together along~$\partial M_1$
to obtain a cooriented contact structure~$\xi_1 \cup \xi_2$ on~$Z_0 \cup_{\partial M_1} Z_1$, well-defined
up to equivalence. Then the \emph{composition}~$\W_1 \circ \W_0$ is the cobordism from $(M_0,\g_0)$ to $(M_2,\g_2)$ given by the triple
\[
(W_0 \cup_{M_1} W_1, Z_0 \cup_{\partial M_1} Z_1, [\xi_0 \cup \xi_1]).
\]
\end{defn}

\begin{defn}
The \emph{cobordism category of sutured manifolds}, $\Sut$, is given as follows. Its objects are sutured manifolds $(M,\g)$ that have
at least one suture on each boundary component. The set of morphisms from $(M_0,\g_0)$ to $(M_1,\g_1)$ is the set of equivalence classes of cobordisms from $(M_0,\g_0)$ to $(M_1,\g_1)$. Composition is given by Definition~\ref{defn:comp}. The identity morphism from $(M,\g)$ to itself is the equivalence class
of the trivial cobordism introduced in Definition~\ref{defn:triv}.

By Lemma~\ref{lem:chi}, for a given integer $k \in \Z$, those sutured manifolds that satisfy
\[
\chi(R_+(\g)) - \chi(R_-(\g)) = k
\]
form a full subcategory of $\Sut$
called $\Sut_k$, and the sum category of~$\{\, \Sut_k \,\colon\, k \in \Z \,\}$ is exactly $\Sut$.
\end{defn}

Note that, in order to have a unique identity morphism for each sutured manifold and to be able to define the composition of cobordisms,
it was necessary to work with equivalence classes of contact structures. It is not possible to set up a cobordism
category using contact structures without factoring out by this equivalence relation in addition to
taking equivalence classes of cobordisms. Indeed, an
equivalence from $(W,Z,\xi)$ to $(W',Z',\xi')$ would have to map the characteristic
foliation of~$\xi$ on~$\partial Z$ to that of~$\xi'$ on $\partial Z' = \partial Z$. Hence, the equivalence classes
of trivial cobordisms for a given sutured manifold $(M,\g)$ would decompose along the set of
possible characteristic foliations on~$\partial Z$.
Furthermore, to be able to compose the cobordism $(W_0,Z_0,\xi_0)$ from~$(M_0,\g_0)$ to~$(M_1,\g_1)$
with the cobordism $(W_1,Z_1,\xi_1)$ from~$(M_1,\g_1)$ to~$(M_2,\g_2)$,
the characteristic foliation of~$\xi_0$ on~$\partial M_1$ has to agree with the characteristic
foliation of~$\xi_1$ on~$\partial M_1$. If we are working with equivalence classes of contact
structures, we can always homotope~$\xi_0$ and~$\xi_1$ until the two characteristic foliations
line up and we can perform the gluing.

The following notion was introduced by the author~\cite{sutured}.

\begin{defn}
A sutured manifold $(M,\g)$ is \emph{balanced} if
\begin{enumerate}
\item $\chi(R_+(\g)) = \chi(R_-(\g))$,
\item the map $\pi_0(\g) \to \pi_0(\partial M)$ is surjective, and
\item $M$ has no closed components.
\end{enumerate}
\end{defn}

\begin{rem} \label{rem:upside}
The objects of $\Sut_0$ are precisely those sutured manifolds that can be written as finite disjoint unions of
balanced sutured manifolds and closed oriented 3-manifolds.

It is also worth noting that if $\W = (W,Z,[\xi])$ is a cobordism from $(M_0,\g_0)$ to $(M_1,\g_1)$, then $(-W,-Z,[-\xi])$ is not a cobordism
from $(-M_1,-\g_1)$ to $(-M_0,-\g_0)$ since $-\xi$ is a negative contact structure on $-Z$. But we can view
$(W,Z,[-\xi])$ as a cobordism~$\ol{\W}$ from~$(-M_1,-\g_1)$ to~$(-M_0,-\g_0)$ by writing
\[
\partial W = -(-M_1) \cup Z \cup -M_0.
\]
Loosely speaking, this is turning the cobordism $\W$ upside down.
\end{rem}

\begin{defn} \label{defn:bal}
We say that a cobordism from $(M_0,\g_0)$ to $(M_1,\g_1)$ is \emph{balanced} if
both $(M_0,\g_0)$ and $(M_1,\g_1)$ are balanced sutured manifolds.
The balanced sutured manifolds and equivalence classes of balanced cobordisms
form a full subcategory of $\Sut_0$ that we denote by $\BSut$.
\end{defn}

Sutured Floer homology was introduced by the author~\cite{sutured}. Over $\Z_2$, it assigns a finite-dimensional $\Z_2$ vector space $\SFH(M,\g)$ to every balanced sutured manifold~$(M,\g)$. The main goal of the present paper is to promote~$\SFH$ to a functor from~$\BSut$ to~$\Vect$. That is, for every balanced cobordism $\W$ from $(M_0,\g_0)$ to~$(M_1,\g_1)$, we are going to define a linear map
\[
F_{\W} \colon \SFH(M_0,\g_0) \to \SFH(M_1,\g_1)
\]
such that $F_{\W_1 \circ \W_0} = F_{\W_1} \circ F_{\W_0}$, and if $\W$ is a trivial cobordism, then $F_{\W} = \text{Id}$. In Theorem~\ref{thm:TQFT}, we will show that this is an instance of a $(3+1)$-dimensional TQFT, as defined by Atiyah~\cite{Atiyah, Blanchet}.

\section{Relative $\spinc$ structures}

First, we briefly review the definition of relative $\spinc$ structures on sutured manifolds as defined by the author~\cite{sutured}.
The definition given here requires a slightly less restrictive but equivalent boundary condition in order to be able to talk about $\spinc$
structures represented by contact structures with convex boundary.

\begin{defn}
Given a sutured manifold~$(M,\g)$, we say that a vector field~$v$ defined on a subset of~$M$
containing~$\partial M$ is \emph{admissible} if it is nowhere vanishing, it points into~$M$ along~$R_-(\g)$,
it points out of~$M$ along~$R_+(\g)$, and~$v|_{\g}$ is tangent to~$\partial M$ and either
points into~$R_+(\g)$ or is positively tangent to~$\g$
(as before, we think of~$\partial M$ as a smooth surface, and of~$\g$ as a $1$-manifold).

Let~$v$ and~$w$ be admissible vector fields on~$M$.
We say that~$v$ and~$w$ are homologous, and we write $v \sim w$,
if there is a collection of balls $B \subset M$, one in each component of~$M$, such that $v$ and~$w$ are homotopic on $M \setminus B$
through admissible vector fields.
Then $\spinc(M,\g)$ is the set of homology classes of admissible vector fields on~$M$.
\end{defn}

According to~\cite[Proposition 3.5]{polytope},
$\spinc(M,\g) \neq \emptyset$ if and only if for every component~$M_0$ of~$M$, we have
\[
\chi(M_0 \cap R_+(\g)) = \chi(M_0 \cap R_-(\g)).
\]
The space of vector fields arising as~$v|_{\partial M}$ for~$v$ admissible is convex, hence contractible.
Suppose that~$(M,\g)$ is a sutured submanifold of~$(N,\nu)$; i.e., $M \subset \Int(N)$.
If~$v$ is an admissible vector field on~$(M,\g)$ and~$w$ is an admissible vector field
on~$(N \setminus \Int(M), \g \cup \nu)$, then there is a homotopically unique deformation of~$v$ through
admissible vector fields such that $v|_{\partial M} = w|_{\partial M}$. This gives a unique
way of gluing the $\spinc$ structures represented by~$v$ and~$w$ to obtain a $\spinc$ structure on~$(N,\nu)$.

\begin{defn}
Let $(M,\g)$ be a sutured manifold. We say that an oriented 2-plane field~$\xi$ defined on a subset of~$M$
containing~$\partial M$ is \emph{admissible} if
there exists a Riemannian metric~$g$ on~$M$ such that~$\xi^{\perp_g}$ is an admissible vector
field. If~$v$ is defined on the whole manifold~$M$, we write
\[
\s_\xi = [\xi^{\perp_g}] \in \spinc(M,\g).
\]
This is independent of the choice of~$g$ since the space of metrics~$g$ for which
$\xi^{\perp_g}$ is an admissible vector field is convex.
\end{defn}

\begin{lem} \label{lem:contadm}
If $\xi$ is a contact structure on~$M$ such that~$\partial M$ is a convex surface with
dividing set~$\g$, then~$\xi$ is admissible.
\end{lem}

\begin{proof}
Let~$w$ be a contact vector field on~$M$ transverse to~$S$ such that
\[
\{\, p \in S \,\colon\, w(p) \in \xi_p \,\} = \g.
\]
Then we choose a Riemannian metric~$g$ on~$M$ such that~$w(p) \perp T_pS$ for every~$p \in S$,
and such that~$\xi_p \perp T_p\g$ for every~$p \in \g$ (the latter is possible since~$\xi$
is transverse to~$\g$). Then the vector field~$\xi^{\perp_g}$ is admissible.
So we can talk about the induced relative $\spinc$-structure~$\s_\xi$.
\end{proof}

Next, we recall a standard result from complex geometry.

\begin{lem} \label{lem:cx}
Let $V$ be a 4-dimensional real vector space, together with an endomorphism $J$ such that $J^2 = -I$. Then every 3-dimensional
subspace $U < V$ contains a unique $J$-invariant plane.
\end{lem}

\begin{proof}
Think of $(V,J)$ as a complex vector space. Since two different complex lines span $V$ over $\R$, they cannot both lie in $U$.
Thus $U \cap J(U)$ is the unique $J$-invariant 2-plane in $U$.
\end{proof}

So if $J$ is an almost complex structure on a $4$-manifold $W$ and $H$ is a $3$-dimensional
submanifold, then there is a $2$-plane field induced on $H$ called the field of \emph{complex tangencies} along $H$.
The following definition generalizes the one given by Ozsv\'ath and Szab\'o~\cite[Section~8.1.3]{OSz}, also
see~\cite[Lemma~2.1]{KM}.

\begin{defn} \label{defn:spinc}
Suppose that $\W = (W,Z,[\xi])$ is a cobordism from the sutured manifold $(M_0,\g_0)$ to~$(M_1,\g_1)$.
We say that an almost complex structure~$J$ defined on a subset of~$W$ containing~$\partial Z$
is \emph{admissible} if the field of complex tangencies in~$TM_i|_{\partial M_i}$ is admissible in~$(M_i,\g_i)$ for $i \in \{0,1\}$,
and the field of complex tangencies in $TZ|_{\partial Z}$ is admissible in~$(Z,\g_0 \cup \g_1)$.

A \emph{relative $\spinc$ structure} on~$\W$ is a homology class of pairs $(J,P)$, where
\begin{itemize}
\item $P \subset \text{Int}(W)$ is a finite collection of points,
\item $J$ is an admissible almost complex structure defined over $W \setminus P$, and
\item if $\xi_J$ is the field of complex tangencies along $Z$, then $\s_\xi = \s_{\xi_J}$.
\end{itemize}
We say that $(J,P)$ and $(J',P')$ are \emph{homologous} if there exists a compact $1$-manifold $C \subset W \setminus \partial Z$ such that
$P$, $P' \subset C$;
furthermore, $J|_{W \setminus C}$ and $J'|_{W \setminus C}$ are isotopic through admissible almost complex structures.
Denote by $\spinc(\W)$ the set of relative $\spinc$ structures over~$\W$.
\end{defn}

Given any $\spinc$ structure $\s \in \spinc(\W)$ and $i \in \{1, 2\}$, we can define
\[
\s_i = \s|_{(M_i,\g_i)} \in \spinc(M_i,\g_i)
\]
as the $\spinc$ structure of the field of complex tangencies of~$J$ along $M_i$ for an arbitrary representative
$(J,P)$ of $\s$. By definition, $\s|_{(Z,\g_0 \cup \g_1)} = \s_\xi$.

Let $i$ be the embedding of the pair~$(Z,\partial Z)$ into $(W,\partial Z)$,
and consider the induced restriction map
\[
i^* \colon H^2(W, \partial Z) \to H^2(Z, \partial Z).
\]
Then $\spinc(\W)$ is an affine space over $\ker(i^*)$. Indeed,
homology classes of admissible almost complex structures on~$W$
form an affine space over $H^2(W, \partial Z)$ as the space of admissible
almost complex structures on~$TW|_{\partial Z}$ is contractible.
Two such almost complex structures restrict to the same element of~$\spinc(Z, \g_0 \cup \g_1)$
if and only if their difference lies in $\ker(i^*)$.
We now define a related space of relative $\spinc$ structures.

\begin{defn} \label{def:spin-other}
Suppose that we are given an admissible almost complex structure~$J'$ on~$TW|_Z$
such that $\s_\xi = \s_{\xi_{J'}}$,
where~$\xi_{J'}$ is the filed of complex tangencies of~$J'$ along~$Z$.
Then $\spinc(\W, J')$ is the set of homology classes of pairs~$(J,P)$ such that~$J$
is an almost complex structure on~$W \setminus P$ and~$J|_Z = J'$.
\end{defn}

By obstruction theory, $\spinc(\W,J')$ is an affine space over $H^2(W,Z)$. Note that we mainly
focus on $\spinc(\W)$ instead of $\spinc(\W,J')$ in this paper because
in the definition of $\W$ we only fix the equivalence class $[\xi]$ of a contact
structure, so there is no (homotopically) unique almost complex structure along $Z$ that we could
use as a boundary condition. Had we fixed a concrete contact structure $\xi$ along $Z$,
equivalent balanced cobordisms would induce the same characteristic
foliations on $\partial M_i$, making it impossible to compose cobordisms,
or to define the identity morphism from $(M,\g)$ to itself.

\begin{lem} \label{lem:spec}
Suppose that for the balanced cobordism $\W = (W,Z,[\xi])$ we have
$H^k(Z, \partial M_1) = 0$ for $k \in \{1,2\}$.
Then, given an almost complex structure~$J'$ on~$Z$ as above, the
restriction map
\[
q \colon \spinc(\W,J') \to \spinc(\W)
\]
is a bijection.
\end{lem}

\begin{proof}
Consider the sequence of embeddings
\[
(W, \partial M_1) \stackrel{e}{\hookrightarrow} (W,\partial Z) \stackrel{f}{\hookrightarrow} (W,Z).
\]
Then, on second cohomology, $(f \circ e)^* = e^* \circ f^*$.
The restriction map $q$ is an affine map modeled on
\[
f^* \colon H^2(W,Z) \to \ker(i^*).
\]
From the long exact sequence of the triple $(W,Z,\partial Z)$,
we have $\text{im}(f^*) = \ker(i^*)$.
Furthermore, by the long exact sequence of the triple $(W,Z,\partial M_1)$
and our assumptions on $H^k(Z, \partial M_1)$, we see that
the map
\[
(f \circ e)^* \colon H^2(W,Z) \to H^2(W,\partial M_1)
\]
is an isomorphism. Hence $f^*$ is injective. By the above, $f^*$ is also surjective onto $\ker(i^*)$.
This shows that~$f^*$ is a bijection, and so is~$q$.
\end{proof}

\begin{rem}
As we shall see in Section~\ref{sec:special}, the space $\spinc(\W,J)$ naturally appears
when parameterizing homotopy classes of pseudo-holomorphic polygons.
In Definition~\ref{def:special}, we will introduce \emph{special} cobordisms, these
satisfy $Z = \partial M_0 \times I$. To define maps induced by special cobordisms, we
will count pseudo-holomorphic triangles.
Lemma~\ref{lem:spec} implies that, for special
cobordisms, the spaces $\spinc(\W)$ and $\spinc(\W,J')$ are isomorphic.
\end{rem}

\section{Link cobordisms} \label{sec:links}

\begin{defn} \label{defn:link}
For $i \in \{0,1\}$, let $Y_i$ be a connected, oriented 3-manifold, and let~ $L_i$ be a non-empty link in~$Y_i$.
Then a \emph{link cobordism} from $(Y_0,L_0)$ to $(Y_1,L_1)$ is a pair~$(X,F)$, where
\begin{enumerate}
\item $X$ is a connected, oriented cobordism from $Y_0$ to $Y_1$,
\item $F$ is a properly embedded, compact, orientable surface in $X$,
\item $\partial F = L_0 \cup L_1$.
\end{enumerate}
\end{defn}

We would like to associate to a link cobordism~$(X,F)$ a balanced cobordism $(W,Z,[\xi])$. However, to define the contact structure $\xi$,
we need more information, namely a set of dividing curves on $F$. For this, let us recall the notion of
a surface with divides from Honda et al.~\cite[Definition 4.1]{tight}, with the difference that we drop the orientation
of the surface.

\begin{defn}
A \emph{surface with divides} $(S,\sigma)$
is a compact orientable surface~$S$, possibly with boundary, together with a properly embedded $1$-manifold~$\sigma$
that divides~$S$ into two compact subsurfaces that meet along~$\sigma$.
\end{defn}

Link Floer homology of a link~$L$ is isomorphic to the~$\SFH$ of the sutured manifold complementary to~$L$.
Together with Dylan Thurston~\cite{naturality}, we constructed link Floer homology in a functorial way by first
defining sutured Floer homology functorially, and then applying a real blowup construction to~$L$ to
obtain a unique link complement, without having to make a choice of tubular neighborhood. We now review this
blowup procedure.

\begin{defn}
Suppose that $M$ is a smooth manifold, and let $L \subset M$ be a properly embedded submanifold.
For every $p \in L$, let $N_pL = T_pM/T_pL$ be the fibre of the normal bundle of $L$ over $p$,
and let $UN_pL = (N_pL \setminus \{0\})/\R_+$ be the fibre of the unit normal bundle of $L$ over $p$.
Then the \emph{(spherical) blowup} of $M$ along $L$, denoted by $\Bl_L(M)$, is a manifold with
boundary obtained from $M$ by replacing each point
$p \in L$ by $UN_pL$. There is a natural projection $\Bl_L(M) \to M$. For further details, see
Arone and Kankaanrinta~\cite{AK}.
\end{defn}

\begin{defn} \label{defn:declink}
A \emph{decorated link} is a triple $(Y,L,P)$, where $L$ is a non-empty link in the connected oriented 3-manifold $Y$, and $P \subset L$ is a finite set of points. We require that for every component~$L_0$ of~$L$, the number $|L_0 \cap P|$ is positive and even.
Furthermore, we are given a decomposition of~$L$ into compact $1$-manifolds~$R_+(P)$ and~$R_-(P)$ such that $R_+(P) \cap R_-(P) = P$.

We can canonically assign a balanced sutured manifold $\W(Y,L,P) = (M,\g)$ to every decorated link $(Y,L,P)$, as follows.
Let $M = \Bl_L(Y)$ and $\g = \bigcup_{p \in P} UN_pL$. Furthermore,
\[
R_\pm(\g) := \bigcup_{x \in R_\pm(P)} UN_xL,
\]
oriented as $\pm \partial M$, and we orient~$\g$ as~$\partial R_+(\g)$.
\end{defn}

\begin{defn} \label{def:linkcob}
We say that the triple $\X = (X,F,\sigma)$ is a \emph{decorated link cobordism} from $(Y_0,L_0,P_0)$ to $(Y_1,L_1,P_1)$ if
\begin{enumerate}
\item $(X,F)$ is a link cobordism from $(Y_0,L_0)$ to $(Y_1,L_1)$,
\item $(F,\sigma)$ is a surface with divides such that the map
\[
\pi_0(\partial\sigma) \to \pi_0((L_0 \setminus P_0) \cup (L_1 \setminus P_1))
\]
is a bijection,
\item \label{it:crossing} we can orient each component~$R$ of~$F \setminus \sigma$ such that whenever~$\partial\ol{R}$
crosses a point of~$P_0$, it goes from~$R_+(P_0)$ to~$R_-(P_0)$,
and whenever it crosses a point of~$P_1$, it goes from~$R_-(P_1)$ to~$R_+(P_1)$,
\item if $F_0$ is a closed component of $F$, then $\sigma \cap F_0 \neq \emptyset$.
\end{enumerate}

Two decorated link cobordisms $\X = (X,F,\sigma)$ and $\X' = (X',F',\sigma')$ between
the same decorated links $(Y_0,L_0,P_0)$ and $(Y_1,L_1,P_1)$ are said to be \emph{equivalent}
if there is an orientation preserving diffeomorphism $d \colon X \to X'$ such that $d(F) = F'$ and~$d(\sigma) = \sigma'$;
moreover, $d(y) = y$ for every $y \in Y_0 \cup Y_1$.

Suppose that $\X = (X,F,\sigma)$ is a cobordism from $(Y_0,L_0,P_0)$ to $(Y_1,L_1,P_1)$ and let $\X' = (X',F',\sigma')$ be a
cobordism from $(Y_0',L_0',P_0')$ to $(Y_1',L_1',P_1')$.
We say that~$\X$ and~$\X'$ are \emph{diffeomorphic} if there exists an orientation preserving
diffeomorphism~$d$ from~$X$ to~$X'$ such that $d(F) = F'$ and $d(\sigma) = \sigma'$,
and $d(R_\pm(P_i)) = R_\pm(P_i')$ for~$i \in \{0,1\}$.

Decorated links and equivalence classes of decorated link cobordisms form a category $\link$
with the obvious composition and identity morphisms.
As each link component has at least two marked points, when composing two
decorated link cobordisms, we do not create undecorated closed components of the surface.
\end{defn}

Note that in the above definition, neither the links~$L_0$ and~$L_1$, nor the surface~$F$ are required to be oriented.

\begin{prop}
Condition~\eqref{it:crossing} of Definition~\ref{def:linkcob} implies that
every non-closed component~$s$ of~$\sigma$ connects either~$R_+(P_0)$ and~$R_+(P_1)$, or $R_-(P_0)$ and~$R_-(P_1)$,
or~$R_+(P_i)$ and~$R_-(P_i)$ for~$i \in \{0,1\}$.
\end{prop}

\begin{proof}
Let~$R$ be the closure of a component of~$F \setminus \sigma$ such that~$s \subset \partial R$.
Then~$\partial R$ is a collection of polygonal
curves with edges alternatingly in~$\sigma$ and~$\partial F = L_0 \cup L_1$ along each component. Each edge in~$\partial F$
contains exactly one point of~$P_0 \cup P_1$. If we can orient~$R$ such that~\eqref{it:crossing} is satisfied,
then if~$s$ has both endpoints in~$L_0$ then it starts in~$R_-(P_0)$ and ends in~$R_+(P_0)$,
and if it has both endpoints in~$L_1$, then it starts in~$R_+(P_1)$ and ends in~$R_-(P_1)$.
As we go from~$+$ to~$-$ along~$L_0$ and from~$-$ to~$+$ along~$L_1$,
if~$s$ goes from~$L_0$ to~$L_1$, it start in~$R_-(P_0)$ and ends in~$R_-(P_1)$,
and if it goes from~$L_1$ to~$L_0$, then it starts in~$R_+(P_1)$ and ends in~$R_+(P_0)$.
 \end{proof}

\begin{rem}
Note that the converse of the above statement is not true. For this end, take the product cobordism
from the two-component unlink with two marked points on each component to itself, where the dividing set~$\sigma$
consists of four vertical lines, then connect the two cylinders with a tube. If chosen appropriately,
it is not possible to orient the component of~$F \setminus \sigma$ containing the tube correctly as for each orientation
exactly one of the two boundary components will go from~$R_+(P_0)$ to~$R_-(P_0)$.
\end{rem}

Let $\pi \colon M \to F$ be a principal circle bundle over the compact oriented surface~$F$, where the
orientation of~$M$ is determined by the orientation of the base and the fibre.
If $\xi$ is an $S^1$-invariant
contact structure on $M$, then it defines a dividing set $\sigma$ on $F$ as follows. A point $x \in F$ lies in $\sigma$
if and only if $\xi$ is tangent to the fibre $M_x = \pi^{-1}(x)$. Let $R_+(\sigma)$ consist of those $x \in F$ for which $M_x$ is positively
transverse or tangent to~$\xi$. Similarly, $R_-(\sigma)$ is the set of those~$x \in F$ for which~$M_x$ is negatively
transverse or tangent to~$\xi$. Then~$R_+(\sigma)$ and~$R_-(\sigma)$ are compact subsurfaces of~$F$ that meet along~$\sigma$.
The $S^1$ action defines a contact vector
field $v$ on~$M$ tangent to the fibres. The image of any (local) section of $\pi$ is hence a convex surface with
dividing set projecting onto $\sigma$.

The converse of the above is also true, in the following sense. %Suppose that $(X,F,\sigma)$ is a decorated link cobordism.
Let~$\pi \colon M \to F$ be as above, and let~$\sigma$ be a dividing set on~$F$ that intersects each component of~$F$ non-trivially
and divides~$F$ into the subsurfaces~$R_+(\sigma)$ and~$R_-(\sigma)$.
According to Lutz~\cite{Lutz} and Honda~\cite[Theorem~2.11 and Section~4]{Ko}, up to isotopy,
there is a unique $S^1$-invariant contact structure~$\xi_{\sigma}$ on $M$ such that
the dividing set associated to~$\xi_\sigma$ is exactly~$\sigma$,
the coorientation of~$\xi_\sigma$ induces the splitting~$R_\pm(\sigma)$,
and the boundary~$\partial M$ is a convex.
Furthermore, if~$F$ has no $S^2$ or $T^2$ components, this correspondence is bijective between the isotopy classes of those
dividing sets~$\sigma$ that have no homotopically trivial components, and the isotopy classes of
universally tight contact structures on~$M$.

The dividing set of~$\xi_\sigma$ on~$\partial M$, which we denote by~$\g_\sigma$, is $S^1$-invariant.
In other words, each component of~$\g_\sigma$ projects to a single point in~$\partial M$ under~$\pi$.
By~\cite[Lemma~6.6]{Etnyre}, between any two adjacent points of~$\partial\sigma$, there is exactly
one point of~$P = \pi(\g_\sigma)$ and vice versa; i.e., the map
\[
\pi_0(\partial \sigma) \to \pi_0(\partial F \setminus P)
\]
is a bijection. The coorientation of~$\xi_\sigma$ determines a splitting of~$\partial M$ into
compact subsurfaces~$R_+(\g_\sigma)$ and~$R_-(\g_\sigma)$ that meet along~$\g_\sigma$. Let~$R_\pm(P) = \pi(R_\pm(\g_\sigma))$.

\begin{lem} \label{lem:crossing}
Whenever~$\partial R_+(\sigma)$ crosses a point of~$P$, it goes from~$R_+(P)$ to $R_-(P)$.
\end{lem}

\begin{proof}
Let $p \in P \cap R_+(\sigma)$. Since $p \in P$, the fiber~$M_p$ is a component of~$\g_\sigma$.
As~$p \in R_+(\sigma)$, the orientation of~$M_p$ is positively transverse to~$\xi_\sigma$,
and hence on~$M_p$ the fiber orientation coincides with the orientation of the dividing set~$\g_\sigma$.
On the other hand, $\g_\sigma$ is oriented as the boundary of~$R_+(\g_\sigma)$. Given
an arbitrary point~$x \in M_p$, a vector~$v_+ \in T_x\partial M$ pointing out of~$R_+(\g_\sigma)$,
and a vector~$v_p \in T_x M_p$ orienting~$M_p$, the pair $(v_+,v_p)$ orients~$\partial M$
as~$\partial R_+(\g_\sigma)$ is oriented via the ``outward normal first'' rule.
If~$w \in T_x M$ is an outward normal of~$M$, then the basis~$(w,v_+,v_p)$ orients~$M$.
But~$M$ is oriented via taking the orientation of the base~$F$, followed by the orientation
of the fibre~$M_p$. Since~$v_p$ orients~$M_p$, it follows that $(d\pi(w), d\pi(v_+))$
is a positive basis of~$T_pF$. Since~$d\pi(w)$ is an outward normal of~$F$, we get that~$d\pi(v_+)$
orients~$\partial F$. This proves that if~$p$ lies in~$R_+(\sigma)$, then~$\partial F$
is oriented from~$R_+(P)$ to~$R_-(P)$.
\end{proof}

\begin{defn} \label{defn:W}
Let $(X,F,\sigma)$ be a decorated link cobordism from $(Y_0,L_0,P_0)$ to~$(Y_1,L_1,P_1)$.
Then we define the cobordism $\W = \W(X,F,\sigma)$ as follows.
Choose an arbitrary splitting of~$F$ into~$R_+(\sigma)$ and~$R_-(\sigma)$, and orient~$F$
such that~$\partial R_+(\sigma)$ crosses~$P_0$ from~$R_+(P_0)$ to~$R_-(P_0)$
and~$P_1$ from~$R_-(P_1)$ to~$R_+(P_1)$.
Then~$\W$ is defined to be the triple $(W,Z,[\xi])$, where $W = \Bl_F(X)$ and $Z = UNF$, oriented
as a submanifold of~$\partial W$, finally $\xi = \xi_{\sigma}$.

Note that $\W$ is a cobordism from $\W(Y_0,L_0,P_0) = (M_0,\g_0)$ to $\W(Y_1,L_1,P_1) = (M_1,\g_1)$.
Indeed, let $\pi \colon Z \to F$ be the natural projection, then we can assume that~$\pi(\g_\sigma) = P_0 \cup P_1$.
By Lemma~\ref{lem:crossing} and our assumptions above,
\[
\pi(R_\pm(\g_\sigma)) \cap L_0 = R_\pm(P_0) \text{ and } \pi(R_\pm(\g_\sigma)) \cap L_1 = R_\mp(P_1).
\]
This implies that $R_\pm(\g_\sigma) \cap \partial M_0 = R_\pm(\g_0)$ and $R_\pm(\g_\sigma) \cap \partial M_1 = R_\mp(\g_1)$.
Since $\partial Z = \partial M_0 \cup (-\partial M_1)$, we obtain that~$\g_\sigma \cap \partial M_0 = \g_0$
and $\g_\sigma \cap \partial M_1 = \g_1$.
\end{defn}

The above definition is independent of the choice of splitting of~$F$ into~$R_+(\sigma)$ and~$R_-(\sigma)$.
Indeed, if we swap the splitting on a component~$F_0$ of~$F$, then on~$F_0$ the orientation of~$F$
is swapped as well. Hence, the orientation of each fiber of~$UNF_0$ is also reversed.
The same contact structure~$\xi_\sigma$ with the same coorientation is still $S^1$-invariant,
and induces the reversed~$R_+$ and~$R_-$ on~$F_0$ as the fiber orientation is reversed.

The following proposition is straightforward to verify using the definitions. Recall that for an object $(Y,L,P)$ of $\link$,
the sutured manifold $\W(Y,L,P)$ was introduced in Definition~\ref{defn:declink},
and for a morphism~$\X$ in~$\link$, the cobordism~$\W(\X)$ was defined in Definition~\ref{defn:W}.

\begin{prop} \label{prop:Wfunct}
The map~$\W$ is a \emph{functor} from $\link$ to $\BSut$. Furthermore, if the decorated
link cobordisms~$\X$ and~$\X'$ are equivalent or diffeomorhpic,
then~$\W(\X)$ and~$\W(\X')$ are also equivalent or diffeomorhpic, respectively.
\end{prop}

Hence, the composition $\SFH \circ \W$ gives a functor from $\link$ to $\Vect$.
If the triple $(Y,L,P)$ is an object of $\link$ and the components of $L$ are $L_1,\dots,L_k$, set
\[
d = d(Y,L,P) = \sum_{i = 1}^k (|L_i \cap P|/2 -1).
\]
Then, by work of the author~\cite[Proposition~9.2]{polytope},
\[
\SFH(\W(Y,L,P)) \cong \HFLh(Y,L) \otimes V^{\otimes d},
\]
where $V = \Z_2^2$.

\section{Special cobordisms, sutured multi-diagrams, and naturality} \label{sec:special}

Our first goal is to extend the hat version of the cobordism maps introduced by Ozsv\'ath and Szab\'o~\cite{OSz10}
to the class of sutured manifold cobordisms that are trivial along the boundary.

\begin{defn} \label{def:special}
We say that a cobordism $\W = (W,Z,[\xi])$ from $(M_0,\g_0)$ to $(M_1,\g_1)$ is \emph{special} if
\begin{enumerate}
\item $\W$ is balanced,
\item $\partial M_0 = \partial M_1$, and $Z = \partial M_0 \times I$ is the trivial cobordism between them,
\item \label{it:invar} $\xi$ is an $I$-invariant contact structure on $Z$ such that
each $\partial M_0 \times \{t\}$ is a convex surface with dividing set
      $\g_0 \times \{t\}$ for every $t \in I$ with respect to the contact vector field~$\partial/\partial t$.
\end{enumerate}
In particular, it follows from \eqref{it:invar} that $\g_0 = \g_1$.
\end{defn}

Recall that we introduced the notion of equivalence and diffeomorphism of sutured manifold
cobordisms in Definition~\ref{defn:cobeq}.
Balanced sutured manifolds and equivalence classes of special cobordisms form a
subcategory~$\BSut'$ of~$\BSut$.

For a special cobordism~$\W = (W,Z,[\xi])$, we have $H^k(Z, \partial M_1) = 0$ for $k \in \{1,2\}$.
Hence Lemma~\ref{lem:spec} implies that, given an admissible
almost complex structure $J'$ on $Z$ such that $\s_\xi = \s_{\xi_{J'}}$, the restriction map
$q \colon \spinc(\W,J') \to \spinc(\W)$ is a bijection.

We now make $\SFH$ into a functor from $\BSut'$ to~$\Vect$; i.e., we define the map~$\Phi_{\W}$ if~$\W$ is a special cobordism.
For this, we generalize the work of Ozsv\'ath and Szab\'o~\cite{OSz10} on cobordism maps induced
on the Heegaard Floer homology of closed $3$-manifolds.

\subsection{Sutured multi-diagrams and pseudo-holomorphic polygons}
Some of the necessary steps have already been done by Grigsby and Wehrli~\cite{GW},
we review and extend their results first.
In particular, we include the contact structure~$\xi$ on~$Z$ into the theory.

\begin{defn} \label{defn:multi}
A \emph{balanced sutured multi-diagram} is a tuple
$(\S, \bolde^0,\dots,\bolde^n)$,
where $\S$ is a compact, oriented, surface without closed components, and
there is a non-negative integer~$d$ such that, for every $0 \le i \le n$,
the set $\bolde^i$ consists of~$d$ pairwise disjoint simple closed curves
$\eta^i_1,\dots,\eta^i_d \subset \Int(\S)$ that are linearly independent in~$H_1(\S)$.
\end{defn}

\begin{rem}
By a slight abuse of notation, we will also write $\bolde^i$ for the $1$-dimensional submanifold $\bigcup \bolde^i$ of $\Int(\S)$.
\end{rem}

Suppose that we are given a balanced sutured multi-diagram $(\S, \bolde^0,\dots,\bolde^n)$. Then we associate to it a balanced cobordism
\[
\W_{\eta^0,\dots,\eta^n} = (W_{\eta^0,\dots,\eta^n},Z_{\eta^0,\dots,\eta^n},[\xi_{\eta^0,\dots,\eta^n}]).
\]
For an illustration of the construction, see Figure~\ref{fig:0}.

For $0 \le i \le n$, let $U_i$ be the sutured
compression body obtained from~$\S \times I$ by attaching $2$-handles along $\bolde^i \times \{0\} \subset \S \times \{0\}$,
and rounding the corners along $\partial \S \times \{0\}$. Then
\[
\partial U_i = \S^1 \cup (\partial \S \times I) \cup \S_{\eta^i},
\]
where $\S^1 = \S \times \{1\}$ and $\S_{\eta^i}$ is obtained from $\S \times \{0\}$
by performing surgery along each component of $\bolde^i \times \{0\}$.

\begin{center}
\begin{figure}
\includegraphics{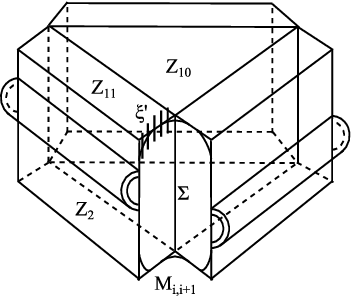}
\caption{A schematic picture of the cobordism $W_{\eta^0,\dots,\eta^n}$. Here $n = 2$,
and $P_{n+1} \times \S$ is represented by a triangle times a vertical interval.
Note how the corners are rounded. We also illustrate the $2$-plane field~$\xi'$ on~$Z$.}
\label{fig:0}
\end{figure}
\end{center}

Let $P_{n+1}$ denote a regular $(n+1)$-gon, with vertices
$v_i$ for $i \in \Z_{n+1}$, labeled in a clockwise fashion. Denote the edge connecting~$v_i$ and~$v_{i+1}$ by~$e_i$. Then let
\[
W_{\eta^0,\dots,\eta^n}  = \frac{(P_{n+1} \times \S) \sqcup \coprod_{i=0}^n (e_i \times U_i)}{(e_i \times \S) \sim (e_i \times \S^1)},
\]
where we round the corners along each $\{v_i\} \times \S$ for $i \in \Z_{n+1}$.

Denote by $(M_{i,j},\g_{i,j})$ the balanced sutured manifold defined by the diagram $(\S,\bolde^i,\bolde^j)$. Then
\[
M' = M_{0,1} \sqcup \dots \sqcup M_{n-1,n} \sqcup -M_{0,n} \subset \partial W,
\]
and we write $Z_{\eta^0,\dots,\eta^n}$ for $\partial W \setminus \Int(M')$.

Finally, we define the contact structure $\xi = \xi_{\eta^0,\dots,\eta^n}$ on the balanced sutured manifold
\[
(Z,\g) = (Z_{\eta^0,\dots,\eta^n}, \g_{0,1} \cup \dots \cup \g_{n-1,n} \cup \g_{n,0})
\]
by giving a sutured manifold hierarchy of $(Z,\g)$.
A sutured manifold hierarchy is a special case of a convex hierarchy by Honda et al.~\cite{tight}, hence it gives rise to a
contact structure~$\xi$, well-defined up to equivalence.

\begin{figure}[tb]
\includegraphics{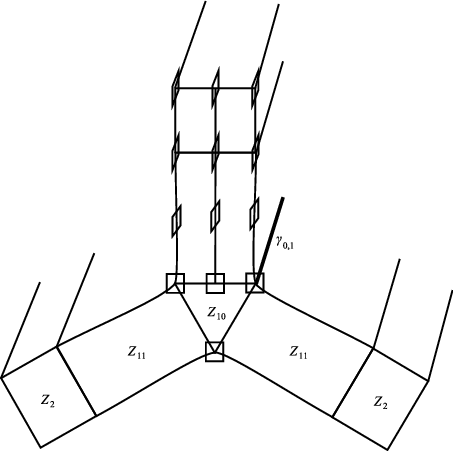}
\caption{A schematic picture of the sutured manifold $(Z,\g)$ and the $2$-plane field $\xi_{J'}$ isotopic
to the contact structure $\xi$.} \label{fig:1}
\end{figure}

Note that $Z$ consists of three parts: $Z_{10} = P_{n+1} \times \partial \S$, $Z_{11}=\bigcup_{i=0}^n (e_i \times \partial \S \times I)$,
and $Z_2 = \bigcup_{i=0}^n (e_i \times \S_{\eta^i})$, see Figure~\ref{fig:1}. We put $Z_1 = Z_{10} \cup Z_{11}$; then
\[
Z_1 = \left(P_{n+1} \cup \bigcup_{i=0}^n (e_i \times I) \right) \times \partial \S = P_{2n+2} \times \partial \S.
\]
Here, we get $P_{2n+2}$ by gluing the rectangle $e_i \times I$ to $P_{n+1}$ by identifying $e_i \times \{1\}$ with $e_i$ for each $i \in \Z_{n+1}$,
then rounding the corners at $v_0,\dots,v_n$. We still label the edge $e_i \times \{0\}$ of $P_{2n+2}$ by~$e_i$,
and the edge containing~$v_i$ is called~$g_i$. Observe that $\g_{i,i+1} = g_{i+1} \times \partial \S$.
Recall that when we defined the compression body $U_i$, we rounded its corners along $\partial \S \times \{0\}$.
When we glue~$Z_1$ and~$Z_2$, this corresponds to rounding the corners along $e_i \times \partial \S$,
so for every $x \in e_i$, the surfaces $\{x\} \times \partial \S \times I \subset Z_{11}$ and
$\{x\} \times \S_{\eta^i} \subset Z_2$ match smoothly, cf.~Figure~\ref{fig:0}.

Let $s_i$ be a $1$-manifold parallel to $\partial \S_{\eta^i}$ inside $\Int(\S_{\eta^i})$, and
let $A_i = e_i \times s_i$. Then
$A = A_0 \cup \dots \cup A_n$ is a decomposing surface, a union of product annuli, inside~$(Z,\g)$. Consider the sutured manifold
decomposition $(Z,\g) \rightsquigarrow^A (Z',\g')$. Then $(Z',\g')$ is the disjoint union of the product sutured manifolds
$(\S_{\eta^i} \times I, \partial \S_{\eta^i} \times I)$ for $i \in \Z_{n+1}$, and $S^1 \times D^2$ components with $2n+2$ longitudinal sutures on each.
Every product piece has a product disk decomposable contact structure, unique up to equivalence.
We further decompose each $S^1 \times D^2$ along $\{\text{pt}\} \times D^2$ to get a ball with a single suture, which
carries a unique tight contact structure. Note that we orient~$\{\text{pt}\} \times D^2$ such that
it is positively transverse to the~$\partial \S$ factor.
Our sequence of decompositions terminates in a product sutured manifold. Hence,
by the work of Honda et al.~\cite{tight}, we obtain a tight contact structure $\xi$ on $(Z,\g)$, which is well-defined up
to equivalence.

We will use the following lemma to show that the relative $\spinc$ structures used by
Grigsby and Wehrli~\cite[Proposition~3.7]{GW} and defined by a $2$-plane field~$\xi'$
along~$Z$ also define $\spinc$ structures relative to~$\xi$.

\begin{lem} \label{lem:bcond}
Consider $\W_{\eta^0,\dots,\eta^n} = (W,Z,[\xi])$. Let $\xi' = \xi'_{\eta^0,\dots,\eta^n}$ be the 2-plane field in
$TW|_Z$ such that on $Z_1 = P_{2n+2} \times \partial \S$ it is tangent to $\S$, and on $Z_2$ it is tangent to~$\S_{\eta^i}$.
(This is smooth on $Z$ since we rounded the corners along $e_i \times \partial \S$, cf.~Figure~\ref{fig:0}.)
Choose an arbitrary almost complex structure $J'$ on $TW|_Z$ such that~$\xi'$ consists of complex lines,
and let $\xi_{J'}$ denote the $2$-plane field of complex tangencies along~$Z$.
Then
\[
\s_\xi = \s_{\xi_{J'}} \in \spinc(Z,\g).
\]
\end{lem}

\begin{proof}
It suffices to check that the $2$-plane field $\xi_{J'}$ in $TZ$ never agrees with $-\xi$,
for some representative $\xi$ of the equivalence class $[\xi]$.
For an illustration of the following argument, see Figure~\ref{fig:1}.
Since $\xi'|_{Z_1}$ is tangent to $\S$ and is $J'$-invariant, the $2$-planes
$\xi_{J'}$ on $Z_1 = P_{n+1} \times \partial \S$ must be positively transverse to the $\partial \S$ factor.
On $Z_2$, the planes $\xi_{J'}$ agree with $\xi'$, which are tangent to $\S_{\eta^i}$.
The contact structure $\xi$ is a perturbation of the horizontal foliation on each $\S_{\eta^i} \times I$.
On $Z_1$, it is a perturbation of
the foliation by multi-saddles, so it is also transverse to the~$\partial \S$ factor.
In particular, $\xi$ and $\xi_{J'}$ are never opposite.
\end{proof}

The almost complex structures $J'$ in Lemma~\ref{lem:bcond} form a contractible space, so
homotopically it is unique, we denote it by $J'_{\eta^0,\dots,\eta^n}$. Since
$\s_\xi = \s_{\xi_{J'}}$ and it is admissible, we can talk about the set
$\spinc(\W_{\eta^0,\dots,\eta^n}, J'_{\eta^0,\dots,\eta^n})$ of $\spinc$ structures restricting
to~$J'_{\eta^0, \dots, \eta^n}$ along~$Z$. We are going to use the notation
\[
\spincu\left(\W_{\eta^0,\dots,\eta^n}\right) := \spinc\left(\W_{\eta^0,\dots,\eta^n}, J'_{\eta^0,\dots,\eta^n}\right).
\]
Furthermore, just as in Lemma~\ref{lem:spec}, we have a restriction map
\[
q \colon \spincu\left(\W_{\eta^0,\dots,\eta^n}\right) \to \spinc \left(\W_{\eta^0,\dots,\eta^n}\right).
\]

As usual, we denote by $\T_{\eta^i}$ the $d$-torus $\eta^i_1 \times \dots \times \eta^i_d$ inside $\sym^d(\S)$. For $i \in \Z_{n+1}$,
let $\x_{i+1} \in \T_{\eta^i} \cap \T_{\eta^{i+1}}$. Then we write $\pi_2(\x_0,\dots,\x_n)$ for the set of homotopy
classes of Whitney $(n+1)$-gons inside $\sym^d(\S)$ connecting $\x_0,\dots,\x_n$. We now recall
a result of Grigsby and Wehrli~\cite[Proposition~3.7]{GW},
which relates Whitney $(n+1)$-gons and $\spinc$ structures.

\begin{prop} \label{prop:ngon}
Suppose that $(\S,\bolde^0,\dots,\bolde^n)$ is a sutured multi-diagram, and $\W = \W_{\eta^0,\dots,\eta^n}$ is the associated cobordism.
Then there is a well-defined map
\[
\su \colon \pi_2(\x_0,\dots,\x_n) \to \spincu(\W)
\]
such that $\su(\Psi)|_{M_{i,i+1}} = \s(\x_i)$ for every $i \in \Z_{n+1}$.
\end{prop}

\begin{proof}
The construction of Grigsby and Wehrli~\cite{GW}, based on the work of Ozsv\'ath and Szab\'o~\cite[Section~8]{OSz},
associates to any Whitney $(n+1)$-gon~$u$ a $2$-plane field~$\xi_u$ on~$W$ minus a contractible $1$-complex~$c$
that agrees with $\xi' = \xi'_{\eta^0,\dots,\eta^n}$ along $Z$.
Let~$J$ be an almost complex-structure on~$W \setminus c$ such that~$\xi_u$ consists of complex lines.
By construction, $J|_Z = J'_{\eta^0,\dots,\eta^n}$.
The relative homology class of~$J$ gives an element~$\su(\Psi)$ of $\spincu(\W)$
that satisfies the property $\su(\Psi)|_{M_{i,i+1}} = \s(\x_i)$.
\end{proof}

\begin{defn}
Let $\su$ be the map introduced in Proposition~\ref{prop:ngon}, and let~$q$ be
the restriction map from $\spincu(\W)$ to $\spinc(\W)$. Then we define
\[
\s \colon \pi_2(\x_0,\dots,\x_n) \to \spinc(\W)
\]
to be the composition $q \circ \su$.
\end{defn}

It follows from Proposition~\ref{prop:ngon} that
$\s(\Psi)|_{M_{i,i+1}} = \s(\x_i)$ for every $i \in \Z_{n+1}$.

\begin{defn}
Let $(\S,\bolde^0,\dots,\bolde^n)$ be a balanced sutured multi-diagram. Let $D_1,\dots,D_l$ denote the closures of the components of
$\S \setminus (\bolde^0 \cup \dots \cup \bolde^n)$ disjoint from~$\partial \S$. Then the set of \emph{domains}
in $(\S,\bolde^0,\dots,\bolde^n)$ is
\[
D(\S,\bolde^0,\dots,\bolde^n) = \Z \langle\,D_1,\dots,D_l \,\rangle.
\]
For a domain $\D \in D(\S,\bolde^0,\dots,\bolde^n)$, we write $\D \ge 0$ if
$\D \in \Z_{\ge 0} \langle\,D_1,\dots,D_l \,\rangle$. As usual, if
\[
(\x_0,\dots,\x_n) \in (\T_{\eta^n} \cap \T_{\eta^0} ) \times \dots \times (\T_{\eta^{n-1}} \cap \T_{\eta^n}),
\]
then $D(\x_0,\dots,\x_n)$ denotes the set of domains connecting $\x_0,\dots,\x_n$.

Finally, an \emph{$(n+1)$-periodic domain} is an element $\P \in D(\S,\bolde^0,\dots,\bolde^n)$ such that~$\partial\P$
is a $\Z$-linear combination of curves in $\bolde^0,\dots,\bolde^n$.
\end{defn}

The following proposition implies that any two Whitney $(n + 1)$-gons in the affine set $\pi_2(\x_0,\dots,\x_n)$ differ by an
$(n + 1)$-periodic domain.

\begin{prop} \label{prop:peri}
If $\pi_2(\x_0,\dots,\x_n) \neq \emptyset$, then
\[
\pi_2(\x_0,\dots,\x_n) \cong \ker\left(\bigoplus_{i=0}^n H_1(\bolde^i) \to H_1(\S)\right) \cong H_2(W_{\eta^0,\dots,\eta^n}).
\]
Furthermore,
\[
\mathrm{coker} \left (\bigoplus_{i=0}^n H_1(\bolde^i) \to H_1(\S)\right) \cong H_1(W_{\eta^0,\dots,\eta^n}).
\]
\end{prop}

\begin{proof}
This was shown by Grigsby and Wehrli~\cite[Proposition~3.3 and~3.4]{GW}.
\end{proof}

The correspondence in Proposition~\ref{prop:peri} can be made explicit by associating to each periodic
domain~$\mathcal{P}$ an element~$H(\mathcal{P})$ of $H_2(W_{\eta^0,\dots,\eta^n})$, as follows.
Pick an interior point $x \in P_{n+1}$, and connect~$x$ to every~$e_i$ by a straight arc~$a_i$.
For each curve $\eta^i_j \in \bolde^i$, let $E^i_j \subset W_{\eta^0,\dots,\eta^n}$
denote the union of the annulus $a_i \times \eta^i_j \subset P_{n+1} \times \S$,
the annulus $(a_i \cap e_i) \times\eta^i_j \times I$ in~$Z_{11}$,
and the core disk of the $2$-handle attached to $(a_i \cap e_i) \times \eta^i_j \times \{0\}$
in~$Z_2$.
Suppose that
\[
\partial \mathcal{P} = \sum_{i = 0}^n \sum_{j = 1}^d e^j_i \eta^i_j.
\]
Then
\[
H(\mathcal{P}) = \{x\} \times \mathcal{P} + \sum_{i = 0}^n \sum_{j = 1}^d e^j_i E^i_j \in H_2(W_{\eta^0,\dots,\eta^n}).
\]

The following is a corrected version of \cite[Proposition 3.9]{GW}.
Recall from Proposition~\ref{prop:ngon}
that~$\su$ is the relative $\spinc$ map defined by Grigsby and Wehrli,
which is different from the map~$\s$.

\begin{prop}
Let $\Psi$, $\Psi' \in \pi_2(\x_0,\dots,\x_n)$. Then $\su(\Psi) = \su(\Psi')$
if and only if~$\Psi-\Psi'$ can be written as a
$\Z$-linear combination of doubly-periodic domains.
\end{prop}

\begin{defn}
A balanced sutured multi-diagram is \emph{admissible} if every non-trivial
$(n+1)$-periodic domain has both positive and negative coefficients.
\end{defn}

The following statement is \cite[Lemma~3.12]{GW}.

\begin{lem}
Every balanced sutured multi-diagram is isotopic to an admissible one.
\end{lem}

This weak form of admissibility is specific to the hat version of Heegaard Floer
homology. It enables us to talk about various polygon counts without restricting
to a particular $\spinc$ structure, making all sums finite.
For the following, see Grigsby and Wehrli~\cite[Proposition~3.14]{GW}.

\begin{prop} \label{prop:finite}
If $(\S,\bolde^0,\dots,\bolde^n)$ is admissible, then for every
\[
(\x_0,\dots,\x_n) \in (\T_{\eta^n} \cap \T_{\eta^0} ) \times \dots \times (\T_{\eta^{n-1}} \cap \T_{\eta^n}),
\]
the set $\{\,\D \in D(\x_0,\dots \x_n) \,\colon\, \D \ge 0 \,\}$ is finite.
\end{prop}

Let $(\S,\bolde^0,\dots,\bolde^n)$ be an admissible sutured multi-diagram, and for every $i \in \{1, \dots,n \}$,
let $\x_i \in \T_{\eta^{i-1}} \cap \T_{\eta^i}$ and $\y \in \T_{\eta^0} \cap \T_{\eta^n}$.
Fix a complex structure on $\S$, and a $1$-parameter variation of the induced almost complex structure on $\sym^d(\S)$.
As usual, we denote by $\M(\phi)$ the moduli-space of pseudo-holomorphic representatives of Whitney
$(n+1)$-gons lying in the homotopy class $\phi \in \pi_2(\x_1,\dots,\x_n,\y)$.
The Maslov index; i.e., the expected dimension, of $\M(\phi)$ is denoted by $\mu(\phi)$.
If $n=1$ and $\mu(\phi) = 1$, then there is a natural $\R$-action on $\M(\phi)$, we let
$\widehat{\M}(\phi) = \M(\phi)/\R$. When $n > 1$ and $\mu(\phi) = 0$, the moduli space $\M(\phi)$ is compact.
Then $\#\M(\phi)$ denotes the number of points in $\M(\phi)$ modulo two.
In the case $n = 1$ and $\mu(\phi) = 1$, the reduced moduli space $\widehat{\M}(\phi)$ is compact.

For $0 \le i < j \le n$, we let
\[
CF(\S,\bolde^i,\bolde^j) = \Z_2 \langle\,\T_{\eta^i} \cap \T_{\eta^j} \,\rangle.
\]
This becomes a chain complexes when endowed with the differential that counts points in
$\widehat{\M}(\phi)$ modulo two, where~$\phi$ is a homotopy class of Whitney bigons with
boundary on $\T_{\eta^i}$ and $\T_{\eta^j}$ and having $\mu(\phi) = 1$. Its homology is
the sutured Floer homology group  $\SFH(\S,\bolde^i,\bolde^j)$.

\begin{defn} \label{defn:products}
Let $(\S,\bolde^0,\dots,\bolde^n)$ be an admissible sutured multi-diagram such that $n \ge 2$,
and fix a relative $\spinc$ structure $\su \in \spincu(\W_{\eta^0,\dots,\eta^n})$. Then we have chain maps
\[
f_{\eta^0,\dots,\eta^n} \colon \bigotimes_{i=1}^n CF(\S,\bolde^{i-1},\bolde^i) \to CF(\S,\bolde^0,\bolde^n)
\]
and
\[
f_{\eta^0,\dots,\eta^n}(\,\cdot\,,\su) \colon \bigotimes_{i=1}^n CF(\S,\bolde^{i-1},\bolde^i,\su|_{M_{i-1,i}}) \to
CF(\S,\bolde^0,\bolde^n,\su|_{M_{0,n}}),
\]
defined by the formulas
\[
f_{\eta^0,\dots,\eta^n}(x_1 \otimes \dots \otimes x_n) = \sum_{\y \in \T_{\eta^0} \cap \T_{\eta^n}}
\sum_{\{\,\phi \in \pi_2(\x_1,\dots,\x_n,\y)\, \colon \,\mu(\phi) = 0\,\}} \#\M(\phi) \cdot \y
\]
and
\[
f_{\eta^0,\dots,\eta^n}(x_1 \otimes \dots \otimes x_n,\su) = \sum_{\y \in \T_{\eta^0} \cap \T_{\eta^n}}
\sum_{\{\,\phi \in \pi_2(\x_1,\dots,\x_n,\y)\, \colon \,\mu(\phi) = 0,\,\su(\phi) =\su\,\}} \#\M(\phi) \cdot \y.
\]
We denote by $F_{\eta^0,\dots,\eta^n}$ and $F_{\eta^0,\dots,\eta^n}(\,\cdot\,,\su)$ the maps induced on the homology.
Analogous maps can be defined for $\s \in \spinc(\W)$, replacing $\su(\phi)$ by $\s(\phi)$.
\end{defn}

The finiteness of the above sums is ensured by Proposition~\ref{prop:finite},
since if $\phi$ supports a pseudo-holomorphic representative, then its domain $\D(\phi) \ge 0$.

\begin{prop} \label{prop:fin}
Let $(\S,\bolde^0,\dots,\bolde^n)$ be an admissible multi-diagram, and set $\W = \W_{\eta^0,\dots,\eta^n}$.
Then there are only finitely many $\su \in \spincu(\W)$ for which the map $f_{\eta^0,\dots,\eta^n}(\,\cdot\,,\su)$ is non-zero, and
\[
f_{\eta^0,\dots,\eta^n} = \sum_{\su \in \spincu(\W)} f_{\eta^0,\dots,\eta^n}(\,\cdot\,,\su).
\]
An analogous statement holds for $F_{\eta^0,\dots,\eta^n}$ and $F_{\eta^0,\dots,\eta^n}(\,\cdot\,,\su)$,
and also for $\spinc$ structures $\s \in \spinc(\W)$.
\end{prop}

\begin{proof}
Each $\T_{\eta^i} \cap \T_{\eta^j}$ is finite, so there are only finitely many choices
for $\x_1,\dots,\x_n$ and~$\y$, and for each choice, there are only finitely many $\phi \in \pi_2(\x_1,\dots,\x_n,\y)$
such that $\M(\phi) \neq \emptyset$ by Proposition~\ref{prop:finite}. Finally, such a $\phi$ can only appear in
the formula defining $f_{\eta^0,\dots,\eta^n}(\,\cdot\,,\su(\phi))$. The result follows.
\end{proof}

We now generalize the associativity theorem of the triangle maps due to Ozsv\'ath and Szab\'o~\cite[Theorem 8.16]{OSz} to the sutured setting.
Fix a sutured quadruple diagram $(\S,\bolde^0,\dots,\bolde^3)$, and let $\W_{\eta^0,\dots,\eta^3}$ be the corresponding
cobordism. Then we have a restriction map
\[
\spincu(\W_{\eta^0,\dots,\eta^3}) \to \spincu(\W_{\eta^0,\eta^1,\eta^2}) \times \spincu(\W_{\eta^0,\eta^2,\eta^3}),
\]
which corresponds to splitting the cobordism $\W_{\eta^0,\dots,\eta^3}$ along an embedded copy of~$M_{02}$.
There is a subgroup
\[
\d H^1(M_{02},\partial M_{02}) < H^2(W_{\eta^0,\dots,\eta^3},Z_{\eta^0,\dots,\eta^3})
\]
whose orbits on $\spincu(\W_{\eta^0,\dots,\eta^3})$ are the fibers of the restriction map,
where~$\d$ is the coboundary map in the corresponding relative Mayer-Vietoris sequence. Similarly, we have a restriction map
\[
\spincu(\W_{\eta^0,\dots,\eta^3}) \to \spincu(\W_{\eta^0,\eta^1,\eta^3}) \times \spincu(\W_{\eta^1,\eta^2,\eta^3}),
\]
which corresponds to splitting along $M_{13}$, and a subgroup
\[
\d H^1(M_{13},\partial M_{13}) < H^2(W_{\eta^0,\dots,\eta^3},Z_{\eta^0,\dots,\eta^3}).
\]

\begin{thm} \label{thm:assoc}
Let $(\S,\bolde^0,\bolde^1,\bolde^2,\bolde^3)$ be an admissible sutured quadruple diagram,
and fix a
\[
\d H^1(M_{02},\partial M_{02}) + \d H^1(M_{13},\partial M_{13})
\]
orbit~$\mathfrak{S}$
in $\spincu(W_{\eta^0,\dots,\eta^3})$. For any~$\su \in \mathfrak{S}$ and $i \in \{\,0,1,2\,\}$,
the restriction $\s_{i,i+1} = \su|_{M_{i,i+1}}$ is independent of the choice of $\su$;
pick an element
\[
x_{i,i+1} \in \SFH(\S,\bolde^i,\bolde^{i+1},\s_{i,i+1}).
\]
Furthermore, for~$0 \le i < j < k \le 4$, let $F_{ijk} = F_{\eta^i,\eta^j,\eta^k}$
and $\s_{ijk} = \su|_{\W_{\eta^i,\eta^j,\eta^k}}$. Then
\[
\sum_{\su \in \mathfrak{S}} F_{023}(F_{012}(x_{01}\otimes x_{12}, \s_{012}) \otimes x_{23}, \s_{023}) =
\]
\[
= \sum_{\su \in \mathfrak{S}} F_{013}(x_{01} \otimes F_{123}(x_{12} \otimes x_{23}, \s_{123}),\s_{013}).
\]
\end{thm}

\begin{proof}
Every subdiagram of an admissible sutured multi-diagram is also admissible. Hence, the proof of
Ozsv\'ath and Szab\'o~\cite[Theorem 8.16]{OSz} works in this setting too, since the admissibility
of $(\S,\bolde^0,\dots,\bolde^3)$ ensures the finiteness of all the counts of pseudo-holomorphic bigons,
triangles, and rectangles that appear in the formula for the chain homotopy connecting the two sides.
\end{proof}

In a similar manner, one can prove an associativity result without
fixing an orbit~$\mathfrak{S}$ of $\spinc$ structures on $\W_{\eta^0,\dots,\eta^3}$.

\begin{thm}
Let $(\S,\bolde^0,\bolde^1,\bolde^2,\bolde^3)$ be an admissible sutured quadruple diagram, and for every $i \in \{\,0,1,2\,\}$,
pick an element $x_{i,i+1} \in \SFH(\S,\bolde^i, \bolde^{i+1})$. Then
\[
F_{023}(F_{012}(x_{01}\otimes x_{12}) \otimes x_{23}) = F_{013}(x_{01} \otimes F_{123}(x_{12} \otimes x_{23})).
\]
\end{thm}

Note that writing associativity using $\spinc(\W_{\eta^0,\dots,\eta^3})$ would be cumbersome
as such $\spinc$ structures do not restrict to the cobordisms $\W_{\eta^0,\eta^1,\eta^2}$ etc.

\subsection{Naturality of sutured Floer homology} \label{sec:naturality}

When the first draft of this paper appeared,
the assignment of a Heegaard Floer group to a $3$-manifold was not functorial.
Ozsv\'ath and Szab\'o~\cite{OSz} showed that different admissible Heegaard diagrams of the same manifold give rise
to isomorphic Floer homology groups.
However, this is not sufficient to define maps induced by cobordisms, or to talk about the diffeomorphism action.
It had also caused some confusion in the case of the contact invariant, as it is unclear what it means to
be an element of Heegaard Floer homology. Being aware of this issue, Ozsv\'ath and Szab\'o~\cite[Theorem 2.1]{OSz10}
constructed ``canonical isomorphisms'' for equivalent Heegaard diagrams.
However, as the author noticed, they did not check that these maps were indeed isomorphisms,
or that the composition of two canonical isomorphisms was a canonical isomorphism.
Furthermore, it turns out it is not enough to say that a Heegaard diagram of a manifold is an abstract triple $(\S,\bolda,\boldb)$,
one has to take into account how~$\Sigma$ is embedded in the manifold.
Together with Dylan Thurston, we settled the above issues in~\cite{naturality}.
This was a prerequisite for defining our cobordism maps.

In this section, we review the functorial construction of sutured Floer homology, which includes $\widehat{HF}$.
The other flavors of Heegaard Floer homology can be addressed in an analogous manner.
In particular, we explain how to canonically associate a group $\SFH(M,\g)$ to every balanced sutured manifold $(M,\g)$
via a simple limiting procedure.
This will enable us to precisely define how diffeomorphisms act on sutured Floer homology.
For more details, we refer the reader to~\cite{naturality}.

The following definition is a refinement of \cite[Definition 2.7]{sutured}.
By a balanced diagram $(\S,\bolda,\boldb)$, we
mean a balanced sutured double-diagram in the sense of Definition \ref{defn:multi}.
Recall that $(\S,\bolda,\boldb)$ defines a sutured manifold $(M_{\a,\b},\g_{\a,\b})$
as in \cite[Definition 2.8]{sutured}, which is unique up to diffeomorphism relative to $\S$.
The following is~\cite[Definition~2.14]{naturality}.

\begin{defn}
Let $(M,\g)$ be a sutured manifold. Then we say that $(\S, \bolda, \boldb)$ is a \emph{diagram of $(M,\g)$}
if
\begin{enumerate}
\item $\S \subset M$ is an oriented surface with $\partial \S =
  s(\g)$ as oriented 1-manifolds,
\item the components of $\bolda$ bound disjoint disks to the negative side of $\S$, and the components of $\boldb$
bound disjoint disks to the positive side of $\S$,
\item if we compress $\S$ along $\bolda$, we get a surface isotopic to
  $R_-(\g)$ relative to $\g$,
\item if we compress $\S$ along $\boldb$, we get a surface isotopic to $R_+(\g)$ relative to $\g$.
\end{enumerate}
\end{defn}

Let $\H = (\S,\bolda,\boldb)$ be an \emph{admissible} diagram of the balanced sutured manifold~$(M,\g)$.
Since $\chi(R_+(\g)) = \chi(R_-(\g))$, we have $|\bolda| = |\boldb|$, we denote this number by $k$.
We also choose a complex structure $\j$ on $\S$ and a generic perturbation $J_s \subset \mathcal{U}$
of the induced complex structure on $\sym^g(\S)$, where $\mathcal{U}$ is a contractible set
of almost complex structures.
The sutured Floer homology $\SFH_{J_s}(\H)$ is the Lagrangian intersection Floer homology
of the tori $\T_\a$ and $\T_\b$ inside $\text{Sym}^k(\S)$ endowed with a particular
symplectic structure compatible with $J_s$. Recall that we are using $\Z_2$-coefficients.
Furthermore, this splits along relative $\spinc$ structures:
\[
\SFH_{J_s}(\H) = \bigoplus_{\s \in \spinc(M,\g)} \SFH_{J_s}(\H,\s).
\]

Given two different choices $(\j, J_s)$ and $(\j',J_s')$ of complex structures and perturbations,
Ozsv\'ath and Szab\'o~\cite[Lemma~2.11]{OSz10} constructed isomorphisms
\[
\Phi_{J_s \to J_s'} \colon \SFH_{J_s}(\H,\s) \to \SFH_{J_s'}(\H,\s).
\]
These isomorphisms are natural in the sense that
\[
\Phi_{J_s \to J_s''} = \Phi_{J_s' \to J_s''} \circ \Phi_{J_s \to J_s'}
\]
and $\Phi_{J_s \to J_s}$ is the identity. Then we define $\SFH(\H,\s)$
to be the set of those elements $g$ of $\prod_{J_s} \SFH_{J_s}(\H,\s)$
for which $\Phi_{J_s \to J_s'}(g(J_s)) = g(J_s')$ for any pair of
generic perturbations~$J_s$ and~$J_s'$ in~$\mathcal{U}$.

Let $\H(M,\g)$ be the set of all admissible diagrams of $(M,\g)$.
(Note that this is indeed a set, not a proper class, as $\S \subset M$.)
Our goal is to construct an isomorphism
\[
F_{\H,\H'} \colon \SFH(\H) \to \SFH(\H')
\]
for any pair $(\H,\H') \in \H(M,\g) \times \H(M,\g)$
such that these isomorphisms respect the splitting along $\spinc(M,\g)$.
We require that the groups $\SFH(\H)$ and isomorphisms~$F_{\H,\H'}$
form a \emph{transitive system}; i.e., that $F_{\H,\H} = \text{Id}_{\H}$ and
\[
F_{\H,\H''} = F_{\H',\H''} \circ F_{\H,\H'}
\]
for every $\H$, $\H'$, $\H'' \in \H(M,\g)$.
Then $\SFH(M,\g)$ is the set of elements $g$ of $\prod_{\H \in \H(M,\g)} \SFH(\H)$
for which $F_{\H,\H'}(g(\H)) = g(\H')$ for every $\H$, $\H' \in \H(M,\g)$.
Projection onto the factor corresponding to the diagram $\H$ gives an isomorphism
\[
P_{\H} \colon \SFH(M,\g) \to \SFH(\H).
\]

Any pair of diagrams $\H$, $\H' \in \H(M,\g)$ can be connected by a sequence of
moves that we describe next. After that, we define an elementary isomorphism $F_{\H,\H'}$ for each such move.
One of the main results of~\cite{naturality} is that, given any two diagrams $\H$, $\H' \in \H(M,\g)$,
no matter how we get from $\H$ to $\H'$, the composition of the corresponding
elementary isomorphisms is always the same.

Given two $1$-manifolds $\boldd$ and $\boldd'$ on a surface $\S$, we say that
they are \emph{equivalent}, and write $\boldd \sim \boldd'$,
if one can obtain one from the other via a sequence of isotopies and handleslides.
Suppose that $(\S,\bolda,\boldb)$ and $(\S,\bolda',\boldb')$ are both
diagrams of the sutured manifold $(M,\g)$. Then, by~\cite[Lemma~2.11]{naturality},
we have $\bolda \sim \bolda'$ and $\boldb \sim \boldb'$.

\begin{defn}
We say that the diagrams $(\S_1, \bolda_1, \boldb_1)$ and $(\S_2, \bolda_2, \boldb_2)$ are
\emph{strongly equivalent} if $\S_1 = \S_2$, $\bolda_1 \sim \bolda_2$, and $\boldb_1 \sim \boldb_2$.
\end{defn}

\begin{defn}
The sutured diagram $(\S_2, \bolda_2, \boldb_2)$ is obtained from $(\S_1, \bolda_1, \boldb_1)$ by a \emph{stabilization} if
\begin{itemize}
\item there is a disk $D \subset \S_1$ and a punctured torus $T \subset \S_2$ such that we have $\S_1 \setminus D = \S_2 \setminus T$,
\item $\bolda_1  = \bolda_2 \cap (\S_2 \setminus T)$ and $\boldb_1  = \boldb_2 \cap (\S_2 \setminus T)$,
\item $\bolda_2 \cap T$ and $\boldb_2 \cap T$ are simple closed curves that intersect each other transversely in a single point.
\end{itemize}
In this case, we also say that $(\S_1, \bolda_1, \boldb_1)$ is obtained from $(\S_2, \bolda_2, \boldb_2)$ by a \emph{destabilization}.
\end{defn}

\begin{defn}
Given diagrams $\H_1 = (\S_1, \bolda_1, \boldb_1)$ and $\H_2 = (\S_2, \bolda_2, \boldb_2)$, a diffeomorphism $d \colon \H_1 \to \H_2$
is an orientation preserving diffeomorphism $d \colon \S_1 \to \S_2$ such that $d(\bolda_1) = \bolda_2$ and $d(\boldb_1) = \boldb_2$.

Suppose that $\H_1$ and~$\H_2$ are both diagrams for $(M,\gamma)$,
and let $\iota_i \colon \Sigma_i \to M$ be the inclusion for $i \in \{1,2\}$.
Then a diffeomorphism $d \colon \H_1 \to \H_2$ is \emph{isotopic to the identity in~$M$}
if $\iota_2 \circ d \colon \S_1 \to M$ is isotopic to
$\iota_1 \colon \S_1 \to M$ relative to $s(\gamma)$.
\end{defn}

For any two diagrams $\H$, $\H' \in \H(M,\g)$, there exists a sequence
\[
\H_1, \dots, \H_n \in \H(M,\g)
\]
such that for every $i \in \{\, 1, \, \dots, n-1 \,\}$,
the diagrams~$\H_i$ and~$\H_{i+1}$ are related by a strong equivalence,
a (de)stabilization, or a diffeomorphism isotopic to the identity in~$M$.

We now review \cite[Lemma~9.2]{naturality}.
If $\H = (\S,\boldd,\boldd')$ is an admissible diagram and $\boldd \sim \boldd'$,
then this defines the connected sum of a product sutured manifold and~$k$ copies of $S^1 \times S^2$.
(Note that if $\S$ is disconnected, then we might be taking connected sums along different components
of a disconnected product sutured manifold diffeomorphic to
$(\S(\boldd) \times I, \partial\S \times I)$, where $\S(\boldd)$ is $\S$ compressed along~$\boldd$.)
There is a unique $\spinc$ structure $\s_0$ on this manifold with~$c_1(\s_0) = 0$ and
which can be represented by a vertical vector field on the product summand.
Then
\[
\SFH(\H, \s_0) \cong \Lambda^*H_1(S^1 \times S^2; \Z_2)
\]
as relative $\Z$-graded groups. Hence, the ``top'' non-zero grading is isomorphic to~$\Z_2$,
we denote its generator by $\Theta_{\d,\d'}$.

If $(\S,\bolda,\boldb,\boldb')$ is an admissible triple, then we write $\Psi^{\bolda}_{\boldb \to \boldb'}$
for the map
\[
F_{\a,\b,\b'}(- \otimes \Theta_{\b,\b'}) \colon \SFH(\S,\bolda,\boldb) \to \SFH(\S,\bolda,\boldb').
\]
Similarly, if the triple $(\S,\bolda',\bolda,\boldb)$ is admissible and $\bolda \sim \bolda'$,
then we write $\Psi^{\bolda \to \bolda'}_{\boldb}$ for the map
\[
F_{\a',\a,\b}(\Theta_{\a',\a} \otimes -) \colon \SFH(\S,\bolda,\boldb) \to \SFH(\S,\bolda',\boldb).
\]
According to~\cite[Proposition~9.8]{naturality}, these maps are isomorphisms.

Suppose that the diagrams $\H = (\S,\bolda,\boldb)$ and $\H' = (\S,\bolda',\boldb')$ are admissible
and strongly equivalent. If the quadruple diagram $(\S,\bolda,\boldb,\bolda',\boldb')$ is admissible,
then let
\[
\Psi^{\bolda \to \bolda'}_{\boldb \to \boldb'} :=
\Psi^{\bolda'}_{\boldb \to \boldb'} \circ \Psi^{\bolda \to \bolda'}_{\boldb}.
\]
If $(\S,\bolda,\boldb,\bolda',\boldb')$ is not necessarily admissible,
by~\cite[Lemma~9.3]{naturality}, there exist attaching sets $\ol{\bolda}$, $\ol{\boldb} \subset \S$
such that the quadruple diagrams $(\S, \bolda, \boldb, \ol{\bolda}, \ol{\boldb})$
and~$(\S, \bolda', \boldb', \ol{\bolda}, \ol{\boldb})$
are both admissible. Then we obtain an isomorphism
\[
\Phi^{\bolda \to \bolda'}_{\boldb \to \boldb'} :=
\Psi^{\ol{\bolda} \to \bolda'}_{\ol{\boldb} \to \boldb'} \circ \Psi^{\bolda \to \ol{\bolda}}_{\boldb \to \ol{\boldb}}
\]
from $\SFH(\H)$ to $\SFH(\H')$ that is independent of the choice of $\ol{\bolda}$ and $\ol{\boldb}$.
Furthermore, when $(\S,\bolda,\boldb,\bolda',\boldb')$ is admissible, then
$\Phi^{\bolda \to \bolda'}_{\boldb \to \boldb'} = \Psi^{\bolda \to \bolda'}_{\boldb \to \boldb'}$.

We show in~\cite[Section~9]{naturality} that all the isomorphisms above are functorial
in the obvious way. For example, if the diagrams $(\S,\bolda,\boldb)$, $(\S,\bolda',\boldb')$,
and $(\S,\bolda'',\boldb'')$ are all admissible, then
\[
\Phi^{\bolda \to \bolda''}_{\boldb \to \boldb''} =
\Phi^{\bolda' \to \bolda''}_{\boldb' \to \boldb''} \circ \Phi^{\bolda \to \bolda'}_{\boldb \to \boldb'},
\]
and $\Phi^{\bolda \to \bolda}_{\boldb \to \boldb}$ is the identity.

Given a diffeomorphism $d \colon \H \to \H'$ between the diagrams $\H = (\S, \bolda, \boldb)$
and $\H' = (\S',\bolda',\boldb')$, we defined an isomorphism
\[
d_* \colon \SFH(\H) \to \SFH(\H')
\]
in \cite[Definition~9.19]{naturality}. For this end, choose a perturbation $(\j, J_s)$
for $\H$. We can push this forward along $d$ to obtain a perturbation $(\j', J_s')$ for $\H'$.
Then $d$ induces a tautological isomorphism
\[
d_{J_s,J_s'} \colon \SFH_{J_s}(\H) \to \SFH_{J_s'}(\H').
\]
These isomorphisms then descend to the limits along the different perturbations, giving
rise to $d_*$. When $\H$ and $\H'$ are both diagrams of $(M,\g)$ and~$d$ is
isotopic to the identity in~$M$, we let $F_{\H,\H'}$ be $d_*$.

Finally, suppose that the diagram $\H' = (\S',\bolda',\boldb')$ is obtained from
$\H = (\S, \bolda, \boldb)$ via a stabilization. Then $\bolda' = \bolda \cup \{\a\}$
and $\boldb' = \boldb \cup \{\b\}$, where $\a \cap \b$ consists of a single point~$c$.
The map
\[
\sigma_{\H \to \H'} \colon CF(\H) \to CF(\H')
\]
mapping $\x \in \T_\a \cap \T_\b$ to $\x \times \{c\} \in \T_{\a'} \cap \T_{\b'}$
is a chain map for a suitable choice of complex structures on~$\S$ and~$\S'$,
and induces an isomorphism $F_{\H,\H'}$ on homology.
When $\H'$ is obtained from $\H$ via a destabilization, we let
$F_{\H,\H'} = F_{\H',\H}^{-1}$.

As explained above, given any pair of admissible diagrams $\H$, $\H' \in \H(M,\g)$,
we obtain a ``canonical'' isomorphism
\[
F_{\H,\H'} \colon \SFH(\H) \to \SFH(\H')
\]
by composing the elementary isomorphisms associated to an arbitrary sequence of
Heegaard moves connecting $\H$ and $\H'$.

Now we define the diffeomorphism action on sutured Floer homology.
Consider a diffeomorphism $\phi \colon (M,\g) \to (M',\g')$, and pick an admissible
diagram $\H = (\S,\bolda,\boldb)$ for $(M,\g)$. Let $d = \phi|_{\S}$, then
$\H' = d(\H)$ is an admissible diagram of $(M',\g')$.
We define $\phi_* \colon \SFH(M,\g) \to \SFH(M',\g')$ to be $(P_{\H'})^{-1} \circ d_* \circ P_{\H}$.
In other words, the following diagram is commutative:
\[
\xymatrix{
\SFH(M,\g) \ar[r]^{\phi_*} \ar[d]^{P_{\H}} &\SFH(M',\g') \ar[d]^{P_{\H'}} \\
\SFH(\H) \ar[r]^{d_*} &\SFH(\H').
}
\]

We conclude this section with a lemma that will be useful later on.

\begin{lem} \label{lem:n-commute}
Let $d \colon \S \to \S'$ be a diffeomorphism mapping the admissible sutured multi-diagram
$(\S,\eta^0, \dots, \eta^n)$ to $(\S',\nu^0,\dots, \nu^n)$.
For every $i \in \Z_n$, this induces an isomorphism
\[
d^i_* \colon \SFH(\S,\eta^i,\eta^{i+1}) \to \SFH(\S',\nu^i,\nu^{i+1}).
\]
Then, for every $x_i \in \SFH(\S,\eta^{i},\eta^{i+1})$ for $i \in \{\,0, \dots, n-1\,\}$, we have
\[
d^n_*\left(F_{\eta^0,\dots,\eta^n}(x_0 \otimes\dots \otimes x_{n-1})\right) =
F_{\nu^0,\dots,\nu^n}\left(d^0_*(x_0) \otimes \dots \otimes d^{n-1}_*(x_{n-1})\right).
\]
\end{lem}

\begin{proof}
Choose a complex structure $\j$ on $\S$, and let $\j' = d_*(\j)$,
together with corresponding $1$-parameter perturbations $J_s$ on $\sym^k(\S)$ and $J_s'$ on $\sym^k(\S')$.
Then $\sym^k(d)$ is a symplectomorphism from $\sym^k(\S)$ to $\sym^k(\S')$ that maps
the Lagrangians $\T_{\eta^0}, \dots, \T_{\eta^n}$ to $\T_{\nu^0}, \dots, \T_{\nu^n}$, respectively,
and intertwines the almost complex structures $J_s$ and $J_s'$. Hence, the statement becomes a tautology.
\end{proof}

\section{The map associated with a framed link}

Here, we generalize the work of Ozsv\'ath and Szab\'o \cite[Section~4]{OSz10}.
Some of the following notions already appeared in the paper of  Grigsby and Wehrli~\cite[Section~4]{GW}
for framed links in product sutured manifolds, in the context of a link surgery spectral sequence.

\begin{defn}
A \emph{framed link} $\L$ in a sutured manifold $(M,\g)$ is a collection of~$n$ pairwise disjoint,
smoothly embedded circles $K_1, \dots, K_n \subset \Int(M)$,
together with a choice of homology classes $\ell_i \in H_1(\partial N(K_i))$ satisfying $m_i \cdot \ell_i = 1$,
where~$m_i$ is the meridian of~$K_i$.

By attaching $2$-handles along the components of the framed link~$\L$,
we naturally get a special cobordism $W(\L)$ from $(M,\g)$ to a sutured manifold $(M(\L),\g)$.
Note that $M(\L)$ is obtained by surgery along $\L$, and $\g$ is left unchanged.
We call the cobordism $W(\L)$ the \emph{trace} of the surgery.

\end{defn}

For every framed link~$\L$ in~$(M,\g)$ and $\spinc$ structure $\s \in \spinc(W(\L))$, we are going to define maps
\[
F_{M,\L} \colon \SFH(M,\g) \to \SFH(M(\L),\g) \text{ and}
\]
\[
F_{M,\L,\s} \colon \SFH \left(M,\g,\s|_M \right) \to \SFH \left(M(\L),\g,\s|_{M(\L)}\right).
\]

\begin{defn}
A \emph{bouquet~$B(\L)$ for the link~$\L$} is a $1$-complex embedded in~$M$ which is the union of~$\L$
with a collection of arcs $a_1,\dots,a_n$, such that for every $i \in \{\, 1,\dots, n\,\}$,
the arc $a_i$ connects $K_i$ and $R_+(\g)$. We denote the punctured torus
$\partial N(K_i) \setminus N(a_i)$ by $F_i$.
\end{defn}

\begin{defn} \label{defn:bouquet}
A \emph{sutured triple diagram subordinate to the bouquet $B(\L)$} is a triple diagram
\[
(\S,\bolda,\boldb,\boldd) = (\S, \{\,\a_1,\dots,\a_d\,\}, \{\,\b_1,\dots, \b_d\,\}, \{\,\d_1,\dots,\d_d\,\}),
\]
such that
\begin{enumerate}
\item \label{it:b1} the triple $\left(\S, \{\,\a_1,\dots,\a_d\,\}, \{\,\b_{n+1},\dots,\b_d\,\} \right)$
      is a diagram of the sutured manifold $(M',\g') = (M \setminus N(B(\L)),\g)$
      (in particular, $\S \subset M'$),
\item \label{it:b2} the curves $\d_{n+1}, \dots, \d_d$ are small isotopic translates of $\beta_{n+1},\dots,\b_d$,
\item \label{it:b3} after compressing $\S$ along $\b_{n+1},\dots,\b_d$, for $i  \in \{\,1, \dots, n \,\}$,
      the induced curves $\b_i$ and $\d_i$ on $R_+(\g')$  lie in the punctured torus~$F_i$,
\item \label{it:b4} for $i \in \{\, 1,\dots, n \,\}$, the curve $\b_i$ represents a meridian of $K_i$ that is disjoint from all the $\d_j$ for $j \neq i$,
      and meets $\d_i$ in a single transverse intersection point,
\item \label{it:b5} for $i  \in \{\, 1,\dots, n \,\}$, the homology class of $\d_i$ corresponds to the framing $\ell_i$.
\end{enumerate}
\end{defn}

\begin{figure}[tb]
\includegraphics{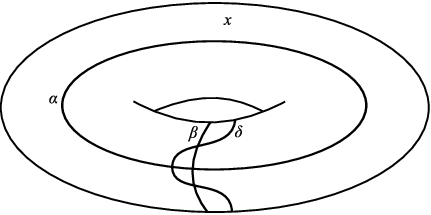}
\caption{A proper stabilization of a triple diagram is obtained by
taking the connected sum with the above diagram at the point marked by $x$.} \label{fig:2}
\end{figure}

\begin{defn} \label{def:stab}
By a \emph{stabilization} of a triple diagram $(\S,\bolda,\boldb,\boldd)$ subordinate to some bouquet, we mean the following.
Take the connected sum of $(\S,\bolda,\boldb,\boldd)$ with a diagram $(E,\a_{d+1},\b_{d+1},\d_{d+1})$, where $E \subset M'$ is a genus one surface,
$|\a_{d+1} \cap \b_{d+1}| = 1$, and $\d_{d+1}$ is a small isotopic translate of $\b_{d+1}$ such that $|\b_{d+1} \cap \d_{d+1}| = 2$, see Figure~\ref{fig:2}. Furthermore, we say that a stabilization is \emph{proper} if the connected sum tube joins a component of $\S \setminus (\bolda \cup \boldb \cup \boldd)$ that intersects $\partial \S$ nontrivially with the component of
$E \setminus (\a_{d+1} \cup \b_{d+1} \cup \d_{d+1})$ disjoint from the isotopy connecting $\b_{d+1}$ and $\d_{d+1}$.
A \emph{(proper) destabilization} is the reverse of a (proper) stabilization.
\end{defn}

The following generalizes the corresponding result of Ozsv\'ath and Szab\'o~\cite[Lemma 4.5]{OSz10}.

\begin{lem} \label{lem:bouquet}
Let $(M,\g)$ be a balanced sutured manifold, together with a framed link $\L \subset M$ and associated
bouquet $B(\L)$. Then there is a sutured triple diagram subordinate to $B(\L)$, and any two such triple
diagrams can be connected by a sequence of the following moves:
\begin{enumerate}
\item \label{it:aisot} isotopies and handleslides amongst $\{\, \a_1,\dots,\a_d \,\}$,
\item \label{it:bisot} isotopies and handleslides amongst $\{\, \b_{n+1},\dots,\b_d \,\}$,
      while carrying along the curves $\d_{n+1},\dots,\d_d$, as well,
\item \label{it:stab} proper stabilizations and destabilizations,
\item \label{it:ahandle} for $i \in \{\, 1, \dots, n \,\}$, an isotopy of $\b_i$, or a handleslide of $\b_i$ across a $\b_j$ with $j \in \{\, n+1, \dots ,d \,\}$,
\item \label{it:bhandle} for $i \in \{\, 1,\dots, n \,\}$, an isotopy of $\d_i$, or a handleslide of $\d_i$ across a $\d_j$ with $j \in \{\, n+1, \dots, d \,\}$
\item \label{it:diff} a diffeomorphism isotopic to the identity in $M'$.
\end{enumerate}
\end{lem}

\begin{proof}
By the work of the author~\cite[Proposition 2.13]{sutured}, there exists a sutured diagram
\[
(\S,\{\, \a_1,\dots\a_d \,\},\{\, \b_{n+1},\dots,\b_d \,\})
\]
defining the sutured manifold $(M',\g') = (M \setminus N(B(\L)),\g)$.
Using \cite[Proposition~2.37]{naturality}, any two sutured diagrams defining $(M',\g')$
can be connected by isotopies and handleslides of the $\a$- and $\b$-curves, diffeomorphisms
of the Heegaard surface isotopic to the identity in~$M'$, stabilizations, and destabilizations.
To see that proper stabilizations suffice (i.e., stabilizations in a region intersecting
$\partial \S$), note that we can obtain an arbitrary stabilization by
isotoping the $\a$- and $\b$-curves via a finger move along an arc~$a$ connecting~$\partial \S$
with the stabilization point, performing a proper stabilization,
followed by a sequence of handleslides along the inverse of~$a$.
The curves $\d_{n+1},\dots,\d_d$ are chosen to be small translates of $\b_{n+1},\dots,\b_d$,
respectively. It follows from the above discussion that any two such triple diagrams
\[
(\S, \{\, \a_1,\dots\a_d \,\},\{\, \b_{n+1},\dots,\b_d \,\}, \{\, \d_{n+1}, \dots, \d_d \,\})
\]
can be connected by moves~\eqref{it:aisot}, \eqref{it:bisot}, \eqref{it:stab}, and~\eqref{it:diff}.

Since $\S$ surgered along $\b_{n+1},\dots,\b_d$ is canonically diffeomorphic to $R_+(\g')$,
parts~\eqref{it:b3}--\eqref{it:b5} of Definition~\ref{defn:bouquet} prescribe how to choose $\b_1,\dots,\b_n$ and $\d_1,\dots,\d_n$.
For $i \in \{\, 1, \dots,  n \,\}$, the framed link specifies the homology classes of~$\b_i$ and~$\d_i$ in~$F_i$. Different choices~$\b_i$
and~$\b_i'$ can be connected by an isotopy in~$F_i$. It follows that in~$\S$ they can be connected by a sequence of
isotopies and handleslides across the curves $\b_{n+1},\dots,\b_d$. The same argument works for $\d_1,\dots,\d_n$.
These give rise to moves~\eqref{it:b4} and~\eqref{it:b5}.
\end{proof}

The following proposition is a generalization of Ozsv\'ath and Szab\'o~\cite[Proposition~4.3]{OSz10}.

\begin{prop} \label{prop:filling}
Let $(\S,\bolda,\boldb,\boldd)$ be a triple diagram subordinate to the bouquet~$B(\L)$ in~$(M,\g)$.
\begin{enumerate}
\item \label{it:f1} $\W_{\a,\b,\d}$ is a cobordism from $(M,\g)$ to the disjoint union of $(M(\L),\g)$
and
\[
(M_{\b,\d},\g_{\b,\d}) \approx (R_+ \times I, \partial R_+ \times I)  \#^{d-n} (S^2 \times S^1),
\]
where $R_+ = R_+(\g)$, and different copies of $S^2 \times S^1$ might be summed along different components of $R_+ \times I$.

\item \label{it:f2} The Heegaard surface $\S$ divides $(M,\g)$ into the sutured compression bodies~$U_\a$
and~$U_\b$, such that the framed link $\L \subset U_\b$. The sutured mono-diagram $(\S,\boldb)$ defines
a cobordism $\W_\b$ from $U_\b \cup_\S (-U_\b)$ to $\emptyset$. Let $\W_\b(\L)$ be the cobordism
obtained from $\W_\b$ by attaching $4$-dimensional $2$-handles to $\L$ along $U_\b$.
Then $\W_\b(\L)$ is a cobordism from $(M_{\b,\d}, \g_{\b,\d})$ to $\emptyset$.
Then we can glue $\W_{\a,\b,\d}$ and~$\W_{\b}(\L)$ along $(M_{\b,\d},\g_{\b,\d})$ such that we obtain
a cobordism equivalent to~$W(\L)$, and this equivalence is well-defined up to isotopy.
\end{enumerate}
\end{prop}

\begin{proof}
It follows from part~\eqref{it:b4} of Definition~\ref{defn:bouquet} that $(\S,\bolda,\boldb)$ is a diagram of
the sutured manifold~$(M,\g)$. Furthermore, by part~\eqref{it:b5},
the diagram $(\S,\bolda,\boldd)$ defines the surgered manifold $(M(\L),\g)$,
since we glue $3$-dimensional $2$-handles to the complement $(M',\g')$ of the bouquet $B(\L)$ along curves specified by the framing of $\L$.
Finally, let $\S_{\b}$ be $\S$ surgered along $\boldb$. Then the diagram
$(\S,\boldb,\boldd)$ defines
\[
\left(\S_{\b} \times I, \partial \S_{\b} \times I \right) \#^n S^3  \#^{d-n} \left(S^2 \times S^1 \right),
\]
where $\b_i$ and $\d_i$ give the $S^3$ components for $i \in \{\, 1, \dots, n \,\}$, and the $S^2 \times S^1$ components for
$i \in \{\, n+1, \dots, d \,\}$. Note that different copies of $S^3$ or $S^2 \times S^1$ might be added to different components of $\S_{\b} \times I$.
However, $\S_{\b} \approx R_+(\g)$, which concludes the proof of~\eqref{it:f1}.

We now prove~\eqref{it:f2}; i.e., that we can glue $\W_{\b}(\L)$ to~$\W_{\a,\b,\d}$ such that we obtain~$W(\L)$.
Let $\W_{\a,\b,\d} = (W,Z,[\xi])$ and
$\W_\b(\L) = (W',Z',[\xi'])$. As usual, $P_3$ denotes a regular triangle, and we label its edges $e_{\a}$, $e_{\b}$, and~$e_{\d}$
in a clockwise fashion. Furthermore, $U_{\a}$, $U_{\b}$, and~$U_{\d}$ are the sutured compression bodies corresponding to
$\bolda$, $\boldb$, and $\boldd$, respectively. Recall that~$W$ is obtained from~$P_3 \times \S$ by gluing
$e_{\a} \times U_{\a}$, $e_{\b} \times U_{\b}$, and $e_{\d} \times U_{\d}$ along
$e_{\a} \times \S$, $e_{\b} \times \S$, and $e_{\d} \times \S$, respectively, and smoothing corners.

Let $P_1$ be a monogon with edge $e$. Then $W_\b$ is obtained from $\S \times P_1$ by gluing
$e \times U_\b$ along $e \times \S$, and smoothing corners. The only end of this cobordism
is obtained by gluing two copies of~$U_\b$, namely the components of $\partial e \times U_\b$,
along~$\S$. Then we attach $4$-dimensional $2$-handles along $\L$ in one of the $U_\b$
components to obtain $W'$. An alternative way to describe $W'$ is the following.
Take $U_\b \times I$, and attach $4$-dimensional $2$-handles to $U_\b \times \{1\}$ along $\L \times \{1\}$,
obtaining the $4$-manifold $(U_\b \times I)(\L)$.
The top boundary $U_\b \times \{1\}$ becomes $U_\b(\L) = U_\d$.
After smoothing corners,
\[
(U_\b \times \{0\}) \cup (\S \times I) \cup U_\d
\]
becomes diffeomorphic to $M_{\b,\d}$.

\begin{figure}
\includegraphics{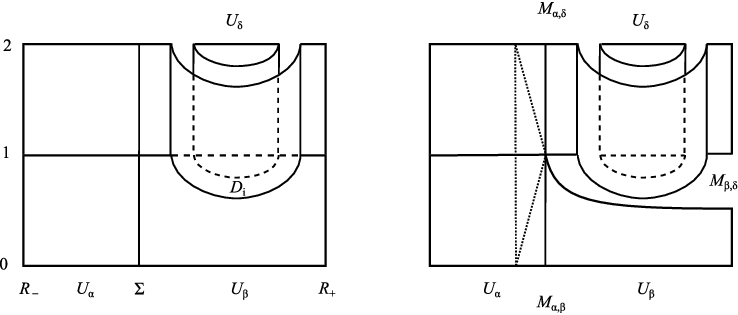}
\caption{The cobordism $\W(\L)'$ on the left, and $\W(\L)''$ on the right.} \label{fig:2.5}
\end{figure}

Recall that the cobordism $W(\L)$ is obtained from $M \times I$ by attaching $4$-di\-men\-sional
$2$-handles $D_1, \dots, D_n$ to $M \times \{1\}$ along the components of the framed link
$\L \times \{1\}$. Note that
\[
W(\L) = (U_\a \times I) \cup_{\S \times I} (U_\b \times I)(\L).
\]
There is an isotopically unique equivalence from $W(\L)$ to the cobordism $\W(\L)'$ obtained by gluing
$W(\L)$ and the trivial cobordism $M(\L) \times [1,2]$ along $M(\L) \times \{1\}$, see the left-hand
side of Figure~\ref{fig:2.5}.
Let $N(\L)$ be the tubular neighborhood of~$\L$ in~$U_\b$ used for attaching
$D_1, \dots, D_n$, and denote by~$E(\L)$ the link exterior $U_\b \setminus \text{Int}(N(\L))$.
If we remove the interiors of $D_1,\dots,D_n$ from $W(\L)'$, then cut the resulting manifold along
$E(\L) \times \{1\}$, and finally smooth the resulting corner at $\S \times \{1\}$, we obtain
a cobordism $\W(\L)''$ that is isotopically uniquely equivalent to $\W_{\a,\b,\d}$.
Indeed, $\W(\L)''$ is obtained by taking $U_\a \times [0,2]$, and gluing $U_\b \times [0,1]$
along $\S \times [0,1]$ and $U_\d \times [1,2]$ along $\S \times [1,2]$. For an illustration
of $\W(\L)''$, see the right-hand side of Figure~\ref{fig:2.5}. The dotted triangle indicates
the identification with $\W_{\a,\b,\d}$.

We also obtain $\W(\L)''$ from $\W(\L)'$ if we remove the interior of
\[
\left( U_\b \times [1-\varepsilon, 1] \right) \cup D_1 \cup \dots \cup D_n,
\]
and the latter manifold is diffeomorphic to $(U_\b \times I)(\L)$, which we have shown
to agree with $W'$. So indeed, we can glue $W'$ to $W$ along $M_{\b,\d}$ such that
we obtain a manifold diffeomorphic to $W(\mathbb{L})$.

\begin{figure}[tb]
\includegraphics{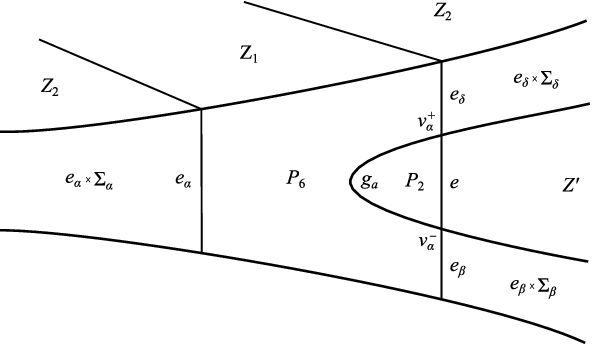}
\caption{A section of the manifold $Z \cup Z'$.} \label{fig:3}
\end{figure}

Next, we check that the diffeomorphism from $W \cup W'$ to $W(\L)$ constructed above maps
$(Z \cup Z',[\xi \cup \xi'])$ to $(\partial M \times I,[\z])$, where~$\z$ is an $I$-invariant contact structure such
that every $\partial M \times \{t\}$ is a convex surface with dividing set~$\g$.
This will conclude the proof of $\W_{\a,\b,\d} \cup \W_\b(\L) = W(\L)$.

Recall that $Z = Z_1 \cup Z_2$, where $Z_1 = P_6 \times \partial \S$ and
\[
Z_2 = (e_{\a} \times \S_{\a}) \cup (e_{\b} \times \S_{\b}) \cup (e_{\d} \times \S_{\d}),
\]
see Figure \ref{fig:3}.
The contact structure~$\xi$ is given by a hierarchy that starts with decomposing along a set of product annuli~$A \subset Z_2$
parallel to
\[
(e_{\a} \times \partial\S_{\a}) \cup (e_{\b} \times \partial \S_{\b}) \cup (e_{\d} \times \partial \S_{\d}),
\]
then along surfaces $P_6 \times \{q\}$, for one~$q$ in each component of~$Z_1$.

Similarly, $Z' = Z_1' \cup Z_2'$, where $Z_1' = P_2 \times \partial \S$ and $Z_2' = e \times \S_\b$.
Furthermore, $\xi'$ is defined by decomposing the sutured manifold $(Z',\g_{\b,\d})$
along product annuli~$A'$ parallel to $e \times \partial R_+$, after which we get $Z_1'$, a union
of tori with two longitudinal sutures on each, and $Z_2'$, a product sutured manifold.
Hence $Z'$ is the manifold $\S_\b \times I$ with~$\xi'$ being the canonical $I$-invariant contact structure.
In particular, $(Z',\g_{\b,\d})$ is diffeomorphic to the product sutured manifold $(\S_\b \times I, \partial \S_\b \times I)$.

Let $g_{\a}$ be the edge of $P_6$ lying between $e_{\b}$ and $e_{\d}$,
and let $e_{\b} \cap g_{\a} = v_{\a}^-$ and $e_{\d} \cap g_{\a} = v_{\a}^+$.
The surface $\S_{\d}$ is naturally identified with $\S_\b$, and $Z'$ is glued to $Z$ along
\[
(\{v_{\a}^-\} \times \S_{\b}) \cup (\{v_{\a}^+\}\times \S_{\d}) \cup (g_{\a} \times \partial \S)
\]
using this identification. More precisely, $R_-(Z',\g_{\b,\d}) = \S_\b \times \{0\}$ is glued to $\{v_{\a}^-\} \times \S_{\b}$ and
$R_+(Z',\g_{\b,\d}) = \S_\b \times \{1\}$ is glued to $\{v_{\a}^+\}\times \S_{\d}$,
whereas $\g_{\b,\d}$ is glued to $g_{\a} \times \partial \S$. It follows that $Z \cup Z'$ is
diffeomorphic to $(\S_{\a} \cup \g_{\a,\b} \cup \S_{\b}) \times I = \partial M \times I$.

The dividing set of $\xi \cup \xi'$ on $\partial M \times \{i\}$ is $s(\g_{\a,\b}) \times \{i\}$ for both $i=0$ and $i=1$.
The product annuli $A$ and $A'$ for $(Z,\g)$ and $(Z',\g_{\b,\d})$ glue up to a set of product annuli $A \cup A'$
inside $\partial M \times I$. After decomposing
$\partial M \times I$ along $A \cup A'$, we get a product sutured manifold diffeomorphic to
\[
(\S_{\a} \times I,\partial \S_{\a} \times I) \cup (\g_{\a,\b} \times I,\partial \g_{\a,\b} \times I) \cup (\S_{\b} \times I,\partial \S_{\b} \times I).
\]
The decomposing surfaces $P_6 \times \{q\}$ and $P_2 \times \{q\}$ glue together to give product disks
inside $(\g_{\a,\b} \times I,\partial \g_{\a,\b} \times I)$. Hence $\xi \cup \xi'$ is
given by a hierarchy which starts with the product annuli $A \cup A'$,
and continues with decompositions along product disks, and is consequently equivalent to an $I$-invariant contact structure.
Since the dividing set on $\partial M \times \{0\}$ is $s(\g_{\a,\b}) \times \{0\}$, we have $\xi \cup \xi' = \z$.

Alternatively, using the description of $\xi$ in the proof of Lemma~\ref{lem:bcond},
we see that the $2$-pane field~$\xi \cup \xi'$ is never opposite to the $2$-plane field
that is tangent to the product foliation on $\S_{\a} \times I$ and $\S_{\b} \times I$,
and rotates $\pi$ as we traverse $\g_{\a,\b} \times \{t\}$ from~$\S_{\a}$ to~$\S_{\b}$.
Hence $\xi \cup \xi'$ is equivalent to a contact structure
such that $\partial M \times \{t\}$ is convex for every $t \in I$ with dividing set $s(\g_{\a,\b}) \times \{t\}$.
\end{proof}

Using part~\eqref{it:f2} of Proposition~\ref{prop:filling}, we get a restriction map
\[
r \colon \spinc(W(\L)) \to \spincu(\W_{\a,\b,\d}).
\]
To define $r(\s)$ for $\s \in \spinc(W(\L))$,
first choose an $I$-invariant contact structure~$\eta$ on $\partial M \times I \subset W(\L)$
such that each $\partial M \times \{t\}$ is a convex surface with dividing set~$\g \times \{t\}$.
Then pick an almost complex structure $J'$ on $T(M \times I)|_{\partial M \times I}$ such that
the field of complex tangencies $\xi_{J'} = \eta$. Since $W(\L)$ is a special cobordism, the map
\[
q \colon \spinc(W(\L),J') \to \spinc(W(\L))
\]
is an affine isomorphism. We take $q^{-1}(\s) \in \spinc(W(\L),J')$, and restrict it
to the complement of
\[
\left( U_\b \times [1-\varepsilon, 1] \right) \cup D_1 \cup \dots \cup D_n,
\]
which we identified with $\W_{\a,\b,\d}$. The result in $\spincu(\W_{\a,\b,\d})$ is independent of the choices made.
This construction will enable us to define maps on $\SFH$ induced by cobordisms equipped with $\spinc$ structures.

By the connected sum formula~\cite[Proposition 9.15]{sutured}, we have
\[
\SFH\left((R_+ \times I, \partial R_+ \times I) \# \left(\#^{d-n} (S^2 \times S^1) \right)\right) \cong
\Lambda^*H^1\left(\#^{d-n}(S^1 \times S^2)\right).
\]
This is supported in the $\spinc$ structure~$\s_0$ that is characterized by
\begin{enumerate}
\item \label{it:horizontal} the restriction of $\s_0$ to $(R_+ \times I, \partial R_+ \times I)$
    is homologous to the $2$-plane field tangent to the horizontal foliation,
\item \label{it:disk} the restriction of $\s_0$ to $\#^{d-n}(S^1 \times S^2)$
    extends to $\sharp^{d-n}(S^1 \times D^3)$, or equivalently, its first Chern class vanishes.
\end{enumerate}

We introduce the shorthand
\[
M(R_+,d-n) = (R_+ \times I, \partial R_+ \times I) \# \left(\#^{d-n} (S^2 \times S^1)\right).
\]
Then $\SFH(M(R_+,d-n))$ with the relative Maslov grading is isomorphic to
\[
\Lambda^* H^1\left(\#^{d-n} (S^1 \times S^2); \Z_2\right).
\]
So the ``top-dimensional'' part of $\SFH(M(R_+,d-n))$ is $\Z_2$, whose generator we denote by~$\t$.

\begin{lem} \label{lem:theta}
Let $(\S,\bolda,\boldb,\boldd)$ be a triple diagram subordinate to the bouquet~$B(\L)$ in~$(M,\g)$.
Furthermore, let $\s \in \spinc(W(\L))$ be an arbitrary $\spinc$ structure. Then, using the identification
between $(M_{\b,\d}, \g_{\b,\d})$ and $M(R_+,d-n)$ in part~\eqref{it:f1} of Proposition~\ref{prop:filling}, we have
\[
r(\s)|_{(M_{\b,\d},\g_{\b,\d})} = r(\s)|_{M(R_+,d-n)} = \s_0.
\]
\end{lem}

\begin{proof}
To see that $r(\s)|_{M(R_+,d-n)}$ satisfies condition~\eqref{it:horizontal} characterizing~$\s_0$ above,
notice that the $(R_+ \times I,\partial R_+ \times I)$ component of~$M(R_+,d-n)$
is parallel to~$(Z',\g_{\b,\d}) \subset \partial M \times I$ in the proof of Proposition~\ref{prop:filling},
which carries the ``horizontal" $\spinc$ structure.

Now we check property~\eqref{it:disk}. Recall from the proof of Proposition~\ref{prop:filling}
that $M_{\b,\d}$ is embedded in $U_\b(\L)$ as
\[
(U_\b \times \{0\}) \cup (\S \times I) \cup U_\d.
\]
Observe that $r(\s)$ extends to $U_\b(\L)$, as we obtain $W(\L)$ from $\W_{\a,\b,\d}$
by gluing~$U_\b(\L)$ to $(M_{\b,\d},\g_{\b,\d})$ along the above subset of its boundary.
Let $C_i$ be the core of the $3$-dimensional
$2$-handle in $U_\b$ glued to $\S \times I$ along $\b_i \times \{0\}$, together with the annulus~$\b_i \times I$.
Then $C_i \times I$ is a 3-ball in $U_\b(\L)$ with boundary a $2$-sphere that is isotopic to $\{pt\} \times S^2$ in the
$i$-th $S^1 \times S^2$ component of $(M_{\b,\d}, \g_{\b,\d}) \approx M(R_+,d-n)$. Since $r(\s)$ extends
to $C_i \times I$, we see that $c_1(\s_0)$ vanishes on this~$S^1 \times S^2$ component.
\end{proof}

\begin{defn}
Let $\L$ be a framed link in $(M,\g)$, and fix a $\spinc$ structure $\s \in \spinc(W(\L))$. Then we define maps
\[
F_{M,\L}  \colon \SFH(M,\g) \to \SFH(M(\L),\g) \text{ and}
\]
\[
F_{M,\L,\s} \colon \SFH\left(M,\g,\s|_{M} \right) \to \SFH \left(M(\L),\g,\s|_{M(\L)}\right),
\]
as follows.

Pick a bouquet $B(\L)$ for $\L$, and an admissible triple diagram $\mathcal{T} = (\S,\bolda,\boldb,\boldd)$ subordinate to this
bouquet. As above, $\t$ is the generator of the top-dimensional part of
\[
\SFH(\S,\boldb,\boldd,\s_0) \cong \SFH(M(R_+,d-n),\s_0).
\]
Let $\H = (\S,\bolda,\boldb)$, which is a diagram of $(M,\g)$, and $\H_\L = (\S,\bolda,\boldd)$, which
is a diagram of $(M(\L),\g)$.
According to Definition~\ref{defn:products}, we have a triangle map
\[
F_{\a,\b,\d} \colon \SFH(\H) \otimes \SFH(\S,\boldb,\boldd) \to \SFH(\H_\L).
\]
We define the map $E_{\mathcal{T}} \colon \SFH(\H) \to \SFH(\H_\L)$
via $E_{\mathcal{T}}(x) = F_{\a,\b,\d}(x \otimes \t)$. By Lemma~\ref{lem:theta}, this can
be refined using $\spinc(W(\L))$, and we define
\[
E_{\mathcal{T},\s} \colon \SFH \left(\H,\s|_M \right) \to \SFH \left(\H_\L,\s|_{M(\L)} \right)
\]
via $E_{\mathcal{T},\s}(x) = F_{\a,\b,\d}(x \otimes \t, r(\s))$.
Then
\[
F_{M, \L} = P_{\H_\L}^{-1} \circ E_{\mathcal{T}} \circ P_{\H},
\]
where $P_{\H} \colon \SFH(M,\g) \to \SFH(\H)$ and $P_{\H_\L} \colon \SFH(M(\L),\g) \to \SFH(\H_\L)$ are
the natural isomorphisms introduced in Section~\ref{sec:naturality}.
Similarly, we let
\[
F_{M, \L, \s} = P_{\H_\L}^{-1} \circ E_{\mathcal{T},\s} \circ P_{\H}.
\]
\end{defn}

The following theorem ensures that the above definition is independent of the choice of bouquet and subordinate triple diagram.

\begin{thm}
Let $(M,\g)$ be a balanced sutured manifold equipped with a framed link $\L$,
and let $\s \in \spinc(W(\L))$.
Suppose $\mathcal{T} = (\S,\bolda,\boldb,\boldd)$ and
$\mathcal{T}' = (\S',\bolda',\boldb',\boldd')$ are admissible triple diagrams
subordinate to bouquets~$B$ and~$B'$ for~$\L$, respectively, and let
$\H = (\S,\bolda,\boldb)$, $\H_\L = (\S,\bolda,\boldd)$, $\H' = (\S',\bolda',\boldb')$,
and $\H_\L' = (\S',\bolda',\boldd')$.
Then we have a commutative diagram
\begin{equation} \label{eqn:bouquet}
\begin{CD}
\SFH\left(\H,\s|_M\right) @>E_{\mathcal{T},\s}>> \SFH\left(\H_\L,\s|_{M(\L)}\right)\\
@VVF_{\H,\H'}V   @VVF_{\H_\L,\H_\L'}V\\
\SFH\left(\H',\s|_M \right) @>E_{\mathcal{T}',\s}>> \SFH \left(\H_\L',\s|_{M(\L)} \right),
\end{CD}
\end{equation}
where $F_{\H,\H'}$ and $F_{\H_\L,\H_\L'}$
are the canonical isomorphisms defined in Section~\ref{sec:naturality}.
An analogous statement holds for $E_{\mathcal{T}}$ and $E_{\mathcal{T}'}$.
\end{thm}

\begin{proof}
We follow the proof of \cite[Theorem 4.4]{OSz10}.
The most important new point is that for naturality reasons
we also consider the embeddings of~$\S$ and~$\S'$ in~$M$.

First, assume that $B = B'$. Then~$\mathcal{T}$ and~$\mathcal{T}'$ can be connected by
a sequence of moves~\eqref{it:aisot}--\eqref{it:diff} of Lemma~\ref{lem:bouquet}.
It suffices to check that diagram~\eqref{eqn:bouquet} is commutative when the two triples differ by exactly one of these moves.
Indeed, the general case follows by writing down a ladder where each small rectangle corresponds to an elementary
move and is hence commutative, while the composition of the maps along the two vertical sides give
the canonical isomorphisms $F_{\H,\H'}$
and $F_{\H_\L,\H_\L'}$ as the composition of any sequence of canonical
isomorphisms is again a canonical isomorphism.

First, we check invariance under move~\eqref{it:stab}, a proper stabilization, cf.~Definition~\ref{def:stab}.
Just as in~\cite[Lemma~4.7]{OSz10}, we obtain that the maps~$E_{\mathcal{T},\s}$
commute with the isomorphisms induced by proper stabilizations.
The argument is somewhat simpler in our case, since we are stabilizing near~$\partial \S$
and holomorphic discs avoid the boundary, which make neck-stretching unnecessary.
Indeed, we obtain~$\mathcal{T}'$ by taking the connected sum of~$\mathcal{T}$ in a region
that intersects~$\partial\S$ with the diagram~$(E,\a_{d+1},\b_{d+1},\d_{d+1})$ at the point $x \in E$,
see Figure~\ref{fig:2}. Let $\a_{d+1} \cap \b_{d+1} = \{x_{d+1}\}$, $\a_{d+1} \cap \d_{d+1} = \{y_{d+1}\}$,
and let~$w_{d+1}$ be the point of~$\b_{d+1} \cap \d_{d+1}$ of higher relative grading.
If~$\theta \in \T_\b \cap \T_\d$ represents the top-dimensional generator $\Theta \in \SFH(\S,\boldb,\boldd,\s_0)$,
then $\theta \times \{w_{d+1}\}$ represents the top-dimensional generator~$\Theta' \in \SFH(\S,\boldb',\boldd',\s_0)$.
For~$\x \in \T_\a \cap \T_\b$ with~$\s(\x) = r(\s)$, the coefficient~$n_\y$ of~$\y \in \T_\a \cap \T_\d$
in the chain $f_{\mathcal{T}}(\x \otimes \theta,r(\s))$ is obtained by
counting rigid pseudo-holomorphic triangles with corners at~$\x$, $\theta$, and~$\y$ and inducing the $\spinc$
structure~$\s$. The stabilization map~$F_{\H_\L,\H_\L'}$
on the chain level is given by~$\y \mapsto \y \otimes \{y_{d+1}\}$.
On the other hand, the stabilization map~$F_{\H,\H'}$ on the chain level is given by~$\x \mapsto \x \times \{x_{d+1}\}$.
The coefficient of~$\y \otimes \{y_{d+1}\}$ in the chain
\[
f_{\mathcal{T}'}((\x \times \{x_{d+1}\}) \otimes (\theta \times \{w_{d+1}\}), r(\s))
\]
is also~$n_\y$, as in~$E$ there is a unique rigid holomprphic triangle with corners~$x_{d+1}$, $y_{d+1}$, and $w_{d+1}$
having multiplicity zero at the connected sum point~$x \in E$, giving a bijection between
pseudo-holomorphic representatives of each homotopy class in~$\pi_2(\x,\y,\theta)$
and of the corresponding class in
\[
\pi_2(\x \times \{x_{d+1}\},\y \times \{y_{d+1}\},\theta \times \{w_{d+1}\}).
\]
This shows that diagram~\ref{eqn:bouquet} is commutative already on the chain level for proper stabilizations.

To check invariance under move~\eqref{it:diff},
suppose that the two triple diagrams~$\mathcal{T}$ and~$\mathcal{T}'$ differ by a diffeomorphism isotopic to the identity in~$M'$.
Then, in particular, there is a diffeomorphism $\phi \colon \S \to \S'$ such that $\phi(\bolda) = \bolda'$,
$\phi(\boldb) = \boldb'$, and~$\phi(\boldd) = \boldd'$.
Note that $\phi_*(\Theta_{\b,\d}) = \Theta_{\b',\d'}$, as~$\phi_*$ is an isomorphism
mapping the top-dimensional part of $\SFH(\S,\boldb,\boldd,\s_0)$, generated by~$\Theta_{\b,\d}$,
to the top-dimensional part of $\SFH(\S',\boldb',\boldd',\s_0)$, generated by~$\Theta_{\b',\d'}$.
Hence the diagram
\[
\begin{CD}
\SFH\left(\H,\s|_M \right) @>F_{\mathcal{T}}(\,\cdot\, \otimes\Theta_{\b,\d},r(\s))>> \SFH\left(\H_\L, \s|_{M(\L)} \right)\\
@VV\phi_*V   @VV\phi_*V\\
\SFH\left(\H',\s|_M \right) @>F_{\mathcal{T}'}(\,\cdot\, \otimes\Theta_{\b',\d'},r(\s))>> \SFH\left(\H_\L',\s|_{M(\L)} \right)
\end{CD}
\]
commutes by Lemma~\ref{lem:n-commute}.

Invariance under moves~\eqref{it:aisot}, \eqref{it:bisot}, \eqref{it:ahandle}, and ~\eqref{it:bhandle}
follows in a way similar to the proof of~\cite[Proposition~4.6]{OSz10}, but it is slightly
simpler as our naturality isomorphisms do not involve the continuation maps~$\Gamma$, cf.~Section~\ref{sec:naturality}.
More concretely, suppose that~$(\S,\bolda,\boldb,\boldd)$ and~$(\S,\bolda',\boldb',\boldd')$ are related by one
of the moves above. Then, by~\cite[Lemma~9.3]{naturality}, we can take isotopic copies~$\ol{\bolda}$, $\ol{\boldb}$, and~$\ol{\boldd}$
of~$\bolda$, $\boldb$, and~$\boldd$, respectively, such that the multi-diagrams
\[
(\S,\bolda,\boldb,\boldd,\ol{\bolda},\ol{\boldb},\ol{\boldd}) \text{ and } (\S,\ol{\bolda},\ol{\boldb},\ol{\boldd},\bolda',\boldb',\boldd')
\]
are both admissible.
Let~$\ol{\mathcal{T}} = (\S,\ol{\bolda},\ol{\boldb},\ol{\boldd})$,
$\ol{\H} = (\S,\ol{\bolda},\ol{\boldb})$, and~$\ol{\H}_\L = (\S,\ol{\bolda},\ol{\boldd})$.
Then we claim that the following diagram is commutative:
\[
\begin{CD}
\SFH\left(\H,\s|_M \right) @>F_{\mathcal{T}}(\,\cdot\, \otimes\Theta_{\b,\d},r(\s))>> \SFH\left(\H_\L,\s|_{M(\L)} \right)\\
@VV\Psi^{\a \to \ol{\a}}_{\b \to \ol{\b}} V  @VV\Psi^{\a \to \ol{\a}}_{\d \to \ol{\d}} V\\
\SFH\left(\ol{\H},\s|_M \right) @>F_{\ol{\mathcal{T}}}(\,\cdot\, \otimes\Theta_{\ol{\b},\ol{\d}},r(\s))>> \SFH\left(\ol{\H}_\L, \s|_{M(\L)} \right).
\end{CD}
\]
Indeed, for~$x \in \SFH\left(\H,\s \right)$, suppressing $\spinc$ structures in our notation, we have
\[
\begin{split}
\Psi^{\a \to \ol{\a}}_{\d \to \ol{\d}}\left(F_{\mathcal{T}}(x \otimes \Theta_{\b,\d})\right) =
F_{\ol{\a},\a,\ol{\d}}\left(\Theta_{\ol{\a},\a} \otimes
F_{\a,\d,\ol{\d}}\left(F_{\mathcal{T}}(x \otimes \Theta_{\b,\d}) \otimes \Theta_{\d,\ol{\d}}\right) \right) = \\
F_{\ol{\a},\a,\ol{\d}}\left(\Theta_{\ol{\a},\a} \otimes
F_{\a,\b,\ol{\d}}\left(x \otimes F_{\b,\d,\ol{\d}}\left(\Theta_{\b,\d} \otimes \Theta_{\d,\ol{\d}}\right)\right)\right) = \\
F_{\ol{\a},\a,\ol{\d}}\left(\Theta_{\ol{\a},\a} \otimes
F_{\a,\b,\ol{\d}}\left(x \otimes \Theta_{\b,\ol{\d}}\right)\right) =
F_{\ol{\a},\b,\ol{\d}} \left(F_{\ol{\a},\a,\b}(\Theta_{\ol{\a},\a} \otimes x) \otimes
\Theta_{\b,\ol{\d}}\right) =\\
F_{\ol{\a},\b,\ol{\d}} \left(F_{\ol{\a},\a,\b}(\Theta_{\ol{\a},\a} \otimes x) \otimes F_{\b,\ol{\b},\ol{\d}}
\left(\Theta_{\b,\ol{\b}} \otimes \Theta_{\ol{\b},\ol{\d}}\right)\right) = \\
F_{\ol{\mathcal{T}}}\left(F_{\ol{\a},\b,\ol{\b}}
\left(F_{\ol{\a},\a,\b}(\Theta_{\ol{\a},\a} \otimes x) \otimes \Theta_{\b,\ol{\b}}\right) \otimes \Theta_{\ol{\b},\ol{\d}} \right) =
F_{\ol{\mathcal{T}}}\left(\Psi^{\a \to \ol{\a}}_{\b \to \ol{\b}}(x) \otimes \Theta_{\ol{\b},\ol{\d}}\right),
\end{split}
\]
where we used Theorem~\ref{thm:assoc} -- the associativity of the triangle maps -- in the second, fourth, and sixth step.
The third step follows from the observation that
\[
F_{\b,\d,\ol{\d}}\left(\Theta_{\b,\d} \otimes \Theta_{\d,\ol{\d}}\right) =
\Psi^\b_{\d \to \ol{\d}}(\Theta_{\d,\ol{\d}}) = \Theta_{\b,\ol{\d}},
\]
which holds as $\Psi^\b_{\d \to \ol{\d}} \,\colon\, \SFH_{\text{top}}(\S,\b,\d,\s_0) \to \SFH_{\text{top}}(\S,\b,\ol{\d},\s_0)$ is an isomorphism,
and so maps the generator~$\Theta_{\b,\d}$ to the generator~$\Theta_{\b,\ol{\d}}$. Similarly, we have
\[
F_{\b,\ol{\b},\ol{\d}}
\left(\Theta_{\b,\ol{\b}} \otimes \Theta_{\ol{\b},\ol{\d}}\right) = \Theta_{\b,\ol{\d}},
\]
which implies the fifth step. An analogous argument gives the commutativity of the following diagram:
\[
\begin{CD}
\SFH\left(\ol{\H},\s|_M \right) @>F_{\ol{\mathcal{T}}}(\,\cdot\, \otimes\Theta_{\ol{\b},\ol{\d}},r(\s))>> \SFH\left(\ol{\H}_\L,\s|_{M(\L)} \right)\\
@VV\Psi^{\ol{\a} \to \a'}_{\ol{\b} \to \b'} V  @VV\Psi^{\ol{\a} \to \a'}_{\ol{\d} \to \d'} V\\
\SFH\left(\H',\s|_M \right) @>F_{\mathcal{T}'}(\,\cdot\, \otimes\Theta_{\b',\d'},r(\s))>> \SFH\left(\H'_\L, \s|_{M(\L)} \right),
\end{CD}
\]
which, together with the previous commutative diagram, implies the commutativity of diagram~\eqref{eqn:bouquet}.
The claim in the case~$B = B'$ follows.

We now show that $F_{M,\L,\s}$ is independent of the bouquet, in the spirit of~\cite[Lemma~4.8]{OSz10}.
Suppose that~$B$ and~$B'$ are a pair of bouquets that differ in the choice of one arc~$a_1 \subset B$ and~$a_1' \subset B'$.
If~$a_1'$ is a small isotopic translate of~$a_1$, then we can use the same triple
diagram~$\mathcal{T} = \mathcal{T}'$ subordinate to both~$B$ and~$B'$, in which case diagram~\eqref{eqn:bouquet}
is obviously commutative. Hence, if~$(B \setminus \L) \cap a_1' \neq \emptyset$, we can perturb~$a_1'$ by a
small isotopy such that it becomes disjoint from~$B \setminus \L$.
We construct two triples $(\S,\bolda,\boldb,\boldd)$ and $(\S,\bolda,\boldb',\boldd')$ subordinate to~$B$
and~$B'$, respectively, such that~$\boldb'$
is obtained from~$\boldb$ by a sequence of isotopies and handleslides, and $\boldd'$ is obtained from $\boldd$ by a sequence
of isotopies and handleslides.

To this end, consider $(M'',\g'') = (M \setminus N(B \cup B'), \g)$. Note that
\[
\chi(R_-(\g'')) - \chi(R_+(\g'')) = 2(n+1).
\]
By~\cite[Proposition~2.13]{sutured}, there exists
a sutured diagram
\[
\left(\S, \{\a_1,\dots,\a_d\}, \{\b_{n+2},\dots, \b_d\} \right)
\]
defining $(M'',\g'')$. Let $\b_1, \dots, \b_n$
be meridians of the components $K_1, \dots, K_n$ of~$\L$, respectively. Furthermore, $\b_{n+1}$ is a meridian of~$a_1$.
Similarly, $\b_{n+1}'$ is a meridian of~$a_1'$, and we set $\b_i' = \b_i$ if $i \in \{\, 1,\dots, d\,\} \setminus \{n+1\}$.
The curves $\d_1,\dots,\d_n$ correspond to the framing of $\L$, and for $i \in \{\, n+1,\dots,d \,\}$,
the curve $\d_i$ is a small isotopic translate of~$\b_i$.
Finally, $\d_{n+1}'$ is a small isotopic translate of $\b_{n+1}'$, and
we set~$\d_i' = \d_i$ for $i \in \{\,1, \dots, d\,\} \setminus \{n+1\}$.
Then $(\S,\bolda,\boldb,\boldd)$ is subordinate to~$B$, while $(\S,\bolda,\boldb',\boldd')$ is subordinate to~$B'$.

If we surger $\S$ along $\b_1,\dots,\b_n,\b_{n+2},\dots,\b_d$, then we obtain
$(R_+(\g) \setminus (B_1 \cup B_2)) \cup A,$
where $B_1$, $B_2 \subset R_+(\g)$ are disjoint disks and $A$ is an annulus glued along $\partial B_1 \cup \partial B_2$.
Then $\b_{n+1}$ and $\b_{n+1}'$ induce two disjoint, embedded, homologically non-trivial curves lying in $A$.
These curves must then be isotopic in~$A$, thus $\b_{n+1}'$ can be obtained by handlesliding $\b_{n+1}$ over
some collection of the curves~$\b_1,\dots,\b_n,\b_{n+2},\dots,\b_d$.
We can obtain $\d_{n+1}'$ from $\d_{n+1}$ in an analogous manner. Consequently, $\boldb'$
is obtained from $\boldb$ by a sequence of isotopies and handleslides, and $\boldd'$ is obtained from $\boldd$ by a sequence
of isotopies and handleslides.

From here, commutativity of diagram~\eqref{eqn:bouquet} follow just like
invariance under moves~\eqref{it:aisot}, \eqref{it:bisot}, \eqref{it:ahandle}, and ~\eqref{it:bhandle}
in the case~$B = B'$, applied to the triples $(\S,\bolda,\boldb,\boldd)$ and~$(\S,\bolda,\boldb',\boldd')$.
As one can always get from a bouquet~$B$ to any other bouquet~$B'$ by a sequence of small isotopies and
changing one arc at a time,
the map~$F_{\L,\s}$ is independent of both the bouquet and the subordinate triple diagram.
\end{proof}

We have the following analogue of \cite[Proposition~4.9]{OSz10}.

\begin{prop} \label{prop:linkcomp}
Let $\L$ be a framed link in $(M,\g)$. Suppose that we are given a partition $\L_1 \cup \L_2$ of $\L$.
Then we have cobordisms
\[
W(\L_1) \colon (M,\g) \to (M(\L_1),\g),
\]
and if we view $\L_2$ as a link in $M(\L_1)$,
\[
W(\L_2) \colon (M(\L_1),\g) \to (M(\L),\g).
\]
Choose $\spinc$ structures $\s_i \in \spinc(W(\L_i))$ for $i \in \{1,2\}$.
Then there is an isotopically unique diffeomorphism between~$W(\L)$ and~$W(\L_2) \circ W(\L_1)$, and
under this identification,
\[
\sum_{\s \in \spinc(W(\L)) \colon \s|_{W(L_1)} = \s_1 \text{, } \s|_{W(\L_2)} = \s_2} F_{M,\L,\s} =
F_{M(\L_1),\L_2,\s_2} \circ F_{M,\L_1,\s_1}.
\]
Moreover, $F_{M,\L} = F_{M(\L_1),\L_2} \circ F_{M,\L_1}$.
\end{prop}

\begin{proof}
The proof is completely analogous to the proof of \cite[Proposition 4.9]{OSz10}.
Note that one can disregard the last paragraph
there, as one can always achieve admissibility for the associativity theorem (Theorem~\ref{thm:assoc}) in our case.
The summation is necessary in the untwisted case due to the corresponding ambiguity of the gluing of $\spinc$
structures in the associativity theorem.
\end{proof}

\section{One- and three-handles}

\begin{defn} \label{def:framed-pair}
Let $(M,\g)$ be a sutured manifold. A \emph{framed pair of points} $\bP$ consists of two distinct points~$p_+$ and~$p_-$
lying in the interior of~$M$, together with a positive frame $\langle\, v_1^+, v_2^+, v_3^+ \,\rangle$ of
$T_{p_+}M$ and a negative frame $\langle\, v_1^-, v_2^-, v_3^- \,\rangle$ of $T_{p_-}M$.
\end{defn}

Given a sutured manifold $(M,\g)$ and a framed pair of points $\bP$, let
\[
W(\bP) = (W,Z,[\xi])
\]
be the cobordism from $(M,\g)$ to $(M(\bP),\g)$ obtained by attaching a single
$4$-dimensional $1$-handle~$H = D^1 \times D^3$ to $M \times I$ along~$\bP$. I.e., $(-1,0) \in H$ is glued
to~$p_-$ and~$(1,0) \in H$ is glued to~$p_+$, and $\{-1\} \times D^3$ induces the framing of~$p_-$, while
$\{1\} \times D^3$ induces the framing of~$p_+$.
This is also a special cobordism, with~$Z = \partial M \times I$,
and~$\xi$ being an $I$-invariant contact structure such that for every~$t \in I$ the surface
$\partial M \times \{t\}$ is convex with dividing set~$\g \times \{t\}$.
The manifold~$M(\bP)$ is obtained from~$M$ by removing the interiors of the balls~$\{-1\} \times D^3$
and~$\{1\} \times D^3$, and gluing~$D^1 \times S^2$.
There are two possibilities for~$M(\bP)$:
\begin{enumerate}
\item \label{it:same} If $p_+$ and $p_-$ lie in the same
component $M_0$ of $M$,
then
\[
M(\bP) \approx M\# (S^1 \times S^2),
\]
where the connected sum is taken between $M_0$ and $S^1 \times S^2$.
\item  \label{it:different} Otherwise, there are distinct components~$M_1$ and~$M_2$ of~$M$
such that $M(\bP)$ has $M_1$ and~$M_2$ replaced by $M_1 \# M_2$.
\end{enumerate}
In both cases, we have $\SFH(M(\bP),\g) \cong \SFH(M,\g) \otimes \Z^2$ by~\cite[Proposition~9.15]{sutured}.
As the restriction map $H^2(W,Z) \to H^2(M, \partial M)$ is an isomorphism, we have
$\spinc(W(\bP)) \cong \spinc(M,\g)$ (recall that for a special cobordism it does not
matter whether we consider $\spinc$ structures relative to~$\partial Z$ or~$Z$).
In case~\eqref{it:same}, a $\spinc$ structure $\s'$ on $(M(\bP),\g)$ extends over $W(\bP)$ if and only if
$\s'|_{S^1 \times S^2} = \s_0$, where~$\s_0$ is characterized by $c_1(\s_0) = 0$.
In case~\eqref{it:different}, every $\spinc$ structure extends from~$(M(\bP),\g)$ over~$W(\bP)$.
Given a $\spinc$ structure $\s \in \spinc(W(\bP))$, we also write~$\s$
for the corresponding element of $\spinc(M,\g)$ by a slight abuse of notation,
and let~$\s_{\bP} = \s|_{(M(\bP),\g)}$.
Note that in case~\eqref{it:same}, we have $\s_{\bP} = \s \# \s_0$.

\begin{defn} \label{def:point-bouquet}
Let $\bP$ be a framed pair of points in the sutured manifold~$(M,\g)$.
A \emph{bouquet for~$\bP$} consists of a pair of disjoint
embedded arcs~$\eta_+$ and~$\eta_-$ in~$M$ and a framing of~$\eta_\pm$
given by a normal vector field~$v_\pm$ such that
\begin{enumerate}
\item $\eta_\pm(0) = p_\pm$ and $\eta_\pm(1) \in s(\g)$,
\item $\eta_\pm'(0) = v_1^\pm$ and $\eta_\pm'(1)$ is transverse to~$\partial M$,
\item $v_\pm(0) = v_2^\pm$.
\end{enumerate}
\end{defn}

\begin{defn}
Let $(M,\g)$ be a sutured manifold, together with a framed pair of points~$\bP$,
and let~$B(\bP)$ be a bouquet for~$\bP$.
We say that the diagram $\H = (\S,\bolda,\boldb)$ is \emph{subordinate to~$B(\bP)$} if
\begin{enumerate}
\item $\eta_+ \cup \eta_- \subset \S$,
\item the framings~$v_+$ and~$v_-$ are tangent to~$\S$, and
\item $(\eta_+ \cup \eta_-) \cap (\bolda \cup \boldb) = \emptyset$.
\end{enumerate}
\end{defn}

\begin{lem} \label{lem:connect-framed}
Let $(M,\g)$ be a balanced sutured manifold, together with a framed pair of points~$\bP$ and bouquet~$B(\bP)$.
Then $(M,\g)$ has an admissible diagram subordinate to $B(\bP)$.

If\, $\H = (\S,\bolda,\boldb)$ and~$\overline{\H} = (\overline{\S},\overline{\bolda},\overline{\boldb})$
are admissible diagrams of $(M,\g)$ subordinate to~$B(\bP)$,
then there is a sequence of admissible diagrams $\H_0,\dots,\H_n$, each subordinate to~$B(\bP)$,
such that $\H_0 = \H$ and $\H_n = \overline{\H}$. Furthermore, for every $i \in \{\,0, \dots, n-1\,\}$,
the diagrams~$\H_i$ and~$\H_{i+1}$ are related by an strong equivalence,
a stabilization, a destabilization, or a diffeomorphism isotopic in~$(M,\g)$ to the identity
through diagrams subordinate to~$B(\bP)$.
\end{lem}

\begin{proof}
Let~$N_\pm$ be a regular neighborhood of~$\eta_\pm$, and choose properly embedded disks~$D_\pm \subset N_\pm$
such that $\eta_\pm \subset D_\pm$, $v_\pm$ is tangent to~$D_\pm$, and
\[
D_\pm \cap \partial M = s(\g) \cap N_\pm.
\]
We define the suture manifold~$(M',\g')$ by taking~$M' = \overline{M \setminus (N_+ \cup N_-)}$ and rounding the corners, and
\[
s(\g') = s(\g) \bigtriangleup \partial (D_+ \cup D_-),
\]
where $\bigtriangleup$ denotes the symmetric difference. By~\cite{sutured}, the sutured manifold~$(M',\g')$
has an admissible diagram $(\S',\bolda,\boldb)$. If we let~$\S = \S' \cup D_+ \cup D_-$, then $(\S,\bolda,\boldb)$ is
an admissible diagram of~$(M,\g)$ subordinate to the bouquet~$B(\bP)$.

Now suppose we are given admissible diagrams~$\H = (\S,\bolda,\boldb)$
and~$\overline{\H} = (\overline{\S},\overline{\bolda},\overline{\boldb})$ of $(M,\g)$ subordinate to~$B(\bP)$.
Choose regular neighborhoods~$N_+$ and~$N_-$ of~$\eta_+$ and~$\eta_-$
so thin that they are disjoint from~$\bolda$, $\boldb$, $\bolda'$, and~$\boldb'$,
and let~$D_\pm = \S \cap N_\pm$.
After removing $D_+ \cup D_-$ from both~$\S$ and $\overline{\S}$,
we obtain admissible diagrams~$\H'$ and $\overline{\H'}$ of $(M',\g')$. By~\cite[Proposition~2.37]{naturality},
these can be connected by a sequence of admissible diagrams $\H_0',\dots,\H_n'$
such that for every $i \in \{\,0, \dots, n-1\,\}$,
the diagrams~$\H_i'$ and~$\H_{i+1}'$ are related by an strong equivalence,
a stabilization, a destabilization, or a diffeomorphism isotopic in~$(M',\g')$ to the identity.
By adding~$D_+ \cup D_-$ to each~$\H_i'$, we obtain the desired sequence of
diagrams~$\H_0,\dots,\H_n$.
\end{proof}

We define the map for $1$-handle addition as follows.

\begin{defn} \label{defn:onehandle}
Let $(M,\g)$ be a balanced sutured manifold together with a framed pair of points~$\bP$, and let
$W(\bP)$ be the corresponding cobordism from~$(M,\g)$ to $(M(\bP),\g)$.
Let $(A,\a,\b)$ be a sutured diagram, where $A = D^1 \times S^1$ is an annulus inside the
1-handle $H = D^1 \times D^3$,
and $\a$ and~$\b$ are homologically non-trivial, transverse simple closed curves such that~$|\a \cap \b| = 2$.
Let $\theta \in \a \cap \b$ be the intersection point with higher relative grading.

Choose a bouquet~$B(\bP)$ for~$\bP$, and suppose that $\H = (\S,\bolda,\boldb)$ is an admissible balanced
diagram of $(M,\g)$ subordinate to~$B(\bP)$.
Consider the open disks $D_- = \{-1\} \times \text{Int}(D^2)$ and $D_+ = \{1\} \times \text{Int}(D^2)$ whose union is~$\text{int}(\S \cap H)$
(these are neighborhoods of $p_-$ and $p_+$ in~$\S$, respectively),
and let $\S^0 = \S \setminus (D_- \cup D_+)$. We write~$\S_{\bP}$ for the surface obtained from~$\S^0$ by gluing~$A$ along
$\partial D_- \cup \partial D_+$, and smoothing the corners. Furthermore, we isotope $\bolda$ and $\boldb$ without
crossing~$p_+$ and~$p_-$ such that they become disjoint from~$D_+$ and~$D_-$.
Then
\[
\H_{\bP} = (\S_{\bP},\bolda_{\bP},\boldb_{\bP}) = (\S^0 \cup A, \bolda \cup \{\a\},\boldb \cup \{\b\})
\]
is a balanced diagram of $(M(\bP),\g)$.

For $\s \in \spinc(W(\bP))$, we define the map
\[
f_{\H,\bP,\s} \colon CF(\H,\s) \to CF(\H_{\bP},\s_{\bP})
\]
by the formula $g_{\H,\bP,\s}(\x) = \x \times \{\theta\}$.
This makes sense since $\s(\x \times \{\theta\}) = \s_{\bP}$,
and is a chain map since~$\partial D_+$ and $\partial D_-$ lie in components of
$\S \setminus (\bolda \cup \boldb)$ that intersect~$\partial \S$ non-trivially.
The induced map on homology is
\[
F_{\H,\bP,\s} \colon \SFH(\H,\s) \to \SFH(\H_{\bP},\s_{\bP}).
\]
Then the map associated to attaching a $1$-handle along $\bP$ is
\[
F_{M,\bP, \s} = P_{\H_{\bP}}^{-1} \circ F_{\H,\bP,\s} \circ P_{\H} \colon \SFH(M,\g,\s) \to \SFH(M(\bP),\g,\s).
\]
Similarly, we define the map
\[
F_{\H,\bP} \colon \SFH(\H) \to \SFH(\H_{\bP})
\]
to be the one induced by $f_{\H,\bP}(\x) = \x \times \{\theta\}$,
and let
\[
F_{M,\bP} = P_{\H_{\bP}}^{-1} \circ F_{\H,\bP} \circ P_{\H} \colon \SFH(M,\g) \to \SFH(M(\bP),\g).
\]
\end{defn}

The following theorem is an analogue of \cite[Theorem 4.10]{OSz10}, and ensures that the
maps $F_{M,\bP}$ and $F_{M,\bP,\s}$ are well-defined; i.e., independent of the
choice of bouquet and subordinate diagram.

\begin{thm} \label{thm:1-nat}
Let $\bP$ be a framed pair of points in the balanced sutured manifold~$(M,\g)$.
If $\H = (\S,\bolda,\boldb)$ and $\H' = (\S',\bolda',\boldb')$ are admissible balanced diagrams for $(M,\g)$
subordinate to the bouquets~$B(\bP)$ and $B(\bP)'$, respectively,
and if~$\s \in \spinc(W(\bP))$, then the following diagram is commutative:
\begin{equation} \label{eqn:one-nat}
\begin{CD}
\SFH(\H,\s) @>F_{\H,\bP,\s}>> \SFH(\H_{\bP},\s_{\bP})\\
@VV F_{\H,\H'} V   @VV F_{\H_{\bP}, \H_{\bP}'} V\\
\SFH(\H',\s) @>F_{\H',\bP,\s}>> \SFH(\H_{\bP}',\s_{\bP}),
\end{CD}
\end{equation}
where the vertical maps are the canonical isomorphisms defined in Section~\ref{sec:naturality}.
An analogous statement holds for~$F_{\H,\bP}$ and~$F_{\H',\bP}$.
\end{thm}

\begin{proof}
First, suppose that~$B(\bP) = B(\bP)'$. By Lemma~\ref{lem:connect-framed}, we can connect~$\H$ and~$\H'$
via a sequence of admissible diagram $\H_0, \dots, \H_n$ such that consecutive diagrams
differ by a strong equivalence, a (de)stabilization, or a diffeomorphism isotopic to the identity
through diagrams subordinate to~$B(\bP)$. This gives rise to a sequence of diagrams
$(\H_0)_{\bP},\dots,(\H_n)_{\bP}$ of $(M(\bP),\g)$ such that consecutive diagrams also differ by one
of the above moves. It suffices to show that diagram~\eqref{eqn:one-nat} commutes for~$\H = \H_i$ and~$\H' = \H_{i+1}$
for every $i \in \{\,0, \dots, n-1\,\}$.

If~$\H$ and~$\H'$ differ by a strong equivalence, then
the maps~$F_{\H,\H'}$ and~$F_{\H_\bP,\H_\bP'}$ are compositions of triangle maps where the relevant
triple diagrams have multiplicity zero along~$\eta_+$ and~$\eta_-$, and there is a single small triangle
in~$A$ that contributes, cf.~\cite[Theorem~4.10]{OSz10}. More precisely, recall that the map
$F_{\H,\H'}$ is constructed by taking attaching sets~$\ol{\bolda}$ and~$\ol{\boldb}$ such that
the quadruple diagrams $(\S,\bolda,\boldb,\ol{\bolda},\ol{\boldb})$ and $(\S,\bolda',\boldb',\ol{\bolda},\ol{\boldb})$
are both admissible, and then
\[
F_{\H,\H'} = \Psi^{\ol{\bolda} \to \bolda'}_{\ol{\boldb} \to \boldb'} \circ
\Psi^{\bolda \to \ol{\bolda}}_{\boldb \to \ol{\boldb}}.
\]
Furthermore,
\[
\Psi^{\ol{\bolda} \to \bolda'}_{\ol{\boldb} \to \boldb'} = \Psi^{\bolda'}_{\ol{\boldb} \to \boldb'} \circ
\Psi^{\ol{\bolda} \to \bolda'}_{\ol{\boldb}}
\text{  and  }
\Psi^{\bolda \to \ol{\bolda}}_{\boldb \to \ol{\boldb}} = \Psi^{\ol{\bolda}}_{\boldb \to \ol{\boldb}} \circ
\Psi^{\bolda \to \ol{\bolda}}_{\boldb}.
\]
By finger moves along~$\eta_\pm$, we can also arrange that $(\eta_+ \cup \eta_-) \cap (\ol{\bolda} \cup \ol{\boldb})  = \emptyset$.
Hence, it suffices to show that diagram~\eqref{eqn:one-nat} is commutative either when
$(\S,\boldd,\bolda,\boldb)$ is an admissible triple, the diagrams
$\H = (\S,\bolda,\boldb)$ and $\H' = (\S,\boldd,\boldb)$ are subordinate to~$B(\bP)$,
and $F_{\H,\H'}$ is given by $\Psi^{\bolda \to \boldd}_{\boldb}$,
or when $(\S,\bolda,\boldb,\boldd)$ is admissible, the diagrams $\H = (\S,\bolda,\boldb)$ and $\H' = (\S,\bolda,\boldd)$
are subordinate to~$B(\bP)$,
and $F_{\H,\H'}$ is given by $\Psi^{\bolda}_{\boldb \to \boldd}$. Since the second case is completely analogous,
we will only prove the first one. The attaching sets $\bolda_\bP = \bolda \cup \{\a\}$, $\boldb_\bP = \boldb \cup \{\b\}$,
and $\boldd_\bP = \boldd \cup \{\d\}$ on $\S_\bP$ are obtained by adding the meridional curves $\a$, $\b$, $\d$ in the annulus~$A$
that pairwise intersect in two points, and are arranged as shown in Figure~\ref{fig:one-handle}.
\begin{figure}
\includegraphics{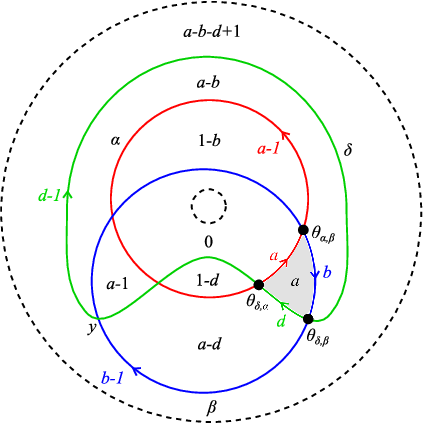}
\caption{The annulus~$A$ is bounded by the two dashed circles. The diagram shows possible multiplicities
of a domain of a triangle in the triple diagram $(\d,\a,\b)$ connecting the ``top'' intersection points $\theta_{\d,\a}$,
$\theta_{\a,\b}$, and $\theta_{\d,\b}$.}
\label{fig:one-handle}
\end{figure}
The ``top'' intersection point between~$\a$ and~$\b$ is denoted by $\theta_{\a,\b}$, and we define~$\theta_{\d,\a}$
and~$\theta_{\d,\b}$ similarly. Pick a generator $\x \in \T_\a \cap \T_\b$,
then commutativity of diagram~\eqref{eqn:one-nat} follows if we show that
\[
F_{\d_\bP,\a_\bP,\b_\bP}\left(\Theta_{\d_\bP,\a_\bP} \otimes \left(\x \times \{\theta_{\a,\b}\}\right)\right) =
F_{\d,\a,\b}(\Theta_{\d,\a} \otimes \x) \times \{\theta_{\d,\b}\}.
\]
Note that $\Theta_{\d_\bP,\a_\bP} = \Theta_{\d,\a} \times \{\theta_{\d,\a}\}$.
Now suppose that~$\D_\bP$ is a positive domain in~$(\S,\boldd_\bP,\bolda_\bP,\boldb_\bP)$ connecting the
intersection points $\Theta_{\d,\a} \times \{\theta_{\d,\a}\}$, $\x \times \{\theta_{\a,\b}\}$,
and $\y \times \{\theta_{\d,\b}\}$, where $\y \in \T_\d \cap \T_\b$. Then the coefficients of~$\D_\bP$
are zero along~$\partial A$, and hence $\D_\bP = \D + \D'$, where~$\D$ is supported in~$\S$, while~$\D'$
is supported in~$A$.

Now we determine~$\D'$. As in Figure~\ref{fig:one-handle}, since $\D'$ connects $\theta_{\d,\a}$, $\theta_{\a,\b}$,
and $\theta_{\d,\b}$, we must have $\partial(\partial \D' \cap \a) = \theta_{\a,\b} - \theta_{\d,\a}$,
$\partial (\partial \D' \cap \b) = \theta_{\d,\b} - \theta_{\a,\b}$, and
$\partial (\partial \D' \cap \d) = \theta_{\d,\a} - \theta_{\d,\b}$. Hence, there are integers $a$, $b$, and $c$
such that $\partial \D' \cap \a$ has coefficients $a$ and $a-1$, while $\partial \D' \cap \b$
has coefficients $b$ and $b-1$, and $\partial \D' \cap \d$ has coefficients $d$ and $d-1$,
with the orientation as in Figure~\ref{fig:one-handle}.
From this, we can determine all the coefficients of $\D'$, they are as in Figure~\ref{fig:one-handle}.
As the coefficients of~$\D'$ are zero along~$\partial A$, it follows that $a-b-d+1 = 0$.
Since~$\D' \ge 0$, we have both $b-1 = a-d \ge 0$ and~$1-b \ge 0$,
hence $b = 1$ and $a = d$. Furthermore, $a-1 \ge 0$ and $1-d = 1-a \ge 0$, so $a = 1$.
It follows that all the coefficients of~$\D'$ are zero, except it is one in the shaded triangle.
Hence, by the Riemann mapping theorem, the moduli space of pseudo-holomorphic triangles~$\M(\D)$
can be identified with $\M(\D_\bP)$ by pasting in the unique holomorphic representative of the shaded triangle.

If $\{y\} = (\d \cap \b) \setminus \{\theta_{\d,\b}\}$, then $\y \times \{y\}$ does not appear in
\[
F_{\d_\bP,\a_\bP,\b_\bP}\left(\Theta_{\d_\bP,\a_\bP} \otimes \left(\x \times \{\theta_{\a,\b}\}\right)\right),
\]
as $\mu \left(\y \times \{\theta_{\d,\b}\}, \y \times \{y\} \right) = 1$.
In fact, a computation analogous to the above shows that in $(A,\d,\a,\b)$,
there are three domains~$\D'$ of triangles connecting $\theta_{\d,\a}$, $\theta_{\a,\b}$, and~$y$,
each one of which has Maslov index one. Hence, if $\D_\bP = \D + \D'$ did contribute to the coefficient of $\y \times \{y\}$,
then $\mu(\D) = \mu(\D_\bP) - \mu(\D') = 0-1 = -1$, and so $\D$, and also $\D_\bP$, would have no pseudo-holomorphic
representative for a generic path of almost complex structures.

Stabilization invariance is straightforward.
When~$\H$ and~$\H'$ differ by a diffeomorphism~$d$ isotopic to the identity, as this isotopy can be chosen
to fix a neighborhood of~$\eta_+ \cup \eta_-$, then~$d$ extends to a diffeomorphism
$d_\bP \colon \H_\bP \to \H_\bP'$ isotopic to the identity in $(M(\bP),\g)$. Furthermore,
$d(\x \times \{\theta\}) = d(\x) \times \{\theta\}$, and the diagram commutes.
So we have shown that diagram~\eqref{eqn:one-nat} commutes when~$B(\bP) = B(\bP)'$.

Now we show commutativity of diagram~\eqref{eqn:one-nat} when the bouquets~$B(\bP)$ and~$B(\bP)'$
are ambient isotopic through bouquets of~$\bP$. More precisely, if~$\phi_t \colon M \to M$
is the isotopy, then we require that
\begin{itemize}
\item $\phi_t(s(\g)) = s(\g)$,
\item $\phi_t(\bP) = \bP$, and
\item $\phi_1(B(\bP)) = B(\bP)'$.
\end{itemize}
By the previous part, it suffices to show commutativity of~\eqref{eqn:one-nat}
for a single pair of admissible diagrams~$\H$ and~$\H'$ subordinate to~$B(\bP)$ and~$B(\bP)'$,
respectively. Choose an arbitrary admissible diagram~$\H$ subordinate to~$\bP$, then $\H' = \phi_1(\H)$ is
a diagram of~$(M,\g)$ subordinate to~$B(\bP)'$. Since~$d = \phi_1|_{\H} \colon \H \to \H'$
is a diffeomorphism isotopic to the identity,
\[
F_{\H,\H'} = d_* \,\colon\, \SFH(\H,\s) \to \SFH(\H',\s).
\]
As each~$\phi_t$ fixes~$\bP$, the diffeomorphism~$d$ extends to a diffeomorphism $d_\bP \colon \H_\bP \to \H_\bP'$
isotopic to the identity in $(M(\bP),\g)$, and such that~$d_\bP$ fixes~$A$ pointwise.
In particular, $F_{\H_\bP,\H_\bP'} = (d_\bP)_*$. For any generator~$\x$ of $CF(\H,\s)$,
we get that
\[
f_{\H',\bP,\s} \circ F_{\H,\H'}(\x) =  d(\x) \times \{\theta\} = d_\bP(\x \times \{\theta\}) = F_{\H_\bP,\H_\bP'} \circ f_{\H,\bP,\s}(\x).
\]
This concludes the proof of invariance under ambient isotopy of the bouquet.

The final step is showing that diagram~\eqref{eqn:one-nat} is commutative for an arbitrary pair
of bouquets~$B(\bP)$ and~$B(\bP)'$. We can make~$B(\bP)$ and~$B(\bP)'$ disjoint in the complement
of~$\bP$ by an ambient isotopy, and by the previous paragraph, it suffices to prove the claim
under this assumption.

Next, we construct a diagram~$\H = (\S,\bolda,\boldb)$ subordinate to~$B(\bP)$ and almost subordinate to~$B(\bP)'$,
except that~$|\eta_\pm' \cap \bolda| = 1$ and~$|\eta_\pm' \cap \boldb| = 1$.
This can be done similarly to the proof of Lemma~\ref{lem:connect-framed}.
Indeed, pick regular neighborhoods~$N_\pm$ of~$\eta_\pm \cup \eta_\pm'$,
and let $D_\pm \subset N_\pm$ be properly embedded disks containing~$\eta_\pm \cup \eta_\pm'$
and tangent to the vector fields~$v_\pm$ and~$v_\pm'$. We also assume
that~$D_\pm \cap \partial M = s(\g) \cap N_\pm$. We define the sutured manifold~$(M',\g')$
by taking~$M' = M \setminus (N_+ \cup N_-)$ and
\[
s(\g') = s(\g) \bigtriangleup \partial (D_+ \cup D_-).
\]
Pick an admissible diagram $(\S',\bolda',\boldb')$ of $(M',\g')$. If we set $\S = \S' \cup D_+ \cup D_-$,
then $(\S,\bolda',\boldb')$ is an admissible diagram of~$(M,\g)$ with holes drilled above and below~$D_\pm$.
We can fill these holes by attaching 3-dimensional 2-handles, which translates to adding an $\a$-curve~$\a_+$ and
a $\b$-curve~$\b_+$ intersecting~$\eta_+ \cup \eta_+'$ in one point each,
which we can then isotope to intersect only $\eta_+'$, and similarly, we add a new $\a$-curve~$\a_-$
and a $\b$-curve~$\b_-$ intersecting~$\eta_-'$ in one point each.
It will be helpful later if we isotope~$\a_\pm$ and~$\b_\pm$ such that they each intersect~$D_\pm$
in a single arc, these arcs have two transverse intersection points on opposite sides of~$\eta_\pm'$,
and the point $\a_\pm \cap \eta_\pm'$ is closer to~$p_\pm$ along~$\eta_\pm'$ than $\b_\pm \cap \eta_\pm'$.

So if we set~$\bolda = \bolda' \cup \a_+ \cup \a_-$ and~$\boldb = \boldb' \cup \b_+ \cup \b_-$, then the diagram
$\H = (\S,\bolda,\boldb)$ is a diagram of~$(M,\g)$ that is subordinate to $B(\bP)$ and almost subordinate to~$B(\bP)'$.
If we perform a finger move on~$\a_\pm$ and~$\b_\pm$ along~$\eta_\pm'$, obtaining curves~$\a_\pm'$
and~$\b_\pm'$, and applying a small Hamiltonian isotopy to $\bolda'$ and $\boldb'$,
we get a diagram~$\H' = (\S,\boldd,\bolde)$ that is
subordinate to~$B(\bP)'$ and almost subordinate to~$B(\bP)$. By a small isotopy, we also arrange that
$\a_\pm$ and~$\a_\pm'$ -- and similarly, $\b_\pm$ and $\b_\pm'$ -- intersect transversely in two points.
The curves~$\a_+$, $\a_+'$, and $\b_+$ inside~$D_+$ are depicted in Figure~\ref{fig:one-handle-nat}.
\begin{figure}
\includegraphics{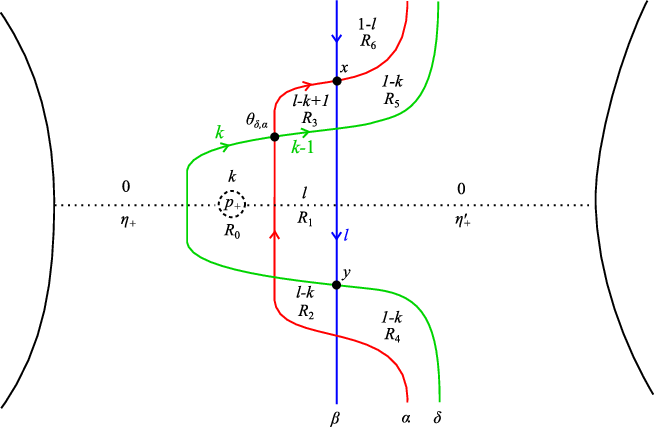}
\caption{The curves $\a = \a_+$, $\b = \b_+$, and $\d =\a_+'$ inside the disk $D_+$.}
\label{fig:one-handle-nat}
\end{figure}
For simplicity, we write $\a = \a_+$, $\b = \b_+$, and~$\d = \a_+'$.
It suffices to show the commutativity of diagram~\eqref{eqn:one-nat} for these diagrams~$\H$ and~$\H'$.
Note that the quadruple diagram $(\S,\bolda,\boldb,\boldd,\bolde)$ is admissible, so
\[
F_{\H,\H'} = \Psi^{\boldd}_{\boldb \to \bolde} \circ \Psi^{\bolda \to \boldd}_{\boldb}.
\]
We will explain why diagram~\eqref{eqn:one-nat} with vertical map $\Psi^{\bolda \to \boldd}_{\boldb}$
and diagrams $\H = (\S,\bolda,\boldb)$ and $\H' = (\S,\boldd,\boldb)$ is commutative, as
commutativity for $\Psi^{\boldd}_{\boldb \to \bolde}$ is analogous.

We denote by $\Theta_{\d,\a}$ the ``top'' intersection point between $\boldd$ and $\bolda$,
this has coordinate $\theta_{\d,\a} \in \d \cap \a$.
Let $\D$ be a positive domain in $(\S,\boldd,\bolda,\boldb)$ connecting the generators $\Theta_{\d,\a}$,
$\x \in \T_\a \cap \T_\b$, and $\y \in \T_\b \cap \T_\d$ that has a rigid pseudo-holomorphic representative
for a generic path of almost complex structures.
We claim that the coefficients of~$\D$ at~$p_+$ and~$p_-$ are both zero. We label the closures of some of the components
of $\S \setminus (\bolda \cup \boldb \cup \boldd)$ by $R_0, \dots, R_6$, as in Figure~\ref{fig:one-handle-nat}.
Let~$k$ and~$l$ denote the coefficients of~$\D$ at~$R_0$ and~$R_1$, respectively.
Since $\partial (\partial\D \cap \a) = x - \theta_{\d,\a}$, where $x = \x \cap \a$,
the multiplicities of $\partial\D \cap \a$ are $l-k$ and $l-k+1$ along the two components of
$\a \setminus \{\theta_{\d,\a},x\}$. As~$\D$ has multiplicity zero along
$\eta_+ \setminus R_0$, the coefficient of~$\D$ at~$R_2$ has to be~$l-k$, hence $l \ge k$.
Furthermore, the multiplicity of $\partial \D \cap (\b \cap R_1)$ is~$l$, so
if~$x$ were not the point $R_3 \cap R_6$ shown in Figure~\ref{fig:one-handle-nat},
then $R_6$ would have coefficient~$-l$. As $\D \ge 0$, this would imply that $l = 0$,
and hence also~$k=0$.

Now assume that~$x = R_3 \cap R_6$, then the coefficient of~$R_6$ is~$1-l$.
So if $k \ge 1$, then $\D \ge 0$ implies that $k = l = 1$. If $y = \y \cap \b$ was not the
point~$R_1 \cap R_4$, then, by considering the multiplicities of~$\partial \D$ along~$\d$,
we see that $R_4$ would have coefficient~$-k = -1$. We conclude that $y = R_1 \cap R_4$,
and inside $D_+$, the domain~$\D$ has coefficients one in $R_0$, $R_1$, and $R_3$, and zero
everywhere else. In particular, around the ``triangle'' $T = R_0 \cup R_1 \cup R_3$,
all multiplicities are zero. Any pseudo-holomorphic representative of such a domain~$\D$
would have to consist of a branched cover where one component is a triangle mapping to $T$.
But one can make cuts along the arcs $R_0 \cap R_1$ and along $R_1 \cap R_3$, and so by the
Riemann mapping theorem, the space of such mappings is one-dimensional, and so cannot be rigid.
In fact, the Maslov index of $T$ is one, so if $\mu(\D) = 0$, then $\mu(\D - T) < 0$, and hence
generically~$\D$ has no pseudo-holomorphic representative.
An analogous argument shows that the coefficient of~$\D$ at~$p_-$ is also zero.

Then the commutativity of diagram~\eqref{eqn:one-nat} follows just like in the case when we had $\H$ and $\H'$
strongly equivalent and both subordinate to the same bouquet.
\end{proof}

Now we define the map induced by attaching a 3-handle.
Let~$(M',\g')$ be a balanced sutured manifold, together with a framed 2-sphere~$S \subset M$.
By this, we mean that a neighborhood of~$S$ is identified with~$S^2 \times D^1$ such that $S = S^2 \times \{0\}$.
Let $W(S)$ be the special cobordism from~$(M',\g')$ to~$(M,\g)$ obtained as the trace of attaching a three-handle~$D^3 \times D^1$
to~$(M',\g')$ along the framed two-sphere $S$. Then we obtain $(M,\g)$ from $(M',\g')$ by removing
$S \times (-1,1)$ and capping off the two boundary components with~$D^3 \times D^0$.
Note that if~$S$ is non-separating, then $(M,\g)$ is automatically balanced. When~$S$ is separating, then
we make the additional assumption that~$(M,\g)$ is also balanced.
Let~$\bP$ be the pair of points
\[
p_\pm = (0,\pm 1) \in D^3 \times D^0 \subset M,
\]
framed by $(\partial/\partial x,\partial/\partial y, \partial/\partial z)$ at~$p_+ \in D^3 \times \{1\}$, and
by the opposite vectors at~$p_- \in D^3 \times \{-1\}$.

\begin{defn} \label{def:adapted-to-sphere}
We say that the diagram $\H' = (\S',\bolda',\boldb')$ is adapted to the framed 2-sphere~$S$
if it is of the form~$\H_\bP = (\S^0 \cup A, \bolda \cup \{\a\},\boldb \cup \{\b\})$
for some admissible diagram~$\H = (\S,\bolda,\boldb)$ of~$(M,\g)$ subordinate to
some bouquet~$B(\bP)$ for~$\bP$ (recall that~$\H_\bP$ was introduced in Definition~\ref{defn:onehandle}).
In other words,
\[
A := \S' \cap (S \times D^1) = S^1 \times D^1,
\]
while $\bolda' \cap A = \{\a\}$ and
$\boldb' \cap A = \{\b\}$ are two meridians of~$A$ that intersect each other in a pair of points,
and both components of~$\partial A$ can be connected with $\partial \S'$ along an embedded arc disjoint
from $\bolda$, $\boldb$, and the interior of~$A$.
\end{defn}

Such a diagram always exists according to Lemma~\ref{lem:connect-framed} applied to~$(M,\g)$
and an arbitrary bouquet~$B(\bP)$ for~$\bP$.

\begin{defn} \label{defn:3handle}
Let $S \subset M'$ be a framed embedded sphere in the balanced sutured manifold $(M',\g')$, and take an adapted balanced diagram
\[
\H' = (\S^0 \cup A, \bolda \cup \{\a\},\boldb \cup \{\b\}) = \H_\bP.
\]
For $\s \in \spinc(W(S))$, let
\[
f_{\H',S,\s} \colon CF\left(\H',\s|_{M'}\right) \to CF\left(\H,\s|_{M}\right)
\]
such that for $\x \in \T_{\bolda} \cap \T_{\boldb}$ and~$y \in \a \cap \b$, we have $f_{\H,\s}(\x \times \{y\}) = 0$ if $y = \theta$,
and~$f_{\H',S,\s}(\x \times \{y\}) = \x$ if $y$ is the intersection point of smaller relative grading.
Then $f_{\H',S,\s}$ is a chain map, since any domain in $\H'$ has coefficient zero along~$\partial A$.
It induces the map
\[
F_{\H',S,\s} \colon \SFH\left(\H',\s|_{M'}\right) \to \SFH\left(\H,\s|_M \right).
\]
on the homology. From this, we obtain the map
\[
F_{M',S,\s} = \P_{\H}^{-1} \circ f_{\H',S,\s} \circ \P_{\H'} \colon \SFH \left(M',\g',\s|_{M'}\right) \to \SFH \left(M,\g,\s|_M \right).
\]
Similarly, we also have a map
\[
F_{M',S} \colon \SFH(M',\g') \to \SFH(M,\g)
\]
that does not refer to a particular $\spinc$ structure.
\end{defn}

\begin{thm}
Let $S \subset M'$ be a framed embedded sphere in the balanced sutured manifold $(M',\g')$, and
let $(M,\g)$, together with the framed pair of points~$\bP$, be the result of surgering~$(M',\g')$
along~$S$.
Suppose that~$(M,\g)$ is balanced, and let~$\H_1$ and~$\H_2$ be diagrams for~$(M,\g)$ adapted to
bouquets~$B_1(\bP)$ and~$B_2(\bP)$, respectively.
Then $\H_1' = (\H_1)_\bP$ and $\H_2' = (\H_2)_\bP$ are diagrams of~$(M',\g')$ adapted to~$S$.
If $\V(S)$ is the 3-handle cobordism corresponding to~$S$
and~$\s \in \spinc(W(S))$, then the following diagram commutes:
\[
\begin{CD}
\SFH \left(\H_1',\s|_{M'} \right) @>F_{\H_1',S,\s}>> \SFH \left(\H_1,\s|_M \right)\\
@VVF_{\H_1',\H_2'}V   @VVF_{\H_1,\H_2}V\\
\SFH \left(\H_2',\s|_{M'} \right) @>F_{\H_2',S,\s}>> \SFH \left(\H_2,\s|_M \right),
\end{CD}
\]
where the vertical maps are the canonical isomorphisms. An analogous result holds for $F_{\H_1',S}$
and $F_{\H_2',S}$.
\end{thm}

\begin{proof}
This is analogous to the proof of Theorem~\ref{thm:1-nat}.
\end{proof}

\section{The map associated to a special cobordism}

In~\cite{naturality}, we constructed a functor~$\SFH$ from the category of balanced sutured manifolds and diffeomorphisms
to~$\Vect$. Furthermore, we have homomorphisms $F_{M,\bP,\s}$ induced by surgery along a framed pair of points~$\bP$,
homomorphisms $F_{M,\L,\s}$ induced by surgery along a framed link~$\L$, and $F_{M,S,\s}$ induced by surgery along a framed 2-sphere~$S$.
In~\cite[Theorem~1.2]{surgery}, we give a set of necessary and sufficient conditions for these
to give rise to a TQFT $F \colon \BSut' \to \Vect$. Note that this result also states that we can avoid 0-handle and 4-handle attachments.
We restate this theorem here for convenience specifically for~$\BSut'$.
We only enrich the theory with $\spinc$ structures later for clarity, as they introduce an additional
layer of complexity.

Let $(M,\g)$ be a balanced sutured manifold.
A framed $k$-sphere in~$M$ is an embedding of~$S^k \times D^{3-k}$ in the interior of~$M$. We denote by~$W(M,\SS)$ the
elementary cobordism obtained by attaching a 4-dimensional $k$-handle to~$M \times I$ along~$M \times \{1\}$ with
an I-invariant contact structure on~$Z = \partial M \times I$
with dividing set~$s(\g) \times \{t\}$ on~$\partial M \times \{t\}$ for every~$t \in I$.
This is a cobordism from~$(M,\g)$ to~$(M(\SS),\g)$, where~$M(\SS)$ is obtained by surgery along~$\SS$, and hence we
call~$W(M,\SS)$ the trace of the surgery. Usually the manifold~$M$ is unambiguous from
the notation~$\SS$, in which case we only write~$W(\SS)$ instead of~$W(M,\SS)$.
If~$\SS \colon S^k \times D^{3-k} \hookrightarrow M$
is a framed $k$-sphere for $k < 3$, let~$\ol{\SS}$ be the framed sphere defined by
\[
\ol{\SS}(\underline{x},\underline{y}) = \SS \left(r_{k+1}(\underline{x}),r_{3-k}(\underline{y}) \right),
\]
where $\underline{x} \in \R^{k+1}$, $\underline{y} \in \R^{3-k}$,
and
\[
r_n(x_1,x_2,\dots,x_n) = (-x_1,x_2,\dots,x_n).
\]

\begin{thm} \label{thm:bsut}
To define a functor $F \colon \BSut' \to \Vect$, it suffices to construct
a functor~$F$ from the category of balanced sutured manifolds and diffeomorphisms to~$\Vect$,
and for every balanced sutured manifold~$(M,\g)$, framed 0-, 1- or 2-sphere~$\SS \subset \text{Int}(M)$,
a linear map
\[
F_{M,\SS} \colon F(M,\g) \to F(M(\SS),\g)
\]
that satisfy the following axioms:
\begin{enumerate}
\item \label{it:isot} %\label{it:id}
If $d \in \Diff_0(M,\g)$, then $F(d) = \Id_{F(M,\g)}$.
\item \label{it:d-F}
Given a diffeomorphism $d \colon (M,\g) \to (M',\g')$ between balanced sutured manifolds and a framed
0-, 1-, or 2-sphere $\SS \subset \text{Int}(M)$, let $\SS' = d(\SS)$ and
\[
d^\SS \colon (M(\SS),\g') \to (M'(\SS'),\g')
\]
be the induced diffeomorphisms. Then the following diagram is commutative:
\[
\xymatrixcolsep{3pc} \xymatrix{
  F(M,\g) \ar[r]^-{F_{M,\SS}} \ar[d]^-{F(d)} & F(M(\SS),\g) \ar[d]^{F(d^\SS)} \\
  F(M',\g')  \ar[r]^-{F_{M',\SS'}} & F(M'(\SS'),\g').}
\]
\item \label{it:commut} If $(M,\g)$ is a balanced sutured manifold and $\SS$ and $\SS'$ are \emph{disjoint} framed
0-, 1-, or 2-spheres in $M$,
then $M(\SS)(\SS') = M(\SS')(\SS)$, we denote this manifold by $M(\SS,\SS')$.
Then the following diagram is commutative:
\[
\xymatrixcolsep{5pc}\xymatrix{
  F(M,\g) \ar[r]^-{F_{M,\SS}} \ar[d]_-{F_{M,\SS'}} & F(M(\SS),\g) \ar[d]^{F_{M(\SS),\SS'}} \\
  F(M(\SS'),\g)  \ar[r]^-{F_{M(\SS'),\SS}} & F(M(\SS,\SS'),\g).}
\]
\item \label{it:birth} If $\SS' \subset M(\SS)$ intersects the belt sphere of the handle attached along~$\SS$ once transversely,
then there is a diffeomorphism $\varphi \colon M \to M(\SS)(\SS')$ (which is the identity on $M \cap M(\SS)(\SS')$,
and is unique up to isotopy),
such that
\[
F_{M(\SS),\SS'} \circ F_{M,\SS} = F(\varphi).
\]
\item \label{it:0-sphere} For every framed 0-, 1-, or 2-sphere~$\SS$, we have
$F_{M,\SS} = F_{M,\ol{\SS}}$.
\end{enumerate}
The functor~$F$ is a TQFT if and only if it is symmetric and monoidal.
In the opposite direction, every functor $F \colon \BSut' \to \Vect$ arises in this way.
\end{thm}

\begin{thm} \label{thm:welldef}
The functor~$\SFH$ together with the surgery maps~$F_{M,\bP}$, $F_{M,\L}$, and~$F_{M,S}$
satisfy the axioms listed in Theorem~\ref{thm:bsut}, hence they
give rise to a TQFT $F \colon \BSut' \to \Vect$.
\end{thm}

\begin{proof}
Axiom~\eqref{it:isot}, isotopy invariance, is part of~\cite[Theorem~1.9]{naturality}.
Axiom~\eqref{it:0-sphere} is straightforward as the construction of the map~$F_{M,\SS}$
is independent of the orientation of the framed sphere~$\SS$.

We now show axiom~\eqref{it:d-F}, naturality of the surgery maps.
First, let~$\SS$ be a framed 1-sphere. Choose a bouquet~$B(\SS)$ for~$\SS$;
this is just an embedded arc connecting~$\SS$ and~$R_+(\g)$. Furthermore,
let $\cT = (\S,\bolda,\boldb,\boldd)$ be a triple diagram subordinate to~$B(\SS)$,
and let~$\H = (\S,\bolda,\boldb)$ and~$\H_\SS = (\S,\bolda,\boldd)$.
Then the triple diagram
\[
\cT' = d(\cT) = (\S',\bolda',\boldb',\boldd')
\]
is subordinate to the bouquet~$B(\SS') = d(B(\SS))$.
Let~$\phi = d|_\S \colon \S \to \S'$.
If we write~$\H' = (\S',\bolda',\boldd')$ and~$\H'_{\SS'} = (\S',\bolda',\boldd')$,
then the following diagram commutes by Lemma~\ref{lem:n-commute}:
\[
\begin{CD}
\SFH\left(\H \right) @>F_{\mathcal{T}}(\,\cdot\, \otimes\Theta_{\b,\d})>> \SFH\left(\H_{\SS} \right)\\
@VV\phi_*V   @VV\phi_*V\\
\SFH\left(\H' \right) @>F_{\mathcal{T}'}(\,\cdot\, \otimes\Theta_{\b',\d'})>> \SFH\left(\H'_{\SS'} \right).
\end{CD}
\]
Commutativity of the diagram in axiom~\eqref{it:d-F} follows from this by considering the commutative
cube whose top face is the diagram above, its bottom face is the diagram in axiom~\eqref{it:d-F},
the maps along the vertical edges are given by~$P_\H$, $P_{\H'_{\SS'}}$, $P_{\H_\SS}$, and~$P_{\H'}$
described in Section~\ref{sec:naturality},
and whose vertical faces commute by the definition of the maps $F(d)$, $F(d^\SS)$, $F_{M,\SS}$,
and~$F_{M',\SS'}$. When~$\SS$ is a framed 0- or 2-sphere, axiom~\eqref{it:d-F} follows similarly,
by choosing a diagram adapted to~$\SS$, and pushing it forward along~$d$.

Next, we verify axiom~\eqref{it:commut}, commutativity of disjoint surgeries.
There are six different cases depending on the dimensions~$s$
and~$s'$ of the framed spheres~$\SS$ and~$\SS'$. The case when~$s = s' = 1$ follows immediately from
Proposition~\ref{prop:linkcomp}.

Now suppose that $s = s' = 0$. Pick disjoint bouquets~$B(\SS)$ and~$B(\SS')$ for~$\SS$ and~$\SS'$,
respectively, as in Definition~\ref{def:point-bouquet}. The proof of Lemma~\ref{lem:connect-framed}
can be adapted to show the existence of
an admissible diagram~$\H = (\S,\bolda,\boldb)$ of~$(M,\g)$ subordinate to both~$B(\SS)$ and~$B(\SS')$.
We obtain the diagram~$\H_{\SS,\SS'}$ of~$M(\SS,\SS')$
by adding two annuli~$A$ and~$A'$ to~$\S$, together with
new curves~$\a$, $\b$ in $A$ and~$\a'$, $\b'$ in~$A'$. We get the diagram~$\H_\SS$
of~$M(\SS)$ by only adding~$(A,\a,\b)$, and~$\H_{\SS'}$ of~$M(\SS')$ by adding~$(A',\a',\b')$.
Then the diagram in axiom~\eqref{it:commut} already commutes on the chain level.
Indeed, for any $\x \in \T_\a \cap \T_\b$, we have
\begin{equation*}
\begin{split}
f_{\H_\SS,\SS'} \circ f_{\H,\SS}(\x) = \x \times \{\theta_{\a,\b}\} \times \{\theta_{\a',\b'}\} = \\
\x \times \{\theta_{\a',\b'}\} \times \{\theta_{\a,\b}\} = f_{\H_{\SS'},\SS} \circ f_{\H,\SS'}(\x).
\end{split}
\end{equation*}

The case $s = s' = 2$ is similar. We consider the manifold~$M(\SS,\SS')$, and if~$\bP$ and~$\bP'$ are the belt
spheres of the handles attached along~$\SS$ and~$\SS'$, respectively, then we choose disjoint bouquets~$B(\bP)$
and~$B(\bP')$ for~$\bP$ and~$\bP'$, respectively, and a diagram~$\H$ subordinate to both. Then commutativity
follows on the chain level in the diagram~$\H_{\bP,\bP'}$ adapted to both~$\SS$ and~$\SS'$.

When $s = 0$ and $s' = 2$, or vice versa, consider the manifold~$M(\SS')$, and let~$\bP'$ be the belt sphere
of the handle attached along~$\SS'$. Then pick disjoint bouquets~$B(\SS)$ and~$B(\bP')$ for~$\SS$ and~$\bP'$, respectively,
and take a diagram~$\H$ subordinate to both. The diagram~$\H_{\bP'}$ is then a diagram subordinate to
both~$\SS$ and~$\SS'$, in which commutativity follows on the chain level.

In the case~$s = 0$ and~$s' = 1$, or vice versa, pick disjoint bouquets~$B(\SS)$ and~$B(\SS')$.
Combining the proofs of  Lemma~\ref{lem:connect-framed} and Lemma~\ref{lem:bouquet},
we obtain a triple diagram~$\cT = (\S,\bolda,\boldb,\boldd)$ subordinate to~$B(\SS')$
such that~$(\S,\bolda,\boldb)$ is subordinate to~$B(\SS)$ and~$\boldd \cap B(\SS) = \emptyset$.
Then commutativity follows from the small triangle argument in the
annulus~$A$ attached along~$\SS$ illustrated by Figure~\ref{fig:one-handle}.

Finally, consider the case~$s = 1$ and~$s' = 2$. Then, as above, we can obtain
a diagram adapted to both~$\SS$ and~$\SS'$, in which commutativity follows from a small
triangle argument.

\begin{lem} \label{lem:onetwo}
$\SFH$ and the surgery maps~$F_{M,\SS}$
satisfy axiom~\eqref{it:birth} when~$\SS$ is a framed $0$-sphere and $\SS'$ is a framed $1$-sphere.
\end{lem}

\begin{proof}
We obtain~$M(\SS)$ from~$M$ by attaching the tube~$D^1 \times S^2$ to~$M \setminus N(\SS)$ along~$\partial N(\SS)$.
Since the curve~$\SS'$ intersects the core~$\{0\} \times S^2$ of the tube $D^1 \times S^2$ once transversely,
both points of~$\SS$ have to lie in the same component~$M_0$ of $M$.
Furthermore, we can isotope~$\SS'$ such that
\[
\SS' \cap (D^1 \times S^2) = D^1 \times \{(0,0,1)\}.
\]
Pick a bouquet~$B(\SS')$ for~$\SS' \subset M(\SS)$ by connecting~$\SS'$ and~$R_+(\g)$ via an arc disjoint from
the tube~$D^1 \times S^2$. Then extend~$D^1 \times \{(0,0,-1)\}$ to an embedded arc~$s$ in~$M(\SS)$
with both endpoints lying in~$R_-(\g)$ and disjoint from~$B(\SS')$.

We are now going to construct a triple diagram
$(\S',\bolda',\boldb',\boldd')$ of~$(M(\SS),\g)$ subordinate to~$B(\SS')$
that has a specific form in~$D^1 \times S^2$. Namely, we require that
\begin{equation} \label{eqn:tube}
\S' \cap (D^1 \times S^2) = D^1 \times S^1,
\end{equation}
and that~$\bolda' \cap (D^1 \times S^1) = \{0\} \times S^1$, call this curve~$\a$.
Furthermore, $\boldb' \cap (D^1 \times S^1)$ is a small translate of~$\a$ such that
$\a$ and~$\b$ intersect in two points transversely.

For this end, consider the sutured manifold
$(M',\g)$ obtained from~$(M(\SS),\g)$ by removing a regular neighborhood~$N$ of~$B(\SS')$
from~$M(\SS)$ and adding~$\partial N \setminus R_+(\g)$ to~$R_+$.
Note that $U := (\{0\} \times S^2) \setminus N$ is diffeomorphic to~$D^2$.
Let~$B$ be a regular neighborhood of~$U \cup s$, properly embedded in~$M'$;
this is a manifold with corners.
Recall that sutured functions and gradient-like vector fields for them were defined in
\cite[Definitions~5.12 and~5.13]{naturality}.
We can choose a Morse function~$f \colon B \to [-1,1]$
and a gradient-like vector field~$v$ for~$f$
such that~$f|_{B \cap \partial N} \equiv 1$, $f|_{B \cap R_-(\g)} \equiv -1$,
it has an index one critical point at~$o = (0,(0,0,-1))$ with stable manifold~$s$ and unstable
manifold~$U$, and no other critical points. Indeed, consider the Morse function
\[
g(x,y,z) = -x^2 + y^2 + z^2
\]
on~$\R^3$. Then there is a diffeomorphism
\[
d \colon B \to g^{-1}([-1,1]) \cap 2D^3,
\]
and we let~$f = g \circ d$. The vector filed~$v$ is obtained by pulling back the Euclidean gradient of~$g$
along~$d$.

We extend~$f$ and~$v$ to the rest of~$(M',\g)$ as a generic sutured function and a gradient-like vector field,
respectively. As in~\cite[Definition~6.14]{naturality}, choose a spanning tree~$T_\pm$ of the graph~$\Gamma_\pm(f,v)$
and a splitting surface~$\S' \in \S(f,v)$. This data gives rise to a diagram~$H(f,v,\S',T_-,T_+)$ of~$(M',\g)$
that we denote~$(\S',\bolda',\boldb)$.
Since~$s = W^s(o)$ corresponds to a loop edge of~$T_-$, it does not lie
in the spanning tree~$T_-$. Hence, $\a := U \cap \S' \subset \bolda'$ by definition.

By construction, $\S' \cap U = \a$. As~$U$ is the unstable manifold of~$o$ and~$\partial U \subset R_+(\g)$,
there is no flow-line from~$o$ to another critical point, and hence~$\a \cap \boldb = \emptyset$.
If follows that we can isotope~$\S'$ by ``stretching out'' a neighborhood of~$\a$ in~$(D^1 \times S^2) \setminus N$
such that it satisfies~\eqref{eqn:tube}, $\bolda \cap (D^1 \times S^1) = \a$, and~$\boldb \cap (D^1 \times S^1) = \emptyset$.
If~$\b$ is a small isotopic translate of~$\a$, then
$(\S',\bolda',\boldb \cup \{\b\})$ is a diagram of~$(M(\SS),\g)$. We let~$\boldb' := \boldb \cup \{\b\}$.
The framing of~$\SS'$ is given by a curve in~$\partial N(\SS')$, flowing this up along~$v$ we obtain a curve~$\d$ in~$\S'$
that intersects both~$\a$ and $\b$ in a single point. Hence $(\S',\bolda',\boldb \cup \{\d\})$ is a diagram
of~$(M(\SS)(\SS'),\g)$.

We write~$A = D^1 \times S^1$, and let~$c$ be a component of~$s(\g)$.
Pick disjoint embedded arcs~$\eta_\pm \subset \S' \setminus \text{Int}(A)$ connecting~$\{\pm 1\} \times S^1$ with~$c$.
We can achieve that
\[
\eta_\pm \cap (\bolda \cup \boldb) = \emptyset
\]
by handlesliding each attaching curve
over~$\a$ or~$\b$ along~$\eta_\pm$. Let~$\boldd$ be a small Hamiltonian translate of~$\boldb$,
and we write~$\boldd' = \boldd \cup \{\d\}$.
By the proof of~\cite[Lemma~3.12]{GW}, we can achieve
that the triple $(\S',\bolda',\boldb',\boldd')$ is admissible
by winding the attaching curves in the complement of~$A$.
We do not need to wind in~$A$, since after connecting~$\eta_+$ and~$\eta_-$
in~$A$, we obtain a properly embedded arc~$\eta$ in~$\S'$ dual to~$\b$
that only intersects~$\a$ and~$\b$, and these two curves already
intersect each other in two points so no further winding is required
along~$\eta$.

We set
\[
\S := (\S' \setminus \text{Int}(A)) \cup (\partial D^1 \times D^2),
\]
where $\partial D^1 \times D^2$ is given by the framing of~$\SS$
that identifies a neighborhood of each point of~$\SS$ with~$D^3$.
Then $(\S,\bolda,\boldb)$ is a diagram of~$(M,\g)$.
We denote $\{\pm 1\} \times D^2$ by~$D_\pm$.
The surface~$\S'$ gives a framing of the arcs~$\eta_+$ and~$\eta_-$,
which then give rise to a bouquet~$B(\SS)$ for~$\SS$.
By construction, the diagram $(\S,\bolda,\boldb)$ is subordinate to~$B(\SS)$.

In summary, we have obtained an admissible triple diagram
\[
(\S',\bolda',\boldb',\boldd') = (\S^0 \cup A, \bolda \cup \{\a\},\boldb \cup \{\b\},\boldd \cup \{\d\}), \text{ where}
\]
\begin{itemize}
\item $(\S,\bolda,\boldb)$ is a balanced diagram of $(M,\g)$ subordinate the bouquet~$B(\SS)$,
\item there is a component~$c$ of $\partial \S$, and disks~$D_-$ and~$D_+$ lying in the component of
$\S \setminus (\bolda \cup \boldb)$ containing $c$, such that $\S^0 = \S \setminus (D_- \cup D_+)$,
\item $A = D^1 \times S^1 \subset M(\SS)$ is an annulus attached to $\S^0$ along $\partial D_- \cup \partial D_+$,
\item the curve $\a = \{0\} \times S^1 \subset A$, and $\b$ is a small Hamiltonian translate of $\a$,
\item the attaching set~$\boldd$ is a small Hamiltonian translate of $\boldb$, and $\d$ intersects both $\a$ and $\b$ transversally in a single point,
\item the triple diagram $(\S',\bolda',\boldb',\boldd')$ is subordinate to the bouquet~$B(\SS')$ for~$\SS'$,
in such a way that $\b$ is a meridian of $\SS'$, and $\d$ represents the framing of $\SS'$.
\end{itemize}
For an illustration, see Figure \ref{fig:4}. Since~$D_-$ and~$D_+$ lie in the same component
of $\S \setminus (\bolda \cup \boldb)$ as~$c$, there is an arc~$a$ that connects
a point~$p$ of~$\delta$ with~$c$, and whose interior lies in $\S^0 \setminus (\bolda' \cup \boldb' \cup \boldd')$. Indeed,
connect~$\partial D_-$ to~$c$ with an arc~$a'$ inside~$\S \setminus (\bolda \cup \boldb)$ that intersects~$\d$ transversely,
and take~$a$ to be the closure of the last component of $a' \setminus \d$. Since $a \cap \boldb = \emptyset$
and~$\boldd$ is a small Hamiltonian
translate of~$\boldb$, we can also achieve that $a \cap \boldd = \emptyset$.

\begin{figure}[tb]
\includegraphics{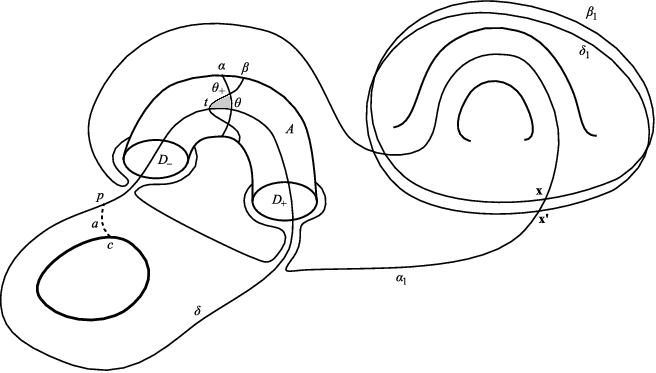}
\caption{A triple diagram corresponding to a canceling pair of one and two-handles.} \label{fig:4}
\end{figure}

We can also assume that $\d \cap \boldb = \emptyset$, since $\boldd$ is a small Hamiltonian translate of~$\boldb$.
Next, we achieve that $\d \cap \bolda = \emptyset$. Just push the curves in $\bolda$ that intersect~$\d$ simultaneously using a finger move along the
arcs $\d \setminus (A \cup \{p\})$ towards $\partial D_- \cup \partial D_+$, then handleslide them over~$\a$.
This process can be done away from the arc~$a$. Furthermore, since $\d \cap \boldb = \emptyset$, the new $(\S,\bolda,\boldb)$
differs from the old one by isotoping the~$\a$ curves inside~$\S \setminus \boldb$. So we can suppose that we started with
a triple diagram where~$\d \cap \bolda = \emptyset$.

We use the diagram
\[
(\S',\bolda',\boldb') = (\S^0 \cup A, \bolda \cup \{\a\},\boldb \cup \{\b\})
\]
to define the map $F_{M,\SS}$, and we compute $F_{M(\SS),\SS'}$
using the triple $(\S',\bolda',\boldb',\boldd')$.
If $\a \cap \b = \{\theta_+,\theta_-\}$, where $\theta_+$ has higher relative grading, then
$F_{M,\SS}(\x) = \x \times \{\theta_+\}$.
The fact that the arc $a$ connects $\d$ and $\partial \S$ and
its interior avoids $\bolda' \cup \boldb' \cup \boldd'$ ensures that every domain in the triple diagram $(\S',\bolda',\boldb',\boldd')$ has multiplicity zero on one side of $\d \setminus A$. Here, we can talk about the two sides of $\d \setminus A$ since
$\d \cap (\bolda \cup \boldb) = \emptyset$. Note that the only component of $\bolda' \cup \boldb' \cup \boldd'$ that intersects $\partial D_- \cup \partial D_+$ is $\delta$.  Moreover, both $\partial D_- \setminus \d$ and $\partial D_+ \setminus \d$ are connected, so every domain in the triple diagram $(\S',\bolda',\boldb',\boldd')$
is the disjoint union of a domain supported in $A$ and a domain supported in $\S^0$; i.e., it has zero multiplicity along
$\partial D_- \cup \partial D_+$. If $\a \cap \d = \{\theta\}$ and $\b \cap \d = \{t\}$, then there is a unique Maslov index zero triangle
connecting $\theta_+$, $t$, and $\theta$ that contributes to the composite map in~$A$, which is shaded in Figure~\ref{fig:4}.
Hence the composite map induces the same map on homology as the map
\[
CF(\S,\bolda,\boldb) \to CF(\S',\bolda',\boldd')
\]
given by $\x \mapsto f_{\a',\b',\d'}((\x \times \{\theta_+\}) \otimes \Theta_{\b',\d'})
= f_{\a,\b,\d}(\x \otimes \Theta_{\b,\d}) \times \{\theta\}$ for $\x \in \T_\a \cap \T_\b$.

We now compute the map~$F(\varphi)$. Recall that~$\varphi$ is the identity on the intersection~$M \cap M(\SS)(\SS')$.
As~$\S' \cap \SS' = \emptyset$ and
\[
\S' \cap M = \S' \setminus A = \S \setminus (D_+ \cup D_-),
\]
it follows that
\[
\S \cap (M \cap M(\SS)(\SS')) = \S \setminus (D_+ \cup D_-).
\]
But~$\bolda$ and~$\boldb$ are disjoint from~$D_\pm$,
so~$\varphi(\bolda) = \bolda$, $\varphi(\boldb) = \boldb$, and
\[
\varphi(\S) \cap \S \supset \S \setminus (D_+ \cup D_-).
\]
Let
\[
\H = (\varphi(\S),\varphi(\bolda),\varphi(\boldb)) = (\varphi(\S),\bolda,\boldb)
\]
and~$\H' = (\S',\bolda',\boldd')$.
For a homology class $x \in SFH(\S,\bolda,\boldb)$ represented by an intersection point $\x \in \T_\a \cap \T_\b$,
we obtain $F(\varphi)(x)$ by taking the homology class represented by~$\varphi(\x)$ in~$SFH(\H)$,
then finding its image under the natural isomorphism with $SFH(\H')$.
As explained in Section~\ref{sec:naturality}, we obtain the canonical isomorphism~$F_{\H,\H'}$
by connecting~$\H$ and~$\H'$ by a sequence of elementary moves, and composing the isomorphisms
induced by these. We can get from~$\H$ to~$\H'$ by first performing a small Hamiltonian isotopy
of~$\boldb$ to~$\boldd$, then stabilizing by adding the tube~$A$, both of whose feet are in the same component
of $\varphi(\S) \setminus (\bolda \cup \boldd)$ containing~$c \subset \partial \varphi(\S)$, and
adding the curves~$\a$ and~$\d$. As we can connect the feet of~$A$ in the complement of the attaching
curves, this is the same as taking the connected sum of the diagram with a punctured torus.
The isotopy induces the map $\Psi^{\bolda}_{\boldb \to \boldd}(\x) = f_{\a,\b,\d}(\x \otimes \Theta_{\b,\d})$
given by the same triangle count as~$F_{M(\SS),\SS'}$ in the complement of~$A$.
The stabilization then maps~$f_{\a,\b,\d}(\x \otimes \Theta_{\b,\d})$ to~$f_{\a,\b,\d}(\x \otimes \Theta_{\b,\d}) \times \{\theta\}$. This shows that
\[
F_{M(\SS),\SS'} \circ F_{M,\SS} = F(\varphi),
\]
and concludes the proof of Lemma~\ref{lem:onetwo}.
\end{proof}

\begin{lem}
$\SFH$ and the surgery maps~$F_{M,\SS}$
satisfy axiom~\eqref{it:birth} when~$\SS$ is a framed $1$-sphere and $\SS'$ is a framed $2$-sphere.
\end{lem}

\begin{proof}
The proof follows from ``turning around'' the proof of Lemma~\ref{lem:onetwo}.
\end{proof}

Having shown that $\SFH$ and the surgery maps~$F_{M,\SS}$ satisfy all the axioms,
this concludes the proof of Theorem~\ref{thm:welldef}.
\end{proof}

\subsection{$\spinc$ structures}
To show that the the cobordism maps split along $\spinc$ structures, we need to be more careful,
since there is no canonical way of composing two $\spinc$ cobordisms. As in the work of Ozsv\'ath and Szab\'o~\cite{OSz10},
the solution is to attach all the 2-handles simultaneously, but then we also need to consider handleslides
among them.

However, 1- and 3-handles cause no trouble, as we shall now see.
Recall from Lemma~\ref{lem:spec} that, for a special cobordism~$\W$, it does not matter whether we take
the $\spinc$ structures relative to~$\partial Z$ or~$Z$. Hence, in the latter case, if $\W = \W_1 \circ \W_2$,
then we can restrict any $\spinc$ structure to~$\W_1$ and~$\W_2$.

\begin{lem} \label{lem:spinc-1-3}
Let~$\W = (W,Z,\xi)$ be a special cobordism from~$(M_0,\g_0)$ to~$(M_1,\g_1)$ that is diffeomorphic
to the trace of attaching 1-handles to~$(M_0,\g_0)$. Then the restriction map
$\spinc(\W) \to \spinc(M_1,\g_1)$ is injective. If~$\W$ is a 3-handle cobordism, then the
restriction map $\spinc(\W) \to \spinc(M_0,\g_0)$ is injective.
\end{lem}

\begin{proof}
We only treat the case when~$\W$ is diffeomorphic to the trace of attaching 1-handles to $(M_0,\g_0)$,
as the other case follows by ``turning the cobordism upside down.''
This follows from the exact sequence of the triple $(W,M_1,\partial M_1)$:
\[
H^2(W,M_1) \to H^2(W, \partial M_1) \cong H^2(W,Z) \to H^2(M_1,\partial M_1),
\]
and the fact that $H^2(W,M_1) \cong H_2(M_0 \cup Z) \cong H_2(W,M_0) = 0$.
\end{proof}

Note that, in the first case, $\ft \in \spinc(M_1,\g_1)$ extends to~$\W$ if and only if $c_1(\ft)$
vanishes on the belt spheres of all the 1-handles attached to~$(M_0,\g_0)$.
In the second case, $\ft$ extends to~$\W$ if and only if $c_1(\ft)$
vanishes on the belt spheres of the 1-handles in the reversed cobordism~$\ol{\W'}$.

\begin{lem} \label{lem:spinc-compose}
Let~$\W = (W,Z,\xi)$ be a special cobordism from~$(M_0,\g_0)$ to~$(M_1,\g_1)$, and let~$\W' = (W',Z',\xi')$
be a special cobordism from~$(M_1,\g_1)$ to~$(M_2,\g_2)$.
If~$\W$ or~$\W'$ is diffeomorphic to the trace of attaching 1- or 3-handles to~$(M_0,\g_0)$, and given $\spinc$ structures
$\s \in \spinc(\W)$ and $\s' \in \spinc(\W')$ such that $\s|_{M_1} = \s'|_{M_1}$,
then there is a unique $\spinc$ structure~$\s''$ on the composite~$\W' \circ \W$ such that $\s''|_{\W} = \s$ and $\s''|_{\W'} = \s'$.
\end{lem}

\begin{proof}
We only treat the case when~$\W$ is the trace of attaching 1- or 3-handles to $(M_0,\g_0)$,
as the other case follows by ``turning $\W$ upside down.''
Consider the following relative Mayer-Vietoris sequence:
\[
\begin{split}
\dots \to H^1(W,Z) \oplus H^1(W',Z') &\stackrel{j}{\to} H^1(M_1,\partial M_1) \stackrel{\d}{\to} \\
H^2(W \cup W', Z \cup Z')  \to H^2(W,Z) \oplus H^2(W',Z') &\to \dots.
\end{split}
\]
Then the gluing of~$\s$ and~$\s'$ is unique if and only if~$\d = 0$, or equivalently, if~$j$ is surjective.
As $j(a,b) = a|_{M_1} - b|_{M_1}$, it suffices to show that the map $H^1(W,Z) \to H^1(M_1,\partial M_1)$
is onto. This again follows from the long exact sequence of the triple $(W,M_1,\partial M_1)$:
\[
H^1(W, \partial M_1) \cong H^1(W,Z) \to H^1(M_1,\partial M_1) \to H^2(W, M_1),
\]
and the fact that $H^2(W,M_1) \cong H_2(W,M_0) = 0$ using cellular homology, as~$W$ only
consists of 1- or 3-handles attached to~$M_0$.
\end{proof}

\begin{defn} \label{def:splits}
We say that a TQFT $F \colon \BSut' \to \Vect$ \emph{splits along $\spinc$ structures} if,
for every sutured manifold~$(M,\g)$ and $\ft \in \spinc(M,\g)$, there is a group~$F(M,\g,\ft)$,
and for every special cobordism~$\W$ from~$(M_0,\g_0)$ to~$(M_1,\g_1)$ and $\spinc$ structure $\s \in \spinc(\W)$, there
is a linear map
\[
F_{\W,\s} \colon F(M,\g,\s|_{M_0}) \to F(M,\g,\s|_{M_1})
\]
that satisfy the following properties:
\begin{enumerate}
\item For every sutured manifold~$(M,\g)$,
\[
F(M,\g) = \bigoplus_{\ft \in \spinc(M,\g)} F(M,\g,\ft),
\]
\item Given a diffeomorphism $d \colon (M,\g) \to (M',\g')$ and $\ft \in \spinc(M,\g)$,
the image of
\[
F(d,\ft) := F(d)|_{\SFH(M,\g,\ft)}
\]
lies in $F(M,\g,d_*(\ft))$.
\item For every special cobordism $\W$,
\[
F_\W = \sum_{\s \in \spinc(M,\g)} F_{\W,\s}.
\]
\item Given a diffeomorphism~$D$ from the $\spinc$ special cobordism
\[
(\W,\s) \colon (M_0,\g_0,\ft_0) \to (M_1,\g_1,\ft_1)
\]
to the $\spinc$ special cobordism
\[
(\W',\s') \colon (M_0',\g_0',\ft_0') \to (M_1',\g_1',\ft_1'),
\]
the following diagram is commutative:
\[
\xymatrixcolsep{3pc} \xymatrix{
  F(M_0,\g_0,\ft_0) \ar[r]^-{F_{\W,\s}} \ar[d]^-{F(D|_{M_0},\ft_0)} & F(M_1,\g_1,\ft_1) \ar[d]^{F(D|_{M_1},\ft_1)} \\
  F(M_0',\g_0',\ft_0')  \ar[r]^-{F_{\W',\s'}} & F(M_1',\g_1',\ft_1').}
\]
\item \label{it:compose-s} Let $\W_1$ be a special cobordism from $(M_0,\g_0)$ to $(M_1,\g_1)$, and $\W_2$ a special cobordism from $(M_1,\g_1)$ to $(M_2,\g_2)$,
and set $\W = \W_2 \circ \W_1$. Fix $\spinc$ structures $\s_i \in \spinc(\W_i)$ for $i \in \{1,2\}$ such that $\s_1|_{M_1} = \s_2|_{M_1}$.
Then
\[
F_{\W_2,\s_2} \circ F_{\W_1,\s_1} = \sum_{\{\s \in \spinc(\W) \colon \s|_{\W_1} = \s_1, \, \s|_{\W_2} = \s_2\}} F_{\W,\s}.
\]
\end{enumerate}
\end{defn}

\begin{defn}
Let $\L$ be a framed link in the sutured manifold~$(M,\g)$. We are given
an embedded framed arc~$a$ connecting two distinct components~$L_i$ and~$L_j$ of~$\L$
whose interior is disjoint form~$\L$, and whose framing at~$\partial a$ is tangent to~$\L$.
We say that~$\L'$ is obtained form~$\L$
by \emph{handlesliding~$L_i$ over~$L_j$ along~$a$} if $\L' = (\L \setminus L_i) \cup L_i'$, where~$L_i'$
is the band sum of~$L_i$ and a push-off of~$L_j$ in the direction of its framing, taken along~$a$.
\end{defn}

Corresponding to the above handleslide,
there is an ambient isotopy~$\{d_t \colon t \in I\}$ in~$M(\L \setminus L_i)$ such that $d_1(L_i) = L_i'$.
This consists of a finger move of~$L_i$ along~$a$, continued by a radial
isotopy across the center of the disk $\{a(1)\} \times D^2$ in the handle~$D^2 \times D^2$
attached along~$L_j$, and then radially ``blowing out'' the portion of the curve in~$\{a(1)\} \times D^2$
until it reaches~$\{a(1)\} \times S^1$. The space of such isotopies is contractible.
Then~$d_t$ extends to an isotopy~$D_t$ of~$W(\L \setminus L_i)$ that is the identity outside
a collar neighborhood of~$M(\L \setminus L_i)$. Furthermore,
$d_1$ induces a diffeomorphism $d_1^{L_i} \colon M(\L) \to M(\L')$, and~$D_1$ induces a diffeomorphism
$D_1^{L_i} \colon W(\L) \to W(\L')$ in a natural manner.

We have the following refinement of Theorem~\ref{thm:bsut}. Note that here we also use the symbol~$\SS$
to denote a framed link.

\begin{thm} \label{thm:bsut-s}
To define a functor $F \colon \BSut' \to \Vect$ that splits along $\spinc$ structures, it suffices to construct
a functor~$F$ from the category of $\spinc$ balanced sutured manifolds and diffeomorphisms to~$\Vect$,
and for every balanced sutured manifold~$(M,\g)$, framed 0-sphere, link, or 2-sphere~$\SS \subset \text{Int}(M)$,
and $\spinc$ structure $\s \in \spinc(W(\SS))$,
a linear map
\[
F_{M,\SS,\s} \colon F(M,\g,\s|_M) \to F(M(\SS),\g,\s|_{M(\SS)})
\]
that satisfy the following axioms:
\begin{enumerate}
\item \label{it:isot-s} %\label{it:id}
If $d \in \Diff_0(M,\g)$, then $F(d,\s) = \Id_{F(M,\g,\s)}$.
\item \label{it:d-F-s}
Consider a diffeomorphism $d \colon (M,\g) \to (M',\g')$ between balanced sutured manifolds,
a framed 0-sphere, link, or 2-sphere $\SS \subset \text{Int}(M)$, and a $\spinc$ structure~$\s \in \spinc(W(\SS))$.
Let $\SS' = d(\SS)$, let
\[
d^\SS \colon (M(\SS),\g') \to (M'(\SS'),\g') \text{ and } D^\SS \colon W(\SS) \to W(\SS')
\]
be the induced diffeomorphisms, and write $\s' = D^\SS(\s) \in \spinc(W(\SS'))$.
Then the following diagram is commutative:
\[
\xymatrixcolsep{3pc} \xymatrix{
  F(M,\g,\s|_M) \ar[r]^-{F_{M,\SS,\s}} \ar[d]^-{F(d)} & F(M(\SS),\g,\s|_{M(\SS)}) \ar[d]^{F(d^\SS)} \\
  F(M',\g',\s'|_{M'})  \ar[r]^-{F_{M',\SS',\s'}} & F(M'(\SS'),\g',\s'|_{M'(\SS')}).}
\]
\item \label{it:commut-s} If $(M,\g)$ is a balanced sutured manifold and $\SS$ and $\SS'$ are \emph{disjoint} framed
0-spheres, links, or 2-spheres in $M$,
then $M(\SS)(\SS') = M(\SS')(\SS)$, we denote this manifold by $M(\SS,\SS')$.
There is an isotopically unique equivalence
\[
\W := W(M(\SS),\SS') \circ W(M,\SS) \cong W(M(\SS'),\SS) \circ W(M,\SS').
\]
Given $\s \in \spinc(\W)$, this equivalence allows us to define $\s_\SS = \s|_{W(M,\SS)}$, $\s_{\SS'} = \s|_{W(\SS')}$,
$\s_{\SS,\SS'} = \s|_{W(M(\SS),\SS')}$, and $\s_{\SS',\SS} = \s|_{W(M(\SS'),\SS)}$.
If at least one of $\SS$ and~$\SS'$ is not 1-dimensional, then the following diagram is commutative:
\[
\xymatrixcolsep{5pc}\xymatrix{
  F(M,\g,\s|_M) \ar[r]^-{F_{M,\SS,\s_{\SS}}} \ar[d]_-{F_{M,\SS',\s_{\SS'}}} & F(M(\SS),\g,\s|_{M(\SS)}) \ar[d]^{F_{M(\SS),\SS'},\s_{\SS,\SS'}} \\
  F(M(\SS'),\g,\s|_{M(\SS')})  \ar[r]^-{F_{M(\SS'),\SS},\s_{\SS',\SS}} & F(M(\SS,\SS'),\g,\s|_{M(\SS,\SS')}).}
\]
\item \label{it:splitlink} Let~$\L = \L_1 \cup \L_2$ be a partition of a framed link in~$M$,
and write~$\W_1 = W(M,\L_1)$, $\W_2 = W(M(\L_1), \L_2)$,
and $\W = \W_2 \circ \W_1 \cong W(\L)$. If~$\s_1 \in \spinc(\W_1)$ and $\s_2 \in \spinc(\W_2)$, then
\[
F_{M(\L_1),\L_2,\s_2} \circ F_{M,\L_1,\s_1} = \sum_{\s \in \spinc(\W) \colon \s|_{\W_1} = \s_1 \text{, } \s|_{\W_2} = \s_2} F_{M,\L,\s}.
\]
\item \label{it:birth-s} If $\SS' \subset M(\SS)$ intersects the belt sphere of the handle attached along~$\SS$ once transversely,
then there is a diffeomorphism $\varphi \colon M \to M(\SS)(\SS')$ (which is the identity on $M \cap M(\SS)(\SS')$,
and is unique up to isotopy),
such that for every $\s \in \spinc(W(\SS') \circ W(\SS))$, if we write $\s_\SS = \s|_{W(\SS)}$
and $\s_{\SS'} = \s|_{W(\SS')}$, then $\s|_{M(\SS)(\SS')} = \varphi_*(\s|_M)$, and
\[
F_{M(\SS),\SS',\s_{\SS'}} \circ F_{M,\SS,\s_\SS} = F(\varphi,\s|_M),
\]
where $F(\varphi,\s|_M)$ is the restriction of~$F(\varphi)$ to $SFH(M,\g,\s|_M)$.
\item \label{it:handleslide} The 2-handle maps are invariant under handleslides. More precisely,
suppose that~$\L$ is a framed link in~$(M,\g)$, and $\L'$ is obtained by sliding~$L_i$ over~$L_j$ along
a framed arc~$a$. Let $\s \in \spinc(W(\L))$, and let $\s' = (D_1^{L_i})_*(\s)$, $\ft = \s|_M$,
$\ft_\L = \s|_{M(\L)}$, and $\ft_{\L'} = \s'|_{M(\L')}$. Then the following diagram is commutative:
\[
\xymatrixcolsep{5pc}\xymatrix{
F(M,\g,\ft) \ar[r]^{\Id} \ar[d]_{F_{M,\L,\s}} & F(M,\g,\ft) \ar[d]^{F_{M,\L',\s'}} \\
F(M(\L),\g,\ft_\L) \ar[r]^{F\left(d_1^{L_i}, \ft_\L \right)} & F(M(\L'),\g, \ft_{\L'}).}
\]
\item \label{it:0-sphere-s} For every framed 0-sphere, link, or 2-sphere~$\SS$ and
$\s \in \spinc(W(\SS)) = \spinc(W(\ol{\SS}))$, we have
$F_{M,\SS,\s} = F_{M,\ol{\SS},\s}$.
\end{enumerate}
The functor~$F$ is a TQFT if and only if it is symmetric and monoidal.
In the opposite direction, every functor $F \colon \BSut' \to \Vect$
that splits along $\spinc$ structures arises in this way.
\end{thm}

\begin{proof}
This is similar to the proof of~\cite[Theorem~1.2]{surgery}.
The additional idea is the following. Let~$\W = (W,Z,\xi)$ be a special cobordism
from~$(M_0,\g_0)$ to~$(M_1,\g_1)$, and fix a metric on~$W$.
We say that a Morse function~$f$ on~$W$ is \emph{nice} if
it is the projection $\partial M_0 \times I \to I$ on~$Z$,
has only critical points of index 1, 2, and 3, and such that all index 2 critical values lie
between the index 1 and 3 critical values. Furthermore, we require that all index~1
and~3 critical points have distinct values, and there is no gradient flow-line between
two index~2 critical points.

A nice Morse function gives rise to a type of parametrized Cerf decomposition (cf.~\cite{surgery})
where we first attach 1-handles, then attach all 2-handles simultaneously along a framed link,
and finally attach the 3-handles. More specifically:

\begin{defn} \label{def:param-Kirby}
A \emph{parameterized Kirby decomposition} of a special cobordism~$\W$ from~$(M,\g)$ to~$(M',\g')$ consists of
\begin{itemize}
\item a decomposition
\[
\W = \W_0 \circ \W_1 \circ \dots \circ \W_m,
\]
where each~$\W_i$ is a cobordism from~$(M_i,\g_i)$ to~$(M_{i+1},\g_{i+1})$; furthermore,
$(M_0,\g_0) = (M,\g)$ and $(M_{m+1},\g_{m+1}) = (M',\g')$,
\item there is a number $c \in \{1,\dots,m\}$
such that $\W_i$ is an elementary 1-handle cobordism for $i < c$ and is an elementary 3-handle cobordism
for $i  > c$, and~$\W_c$ is diffeomorphic to~$W(\L)$ for some framed link~$\L \subset M_c$,
\item an attaching sphere $\SS_i \subset M_i$ for~$W_i$, where $\SS_c = \L$,
\item a diffeomorphism $D_i \colon W(\SS_i) \to W_i$, well-defined up to isotopy,
such that $D_i(x,0) = x$ for~$x \in M_i$.
We write~$d_i$ for
\[
D_i|_{M_i(\SS_i)} \colon (M_i(\SS_i),\g_i) \to (M_{i+1},\g_{i+1}).
\]
\end{itemize}
\end{defn}

Let $\s \in \spinc(\W)$, and consider $\ft = \s|_M$ and $\ft' = \s|_{M'}$.
Given a parameterized Kirby decomposition~$\mathcal{K}$ of~$\W$,
we write~$\s_i$ for~$(D_i)_*^{-1}(\s|_{\W_i}) \in \spinc(W(\SS_i))$ for $i \in \{1,\dots,m\}$.
Then we define
\[
F(\W,\s,\mathcal{K}) = \prod_{i=0}^{m} \left( F(d_i) \circ F_{M_i,\SS_i,\s_i} \right) \colon F(M,\g,\ft) \to F(M',\g',\ft').
\]
This is consistent with the composition rule~\eqref{it:compose-s} of Definition~\ref{def:splits}
since~$\s|_{\W_c}$ uniquely determines~$\s$ according to Lemmas~\ref{lem:spinc-1-3}
and~\ref{lem:spinc-compose}. In particular, if~$F$ is a TQFT that splits along $\spinc$
structures, then
\[
F_{\W,\s} = \prod_{i = 0}^m F_{\W_i,\s|_{\W_i}}
\]
since~$\s$ is the only $\spinc$ structure on~$\W$ that restricts to~$\s|_{\W_i}$ for every $i \in \{1,\dots,m\}$.
Furthermore, if we compose two special cobordisms endowed with parameterized Kirby decompositions,
then we can move the 1-handles to the bottom
and the 3-handles to the top using axiom~\eqref{it:commut-s} of Theorem~\ref{thm:bsut-s},
then compose the 2-handle cobordisms maps via axiom~\eqref{it:splitlink}. Again,
Lemma~\ref{lem:spinc-compose} ensures that gluing $\spinc$ structures is not
unique only in the case of two 2-handle cobordisms.

It remains to show that $F(\W,\s,\mathcal{K})$ is independent of~$\mathcal{K}$.
Let~$f_0$ and~$f_1$ be nice Morse functions on~$W$,
with associated parameterized Kirby decompositions~$\mathcal{K}_0$ and~$\mathcal{K}_1$, respectively.
By part~3 of the paper of Kirby on the calculus~\cite{Kirby},
there is a path of smooth function~$\{f_t \colon t \in I\}$ connecting~$f_0$ and~$f_1$
such that it is through nice Morse functions, except for a finite number of parameter values,
when there is either an index 1-2 birth-death
between the index 1 and 2 Morse critical values,
or an index 2-3 birth-death between the index 2 and 3 Morse critical values.
Furthermore, two index~1 or two index~3 critical points might have the same value,
or there might be a gradient flow-line between two index~2 critical points.

These bifurcations give rise to a set of moves connecting~$\mathcal{K}_0$
and~$\mathcal{K}_1$, where  one has to keep track of the diffeomorphisms~$D_i$, not just~$d_i$.
Invariance of $F(\W,\s,\mathcal{K})$ under birth-death bifurcations
follows from axiom~\eqref{it:birth-s}. Note that if $\W = \W_3 \circ \W_2 \circ \W_1$,
and~$\W_1$ is an elementary 1-handle cobordism canceling the elementary 2-handle cobordism~$\W_2$,
and~$\W_3$ is a 2-handle cobordism, then every~$\s_3 \in \spinc(\W_3)$ uniquely extends to a
$\spinc$ structure $\s \in \spinc(\W)$ as~$\W_1 \circ \W_2$ is a product cobordism, and
hence $\s|_{\W_1}$ and~$\s|_{\W_2}$ are also uniquely determined by~$\s_3$. So there is no
ambiguity in gluing $\spinc$ structures when checking handle cancelation invariance.
An analogous argument applies for canceling a 2- and a 3-handle.
Invariance under a 2-handle slide, corresponding to a flow-line between index~2 critical points,
follows from axiom~\eqref{it:handleslide}.
Invariance under two index~1 or two index~3 critical points passing each other follows from
axiom~\eqref{it:commut-s}. If there are no bifurcations, the Kirby decompositions change
by isotoping the diffeomorphisms~$d_i$ and the attaching spheres~$\SS$, and invariance under these
follows from axioms~\eqref{it:isot-s} and~\eqref{it:d-F-s}.
From here, we proceed as in \cite[Theorem~1.2]{surgery}.
\end{proof}

\begin{thm} \label{thm:welldef-s}
The functor~$\SFH$ together with the surgery maps~$F_{M,\bP,\s}$, $F_{M,\L,\s}$, and~$F_{M,S,\s}$
satisfy the axioms listed in Theorem~\ref{thm:bsut-s}, hence they
give rise to a TQFT $F \colon \BSut' \to \Vect$ that splits along $\spinc$ structures.
\end{thm}

\begin{proof}
We proved that the maps~$F_{M,\L,\s}$ satisfy axiom~\eqref{it:splitlink} in Proposition~\ref{prop:linkcomp}.
Every other axiom follows essentially as in the proof of Theorem~\ref{thm:bsut},
except for axiom~\eqref{it:handleslide}, invariance of~$F_{M,\L,\s}$ under 2-handles slides:

We generalize the proof of \cite[Lemma 4.14]{OSz10}.
Let $K_1,\dots,K_n$ be the components of~$\L$,
and $K_1',\dots,K_n'$ the components of~$\L'$.
Suppose that $K_1'$ is obtained from $K_1$ by a handleslide over $K_2$ along a framed arc~$a$,
and $K_i' = K_i$ for $i \in \{2 ,\dots,n\}$.
After the handleslide, there is a natural path $a'$ joining $K_1'$ and $K_2.$

We construct a bouquet $B(\L)$ for $\L$ by picking arcs~$a_i$ connecting~$K_i$ to~$R_+(\g)$
for $i \in \{1,\dots,n\}$, as follows.
Isotope~$a$ fixing $\partial a$ until it intersects~$R_+(\g)$ in a single point~$p$.
Then choose the arcs~$a_1$ and~$a_2$ such that they run close to~$a$,
and they both end near $p$. We pick $a_3,\dots, a_n$ in an arbitrary way.
Let $(\S,\bolda,\boldb,\boldd)$ be a triple diagram subordinate to $B(\L)$, such that $\b_1$  is dual to~$K_1$
and~$\b_2$ is dual to~$K_2$. Let $(\S,\bolda,\boldb',\boldd')$ be the triple diagram where~$\b_2'$ is obtained by a handleslide
of~$\b_2$ over~$\b_1$ along $a$, and~$\g_1'$ is obtained by a handleslide of~$\g_1$ over~$\g_2$ along~$a$,
while $\b_i' = \b_i$ and $\g_i' = \g_i$ for $i \in \{2,\dots,n\}$.
Then $(\S,\bolda,\boldb',\boldd')$ is subordinate to a bouquet~$B(\L')$ for~$\L'$, constructed using~$a'$.

We then have the commutative diagram
\[
\begin{CD}
\SFH(\S,\bolda,\boldb,\ft) @>\otimes\t_{\b,\b'}>> \SFH(\S,\bolda,\boldb',\ft)\\
@VV\otimes \t_{\b,\d}V   @VV\otimes \t_{\b',\d'}V\\
\SFH(\S,\bolda,\boldd,\ft_{\L}) @>\otimes \t_{\d,\d'}>> \SFH(\S,\bolda,\boldd',\ft_{\L'}),
\end{CD}
\]
where the vertical arrows only count triangles representing the $\spinc$ structures~$r(\s)$
and~$r(\s')$, respectively.
Commutativity follows from associativity, and the observation that
\[
F_{\b,\b',\d'}(\t_{\b,\b'} \otimes \t_{\b',\d'},\s_0) = \t_{\b,\d'} =  F_{\b,\d,\d'}(\t_{\b,\d} \otimes \t_{\d,\d'},\s_0),
\]
according to the handleslide invariance of the homology groups.
This concludes the proof of Theorem~\ref{thm:welldef-s}.
\end{proof}

\begin{prop} \label{prop:specialsum}
Let $\W$ be a special cobordism. Then there are only finitely many $\s \in \spinc(\W)$ for which $F_{\W,\s} \neq 0$, and
\[
F_{\W} = \sum_{\s \in \spinc(\W)} F_{\W,\s}.
\]
\end{prop}

\begin{proof}
Write $\W = \W_3 \circ \W_2 \circ \W_1$, where~$\W_i$ is a composition of $i$-handle cobordisms for~$i \in \{1,2,3\}$.
If $\s$, $\s' \in \spinc(\W)$ satisfy $\s|_{\W_2} = \s'|_{\W_2}$, then $\s = \s'$ according to Lemmas~\ref{lem:spinc-1-3}
and~\ref{lem:spinc-compose}.
So the claim follows from Proposition~\ref{prop:fin}.
\end{proof}

In the following two results, we explicitly spell out two of the consequences of Theorem~\ref{thm:bsut-s}
stating that~$\SFH$ is a TQFT that splits along $\spinc$ structures.

\begin{prop} \label{prop:specialid}
Let $(M,\g)$ be a balanced sutured manifold. If $\W = (W,Z,[\xi])$ is the trivial cobordism from $(M,\g)$ to $(M,\g)$, then $F_{\W}$ is the identity of $\SFH(M,\g)$.
Furthermore, the restriction map from $\spinc(\W)$ to $\spinc(M,\g)$ is an isomorphism, and for every $\s \in \spinc(\W)$, the map $F_{\W,\s}$ is the identity of $\SFH(M,\g,\s|_{M})$.
\end{prop}

\begin{proof}
Since $W = M \times I$, $Z = \partial M \times I$, and $\xi$ is $I$-invariant, $\spinc(\W)$ and $\spinc(M,\g)$ are obviously
isomorphic. The rest follows from the fact that there is a relative handle decomposition of $\W$ with no handles at all.
If~$\L = \emptyset$, then both~$F_{M,\L}$ and~$F_{M,\L,\s}$ are identity maps, as they are defined via a triple diagram~$(\S,\bolda,\boldb,\boldd)$,
where~$\d_i$ is a small Hamiltonian translate of $\b_i$ for every $i \in \{1 ,\dots, d\}$.
\end{proof}

%We have the following analogue of \cite[Theorem 3.4]{OSz10}.

\begin{thm} \label{thm:specialcomp}
Let $\W_1$ be a special cobordism from $(M_0,\g_0)$ to $(M_1,\g_1)$, and $\W_2$ a special cobordism from $(M_1,\g_1)$ to $(M_2,\g_2)$,
and set $\W = \W_2 \circ \W_1$. Fix $\spinc$ structures $\s_i \in \spinc(\W_i)$ for $i \in \{1,2\}$ such that $\s_1|_{M_1} = \s_2|_{M_1}$.
Then
\[
F_{\W_2,\s_2} \circ F_{\W_1,\s_1} = \sum_{\{\s \in \spinc(\W) \colon \s|_{\W_1} = \s_1, \, \s|_{\W_2} = \s_2\}} F_{\W,\s}.
\]
Moreover, $F_{\W} = F_{\W_2} \circ F_{\W_1}$.
\end{thm}

\section{The contact invariant $\EH$ and the gluing map $\Phi_{\xi}$} \label{sub:glue}

Here, we review the necessary definitions and result from \cite{contact}, \cite{Ozbagci}, and \cite{TQFT}, and enrich the theory with $\spinc$ structures.
We will prove the necessary naturality results in a separate paper.

\subsection{The contact invariant $\EH$}
Suppose that $(M,\xi)$ is a contact 3-manifold with convex boundary and dividing set $\g$ on $\partial M$. We denote such a contact
manifold by~$(M,\g,\xi)$. Honda, Kazez, and Mati\'c~\cite{contact} defined an invariant of~$(M,\g,\xi)$
which is an element~$\EH(M,\g,\xi)$ of~$\SFH(-M,-\g)$, also see \cite{Ozbagci}.
We briefly review the construction.

\begin{defn}
A \emph{partial open book decomposition} is a pair $(S,h \colon P \to S)$, where
\begin{itemize}
\item $S$ is a compact oriented surface with $\partial S \neq \emptyset$, called the \emph{page},
\item $P \subset S$ is a compact subsurface, such that each component of $\partial P$ is polygonal with
consecutive sides $r_1,\dots,r_{2n}$, and $r_i \subset \partial S$ for $i$ even,
\item $h \colon P \to S$ is a diffeomorphism such that $h|_{P \cap \partial S} = \text{Id}$.
\end{itemize}
The partial open book $(S,h)$ \emph{defines} a contact manifold $(M,\g,\xi)$ as follows.
Let $M = S \times I/_{\sim_h}$, where $\sim_h$ is the equivalence relation
such that $(x,t) \sim_h (x,t')$ for all~$x \in \partial S$ and~$t \in I$, and~$(x,1) \sim_h (h(x),0)$ for all
$x \in P$. Furthermore, let~$R_+(\g) = \Int(S \setminus P) \times \{1\}$ and $R_-(\g) = \Int(S \setminus h(P)) \times \{0\}$,
whereas the sutures are~$\g = \partial R_+(\g) = \partial R_-(\g)$. To define~$\xi$, notice that one can obtain~$(M,\g)$ by gluing together
the product sutured manifolds $(S \times I,\partial S \times I)$ and $(P \times I,-\partial P \times I)$. Each of them
carries a unique product disc decomposable contact structure, $\xi$ is obtained by gluing these together.
\end{defn}

\begin{figure}[tb]
\includegraphics{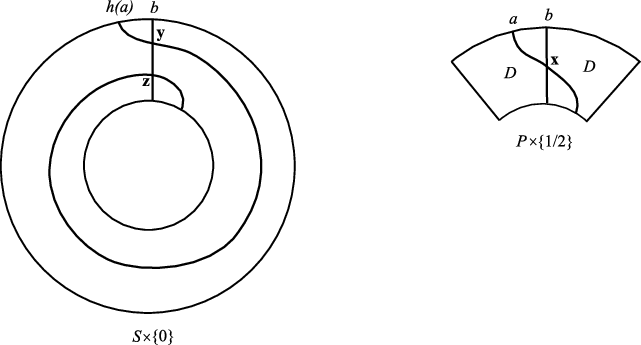}
\caption{A balanced diagram obtained from a partial open book decomposition.} \label{fig:5}
\end{figure}

Given a contact manifold $(M,\g,\xi)$, there exists a compatible partial open book decomposition $(S,h \colon P \to S)$
by~\cite[Theorem~1.3]{contact}. The closure of~$S \setminus P$ is naturally identified with~$R_+(\g)$. A set
$\{b_1,\dots,b_d\}$ of properly embedded, pairwise disjoint arcs in $P$ is called a \emph{basis for $(S,R_+(\g))$} if
$S \setminus \left(\bigcup_{i=1}^d b_i\right)$ deformation retracts onto~$R_+(\g)$. In fact, $\{b_1,\dots,b_d\}$
is a basis for $H_1(P,P \cap \partial S)$; fix such a basis.

For $i \in \{1 ,\dots, d\}$, let $a_i$ be an arc that is isotopic to $b_i$ by a small isotopy such that the following hold:
\begin{itemize}
\item The endpoints of $b_i$ are isotoped along~$\partial S$ in the direction given by the boundary orientation of~$S$.
\item The arcs~$a_i$ and~$b_i$ intersect transversely in one point in~$\Int(S)$.
\end{itemize}

Then we obtain a balanced diagram $(\S,\bolda,\boldb)$ defining $(M,\g)$ by setting
\[
\S = (S \times \{0\}) \cup -(P \times \{1/2\}),
\]
and taking
\[
\a_i = (a_i \times \{1/2\}) \cup (h(a_i) \times \{0\})
\]
and~$\b_i = \partial (b_i \times [0,1/2])$ for~$i \in \{1 ,\dots,d\}$, see Figure~\ref{fig:5}.
To see that the orientations match up, note that $\partial \S = \g$, and~$\S$ is oriented as the boundary of the~$\bolda$ compression body.

Let $x_i = (a_i \cap b_i) \times \{1/2\}$ for $i = 1,\dots,d$. Then $\x = (x_1,\dots,x_d)$ is a cycle in~$CF(-\S,\bolda,\boldb)$.
Indeed, if $D$ is the closure of a component of $-\S \setminus (\bolda \cup \boldb)$ such that $x_i \in \partial D$ and
$\partial D \cap \a_i$ points out of $x_i$, then $D \cap \partial \S \neq \emptyset$. So every domain emanating from $\x$ is forced to be zero. Note that $(-\S,\bolda,\boldb)$ defines $(-M,-\g)$. Hence $\x$ defines a class in $\SFH(-M,-\g)$, which is shown to be an invariant of the contact manifold $(M,\g,\xi)$ in \cite[Theorem 3.1]{contact}. This is the contact invariant $\EH(M,\g,\xi)$.

\begin{rem} \label{rem:EH}
Note that we changed the notation of~\cite{contact}, where they labeled our~$\a$ curves by~$\b$, our~$\b$ curves by~$\a$,
and our~$-\S$ by~$\S$. We did this to make the underlying orientation conventions more transparent. In fact, with the notation
of~\cite{contact}, one has $\x \in CF(\S,\boldb,\bolda)$, but their~$(\S,\bolda,\boldb)$ defines~$(M,-\g)$ instead of~$(M,\g)$.
If~$(\S,\bolda,\boldb)$ defined~$(M,\g)$, then $(\S,\boldb,\bolda)$ would define~$(-M,\g)$.
However, this is not a real issue as~$\SFH(-M,\g)$ and~$\SFH(-M,-\g)$ are canonically isomorphic by~\cite[Proposition~2.14]{decateg}.
\end{rem}

\begin{ex} \label{ex:partial}
To illustrate the above construction, let us review \cite[Example 2]{Ozbagci}. For an illustration, see Figure \ref{fig:5}.
In this example, the page $S$ of our partial open book decomposition is an annulus, and $P$ is a radial sector of~$S$.
The diffeomorphism~$h \colon P \to S$ is the restriction of a left-handed Dehn twist to $P$. This partial open book decomposition
$(S,h \colon P \to S)$ defines a contact
manifold $(M,\g,\xi)$. In the resulting chain complex $CF(-\S,\bolda,\boldb)$, there are three
generators; we denote them by $\x$, $\y$, and $\mathbf{z}$, as in Figure \ref{fig:5}.
Observe that $\partial \x = 0$, while $\partial \y = \x$ and~$\partial \mathbf{z} = \x$. Thus~$\SFH(-M,-\g) \cong \Z_2$ and
$\EH(M,\g,\xi) = [\x] = 0$. The sutured manifold~$(M,\g)$ is the once punctured sphere $S^3(1)$, and~$\xi$ is an overtwisted
contact structure on it.
\end{ex}

As described above \cite[Proposition 2.14]{decateg}, if the vector field~$v$ represents
a $\spinc$ structure on $(M,\g)$, then the same vector field
also represents a $\spinc$ structure on~$(-M,-\g)$. Hence there is a canonical identification
between~$\spinc(M,\g)$ and~$\spinc(-M,-\g)$, and we will not distinguish between the two.

\begin{prop} \label{prop:EHspinc}
Suppose that $(M,\g,\xi)$ is a contact manifold, and let $\s_{\xi} \in \spinc(M,\g)$ be the homology class of the vector field $\xi^{\perp}$. Then
\[
\EH(M,\g,\xi) \in \SFH(-M,-\g,\s_{\xi}).
\]
\end{prop}

\begin{proof}

\begin{figure}[tb]
\includegraphics{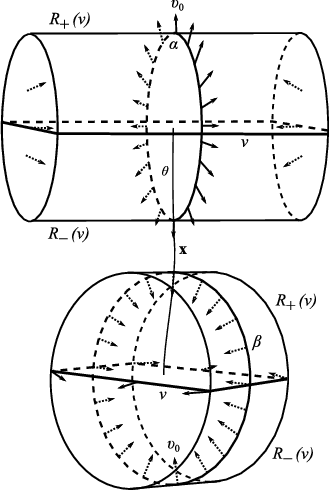}
\caption{A contact 1-handle above a contact 2-handle, together with the vector field $v_0$ used for constructing the relative $\spinc$ structures.} \label{fig:6}
\end{figure}

The partial open book used to define the contact class $\EH(M,\g,\xi)$ is constructed using a relative contact handle
decomposition of $(M,\g,\xi)$. One takes the product $R_-(\g) \times I$, and attaches~$d$ contact 1-handles, then~$d$ contact
2-handles. Suppose that $H$ is a contact 1- or 2-handle. This means that $\xi|_H$ is
contactomorphic to the unique tight contact structure on~$D^3$ with convex boundary and connected dividing set.
Denote by~$\nu$ the dividing set on~$\partial H$.

The handle decomposition gives a balanced diagram $(-\S,\bolda,\boldb)$ for $(-M,-\g)$, which agrees with the one arising from
the partial open book decomposition. The Heegaard surface $\S$ is obtained from $R_-(\g) \times \{1\}$ by performing the 1-handle
surgeries. If $H$ is a 1-handle, then its belt circle is an $\a$-curve that intersects~$\nu$ in exactly two points.
If~$H$ is a 2-handle, then its attaching circle is a $\b$-curve that also intersects $\nu$ in two points.
The contact element is represented by an intersection point $\x \in \T_{\a} \cap \T_{\b}$ such that~$\x \cap H$ lies
on~$R_-(\nu)$ if~$H$ is a 1-handle, and on~$R_+(\nu)$ if~$H$ is a 2-handle.

The construction of the $\spinc$ structure $\s(\x)$ associated to an intersection point~$\x$
coming from a relative handle decomposition is explained in Section~3.4
of~\cite{decateg}. A vector field~$v$ representing~$\s(\x)$ is obtained as follows. One takes the vector field $v_0$ on~$M$ that agrees
with~$\partial/\partial t$ on~$M \times I$. On a 1-handle
$H = D^1 \times D^2$, we take
\[
v_0 = -x \cdot \frac{\partial}{\partial x} + y \cdot \frac{\partial}{\partial y} + z \cdot \frac{\partial}{\partial z},
\]
while on a 2-handle $H = D^2 \times D^1$ one considers
\[
v_0 = -x \cdot \frac{\partial}{\partial x} - y \cdot \frac{\partial}{\partial y} + z \cdot \frac{\partial}{\partial z}.
\]
We also choose a smooth 1-chain $\theta$ in $M$ that connects the centers of the 1- and 2-handles, and passes through~$\x$.
Then one extends $v_0|_{M \setminus N(\theta)}$ to a nowhere zero vector field~$v$ on~$M$. Suppose that $\theta_1,\dots,\theta_d$ are the components
of~$\theta$, and~$\theta_i$ connects the centers of the handles~$H_{\a_i}$ and~$H_{\b_i}$ corresponding to~$\a_i$ and~$\b_i$, respectively,
for $i \in \{1,\dots,d\}$. Along~$\theta_i$, the vectors~$v_0$ point from the center of~$H_{\a_i}$ towards the
center of~$H_{\b_i}$, whereas~$v$ points in the opposite direction.

Now we show that~$v$ and~$\xi^{\perp}$ are homotopic over the 2-skeleton of~$M$; i.e., they are homologous. This will follow
if we prove that for every 1- or 2-handle~$H$, and for every~$p \in \partial H$, we have
\[
v_p \neq -(\xi^{\perp})_p.
\]
Indeed, on $R_-(\g) \times I$ both~$v$ and~$\xi^{\perp}$ point up. If~$v$ is generic, the set where~$v$ and~$\xi^{\perp}$ are
opposite represent the difference of the corresponding $\spinc$ structures. If each component of this difference cycle lies
in a handle, then it is obviously null-homologous.

First, observe that $\xi^{\perp}$ points into~$H$ along~$R_-(\nu)$, and points out of~$H$ along $R_+(\nu)$. Hence, one has
\[
(v_0)_p = -(\xi^{\perp})_p
\]
for exactly one point~$p \in \partial H$, which we can assume to be~$\x \cap H$. Indeed, if~$H$ is a 1-handle~$H_{\a_i}$, then
along its belt circle~$\a_i$ the field~$v_0$ points out, and $\a_i \cap R_-(\nu)$ consists of an arc whose ``midpoint'' $x_i$ is where
$v_0$ and~$\xi^{\perp}$ are opposite. Similarly, if~$H$ is a 2-handle~$H_{\b_i}$ with attaching circle~$\b_i$, then the midpoint~$x_i$
of the arc $\b_i \cap R_+(\g)$ is where~$v_0$ and~$\xi^{\perp}$ are opposite. However, as described above, $v$ is opposite to~$v_0$
along~$\theta$. In particular, $v$ and~$\xi^{\perp}$ point in the same direction at~$x_i$, and are never opposite elsewhere along the
2-skeleton.
\end{proof}

\subsection{The gluing map $\Phi_{\xi}$}
\begin{defn}
We say that $(M',\g')$ is a \emph{sutured submanifold} of the sutured manifold $(M,\g)$ if $M'$ is a submanifold with boundary of
$M$, and $M' \subset \Int(M)$. A connected component $C$ of $M \setminus \Int(M')$ is called \emph{isolated} if $C \cap \partial M = \emptyset$.
\end{defn}

Next, we recall \cite[Theorem 1.1]{TQFT}.

\begin{thm} \label{thm:glue}
Let $(M',\g')$ be a sutured submanifold of $(M,\g)$, and let $\xi$ be a contact structure on $M \setminus \Int(M')$ with convex boundary, and dividing set $\g$ on $\partial M$ and $\g'$ on $\partial M'$. If $M \setminus \Int(M')$ has $m$ isolated components, then $\xi$ induces a natural map
\[
\Phi_{\xi} \colon \SFH(-M',-\g') \to \SFH(-M,-\g) \otimes V^{\otimes m}.
\]
Here $V = \widehat{HF}(S^1 \times S^2) \cong \Z_2 \oplus \Z_2$ is a $\Z$-graded vector space, where the two summands have gradings that differ by one, say $0$ and $1$.

Moreover, if $\xi'$ is any contact structure on $M'$ with dividing set $\g'$
on $\partial M'$, then
\[
\Phi_{\xi}(\EH(M',\g',\xi')) = \EH(M,\g,\xi' \cup \xi) \otimes (x \otimes \dots \otimes x),
\]
where $x$ is the contact class of the standard contact structure on $S^1 \times S^2$.
\end{thm}

Honda, Kazez, and Mati\'c call $\Phi_{\xi}$ the ``gluing map.'' Its construction is rather involved, we only describe it
briefly and highlight its properties that we need later.

Suppose for now that $m = 0$. First, we choose an arbitrary balanced diagram $(-\S'_0,\bolda'_0,\boldb'_0)$ for $(-M',-\g')$. Let $T = \partial M'$, and denote by $\z$ the $I$-invariant contact structure on $T \times I$ such that for every $t \in I$ the surface $T \times \{t\}$ is convex with
dividing set $\g' \times \{t\}$. Then we construct two special partial open book decompositions for $(T \times I,\z)$, and denote by $(-\S_{\z}^1,\bolda_{\z}^1,\boldb_{\z}^1)$ and $(-\S_{\z}^2,\bolda_{\z}^2,\boldb_{\z}^2)$ the corresponding balanced diagrams.
We set $\S' = \S'_0 \cup \S_{\z}^1 \cup \S_{\z}^2$, and extend $\bolda_0' \cup \bolda_{\z}^1 \cup \bolda_{\z}^2$ and
$\boldb_0' \cup \boldb_{\z}^1 \cup \boldb_{\z}^2$ to full sets of attaching circles $\bolda'$ and $\boldb'$, respectively.
We end up with a special kind of balanced diagram $(-\S',\bolda',\boldb')$ for $(-M',-\g')$, that Honda, Kazez, and Mati\'c~\cite{TQFT} call \emph{contact compatible with $\z$ near $\partial M'$.}

In the second step, using the contact structure $\xi$, one extends the balanced diagram $(-\S',\bolda',\boldb')$ to a balanced diagram $(-\S,\bolda,\boldb)$ for $(-M,-\g)$ such that $\S' \subset \S$, while $\bolda = \bolda' \cup \bolda''$ and $\boldb = \boldb' \cup \boldb''$. The curves $\bolda''$ and $\boldb''$ come from a partial open book decomposition, hence there is a distinguished
intersection point
\[
\x'' = (x_1'',\dots,x_m'') \in \T_{\a''} \cap \T_{\b''},
\]
with the properties described above Remark~\ref{rem:EH}. More precisely, if $D$ is the closure of a component of $-\S \setminus (\bolda \cup \boldb)$ such that $x_i'' \in \partial D$ and
$\partial D \cap \a_i$ points out of $x_i''$, then $D \cap \partial \S \neq \emptyset$.

We define the map
\[
\phi_{\xi} \colon CF(-\S',\bolda',\boldb') \to CF(-\S,\bolda,\boldb)
\]
on each generator $\y \in \T_{\a'} \cap \T_{\b'}$ of $CF(-\S',\bolda',\boldb')$ by the formula
$\phi_{\xi}(\y) = (\y,\x'')$, and extend it to $CF(-\S',\bolda',\boldb')$ linearly. This is a chain map, since every domain emanating from $(\y,\x'')$ has zero multiplicities around $\x''$, hence must consist of a domain in $(\S',\bolda',\boldb')$ emanating from $\y$, plus the trivial domain from $\x''$ to itself. The induced map on the homology is
\[
\Phi_{\xi} \colon \SFH(-M',-\g') \to \SFH(-M,-\g).
\]

We point out that the definition of $\Phi_{\xi}$ also works when $(M',\g') = \emptyset$, so $(M,\g,\xi)$ is a contact manifold. Then $\SFH(M',\g') = \Z_2$, and the map $\Phi_{\xi}$ is given by
\[
\Phi_{\xi}(1) = \EH(M,\g,\xi).
\]

The \emph{naturality} of the map $\Phi_{\xi}$ means the following. Suppose that~$\H'_1$ and~$\H'_2$
are balanced diagrams for~$(-M',-\g')$ that are contact compatible with~$\z$
near~$\partial M'$, and let~$\H_1$ and~$\H_2$ be their extensions to~$(-M,-\g)$, respectively. Then the following diagram is
commutative:
\[\begin{CD}
\SFH(\H'_1) @>\Phi_{\xi}^1>> \SFH(\H_1)\\
@VVF_{\H'_1,\H'_2}V   @VVF_{\H_1,\H_2}V\\
\SFH(\H'_2) @>\Phi_{\xi}^2>> \SFH(\H_2),
\end{CD}\]
where~$F_{\H'_1,\H'_2}$ and~$F_{\H_1,\H_2}$ are canonical isomorphisms induced by equivalences between the sutured diagrams, cf.~Section~\ref{sec:naturality}.
Furthermore, if $(M',\g')$ is a sutured submanifold of $(M,\g)$ with contact structure $\xi$ on $Z = M \setminus \Int(M')$
and dividing set~$\g$ on~$\partial M$ and~$\g'$ on~$\partial M'$,
and $(\ol{M'},\ol{\g'})$ is a sutured submanifold of $(\ol{M},\ol{\g})$ with contact structure~$\ol{\xi}$ on $\ol{Z} = \ol{M} \setminus \Int(\ol{M'})$
with dividing set~$\ol{\g}$ on~$\partial \ol{M}$ and~$\ol{\g'}$ on~$\partial \ol{M'}$, then an orientation preserving diffeomorphism $d \colon M \to \ol{M}$ such that $d(Z) = \ol{Z}$, $d(\g') = \ol{\g'}$, $d(\g) = \ol{\g}$, and $d_*(\xi) = \ol{\xi}$ gives rise to a commutative diagram
\[\begin{CD}
\SFH(-M',-\g') @>\Phi_{\xi}>> \SFH(-M,-\g)\\
@VV(d|_{M'})_*V   @VV d_*V\\
\SFH(-\ol{M'},-\ol{\g'}) @>\Phi_{\ol{\xi}}>> \SFH(-\ol{M},-\ol{\g}).
\end{CD}\]

For the case $m>0$, see the discussion on page 7 of \cite{TQFT}. We now list some basic properties of this gluing map. The first is~\cite[Theorem 6.1]{TQFT}.

\begin{thm} \label{thm:id}
Let $(M,\g)$ be a balanced sutured manifold, and let $\xi$ be an $I$-invariant contact structure on $\partial M \times I$ with dividing set $\g \times \{t\}$ on $\partial M \times \{t\}$. Then the gluing map
\[
\Phi_{\xi} \colon \SFH(-M,-\g) \to \SFH(-M,-\g)
\]
obtained by
attaching $(\partial M \times I,\xi)$ onto $(M,\g)$ along $\partial M \times \{0\}$ is the identity map.
\end{thm}

The following statement is~\cite[Proposition~6.2]{TQFT}.

\begin{prop} \label{prop:comp}
Consider the inclusions $(M_0,\g_0) \subset (M_1,\g_1) \subset (M_2,\g_2)$ of sutured submanifolds. For $i \in \{0,1\}$, let $\xi_i$ be a contact structure on $M_{i+1} \setminus \Int(M_i)$ that has convex boundary and dividing set $\g_j$ on $\partial M_j$ for $j \in \{i,i+1\}$. Then
\[
\Phi_{\xi_1} \circ \Phi_{\xi_0}  = \Phi_{\xi_0 \cup \xi_1} \colon \SFH(-M_0,-\g_0) \to \SFH(-M_2,-\g_2).
\]
\end{prop}

\begin{defn}
Let $(M',\g')$ be a sutured submanifold of $(M,\g)$, and let $\xi$ be a contact structure on $M \setminus \Int(M')$ with convex boundary, and dividing set $\g$ on $\partial M$ and $\g'$ on $\partial M'$. As in the proof of Lemma~\ref{lem:contadm}, choose a Riemannian metric
such that~$\xi^\perp$ is admissible.
If $\s' \in \spinc(M',\g')$ and $v'$ is a vector field representing $\s'$, then the homology class of the vector field $v = v' \cup \xi^{\perp}$ is obviously independent of the choice of $v'$.
This defines a map
\[
f_{\xi} \colon \spinc(M',\g') \to \spinc(M,\g),
\]
by setting
$f_{\xi}(\s')$ to be the homology class of $v$.
\end{defn}

Given $\s_1$, $\s_2 \in \spinc(M,\g)$, their difference $\s_1 -\s_2$ is an element of $H^2(M,\partial M)$, which we identify
with $H_1(M)$ using Poincar\'e duality.

\begin{lem} \label{lem:affine}
If $\s_1'$, $\s_2' \in \spinc(M',\g')$, then
\[
f_{\xi}(\s_1') - f_{\xi}(\s_2') = e_*(\s_1'-\s_2'),
\]
where $e_* \colon H_1(M') \to H_1(M)$ is the map induced by the embedding $e \colon M' \hookrightarrow M$.
\end{lem}

\begin{proof}
Let $v_i'$ be a vector field representing $\s_i'$ for $i \in \{1,2\}$. After fixing a trivialization of $TM$, we can view $v_1 = v_1' \cup \xi^{\perp}$ and
 $v_2 = v_2' \cup \xi^{\perp}$ as maps from $M$ to $S^2$. If $p$ is a common regular value of $v_1$ and $v_2$, then
\[
f_{\xi}(\s_1') - f_{\xi}(\s_2') = [v_1^{-1}(p) - v_2^{-1}(p)] \in H_1(M').
\]
But $v_1$ and $v_2$ agree outside $M'$; moreover,
\[
v_1^{-1}(p) - v_2^{-1}(p) = (v_1')^{-1}(p) - (v_2')^{-1}(p).
\]
The right-hand side, thought of as a 1-cycle in $M$, represents $e_*(\s_1' -\s_2')$, which proves the lemma.
\end{proof}

\begin{prop} \label{prop:spincgluing}
Let $(M',\g')$ be a sutured submanifold of $(M,\g)$, and let $\xi$ be a contact structure on $M \setminus \Int(M')$ with convex boundary, and dividing set $\g$ on $\partial M$ and $\g'$ on $\partial M'$. Pick a $\spinc$ structure $\s' \in \spinc(M',\g')$, and choose an element $x' \in \SFH(-M',-\g',\s')$. Then
\[
\Phi_{\xi}(x') \in \SFH(-M,-\g,f_{\xi}(\s')).
\]
\end{prop}

\begin{proof}
Fix an arbitrary contact structure $\xi'$ on $(M',\g')$ such that $\partial M'$ is a convex surface with dividing set $\g'$.
Choose a partial open book decomposition defining $(M',\g',\xi')$, and let $(-\S',\bolda',\boldb')$ be the corresponding balanced diagram of $(-M',-\g')$. Extend $(-\S',\bolda',\boldb')$ to a diagram $(-\S,\bolda,\boldb)$ defining $(-M,-\g)$, as in the definition of the map $\Phi_{\xi}$. More precisely, we have $\bolda = \bolda' \cup \bolda''$ and $\boldb = \boldb' \cup \boldb''$, and there is a distinguished intersection point $\x'' \in \T_{\a''} \cap \T_{\b''}$ such that
$\phi_{\xi}(\y) = (\y,\x'')$ for every $\y \in \T_{\a'} \cap \T_{\b'}$. Let $\y_0 \in \T_{\a'} \cap \T_{\b'}$ be the
distinguished intersection point representing the contact class $\EH(M',\g',\xi')$.

Again, let $e \colon M' \hookrightarrow M$ be the embedding. Then it follows from Lemma~\ref{lem:affine} and the above description of $\phi_{\xi}$ that for any $\y \in \T_{\a'} \cap \T_{\b'}$ we have
\[
\begin{split}
f_{\xi}(\s(\y)) - f_{\xi}(\s(\y_0)) &=  e_*(\s(\y) - \s(\y_0)) = \\
\s(\y,\x'') - \s(\y_0,\x'') &= \s(\phi_{\xi}(\y)) - \s(\phi_{\xi}(\y_0)) \in H_1(M).
\end{split}
\]
Note that $\phi_{\xi}(\y_0) = (\y_0,\x'')$ is the distinguished intersection point representing $\EH(M,\g,\xi)$, see Theorem~\ref{thm:glue}. So the proof of Proposition~\ref{prop:EHspinc} implies that $\s(\y_0) = \s_{\xi'}$ and $\s(\phi_{\xi}(\y_0)) = \s_{\xi \cup \xi'}$. Of course, $f_{\xi}(\s_{\xi'}) = \s_{\xi' \cup \xi}$. We conclude that
\[
f_{\xi}(\s(\y)) = \s(\phi_{\xi}(\y))
\]
for every $\y \in \T_{\a'} \cap \T_{\b'}$. Since $\T_{\a'} \cap \T_{\b'}$ freely generates $CF(-\S,\bolda',\boldb')$, this implies the claim of the proposition.
\end{proof}

\begin{rem}
It is not true in general that for every $x' \in \SFH(-M',-\g')$ there exists a contact structure $\xi'$ on $(M',\g')$ such that $x' = \EH(M',\g',\xi')$.
\end{rem}

\begin{cor} \label{cor:gluingsum}
For any $\s' \in \spinc(M',\g')$, let
\[
\Phi_{\xi,\s'} = \Phi_{\xi}|_{\SFH(-M',-\g',\s')}.
\] Then
\[
\Phi_{\xi,\s'} \colon \SFH(-M',-\g',\s') \to \SFH(-M,-\g,f_{\xi}(\s')).
\]
\end{cor}

\begin{rem}
In the remark on page 11 of~\cite{TQFT}, a construction is outlined for a map from $\SFH(-M',-\g')$ to $\SFH(-M,-\g)$
that depends on a $\spinc$ structure $\s' \in \spinc(M',\g')$. Specifically, one would use a contact structure $\xi'$
on $(M',\g')$ that represents $\s'$ to construct the balanced diagram for $(-M',-\g')$ via a partial open book decomposition.
This would replace the diagrams for $(-M',-\g')$ that are contact compatible near the boundary.
It is natural to conjecture that the restriction of this map to $\SFH(-M',-\g',\s')$ is precisely $\Phi_{\xi,\s'}$, which might
simplify concrete computations of $\Phi_{\xi,\s'}$.
\end{rem}

\section{Construction of the cobordism map $F_{\W}$}

We can now describe the construction of the map $F_{\W}$ for an arbitrary balanced cobordism $\W = (W,Z,[\xi])$. It is a composition of the gluing map $\Phi_{-\xi}$ induced by the contact structure $-\xi$ and the cobordism map induced by a special cobordism.

\begin{figure}[tb]
\includegraphics{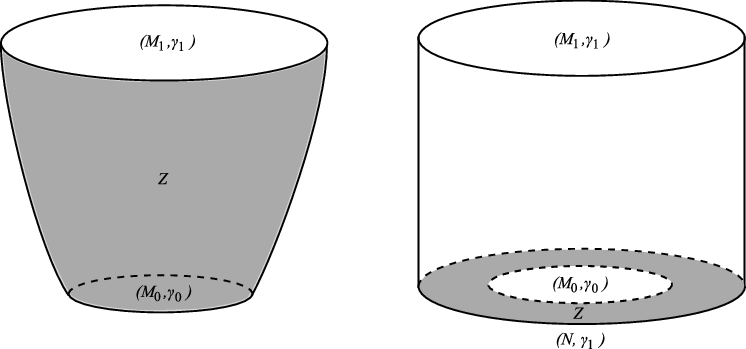}
\caption{A balanced cobordism $\W$ having no isolated components on the left, and the corresponding special cobordism $\W_1$ on
the right.} \label{fig:7}
\end{figure}

\begin{defn} \label{defn:phi}
Let $\W = (W,Z,[\xi])$ be a balanced cobordism from~$(M_0,\g_0)$ to $(M_1,\g_1)$. Following~\cite{TQFT}, we say that a component $Z_0$ of $Z$ is \emph{isolated} if $Z_0 \cap M_1 = \emptyset$. First, we define $F_{\W}$ if $Z$ has no isolated components, see Figure~\ref{fig:7}.

Observing the orientation conventions in Remark~\ref{rem:cob}, note that $(-M_0,-\g_0)$ is a sutured submanifold of $(-N,-\g_1) = (-M_0 \cup Z,-\g_1)$.
Furthermore, $-\xi$ is a positive contact structure on~$Z$ with dividing set~$-\g_0$ on~$\partial M_0$ and~$-\g_1$ on~$\partial M_1$.
Thus Theorem~\ref{thm:glue} provides us with a map
\[
\Phi_{-\xi} \colon \SFH(M_0,\g_0) \to \SFH(N,\g_1).
\]

The sutured manifolds $(N, \g_1)$ and $(M_1,\g_1)$ have the same boundary and same sutures, and $\partial W = -N \cup M_1$. We can think of this as a special cobordism from $(N,\g_1)$ to $(M_1,\g_1)$,  call it $\W_1$. For each such cobordism we constructed a map
\[
F_{\W_1} \colon \SFH(N,\g_1) \to \SFH(M_1,\g_1)
\]
in Section~\ref{sec:special}. Finally, we set
\[
F_{\W} = F_{\W_1} \circ \Phi_{-\xi}.
\]

In the general case, for each isolated component $Z_0$ of $Z$, choose a small standard contact ball
$B_0 \subset \Int(Z_0)$ with convex boundary and connected dividing set $\delta_0$. Let~$B$ be the union of all such balls, and $\delta$ the dividing set on $\partial B$. Then we replace the cobordism $\W$ with $\W' = (W,Z',[\xi'])$, where $Z' = Z \setminus \Int(B)$ and $\xi' = \xi|_{Z'}$. The morphism $\W'$ has source $(M_0,\g_0)$ and range $(M_1',\g_1') = (M_1,\g_1) \sqcup (B,\delta)$. Since $\SFH(M_1',\g_1') \cong \SFH(M_1,\g_1)$, and now $Z'$ has no isolated components, we can define $F_{\W} = F_{\W'}$.
It follows from Proposition~\ref{prop:B} that this is independent of the choice of~$B$.
\end{defn}

The map $F_{\W}$ has a refinement along relative $\spinc$ structure on $\W$. We describe this next.

\begin{defn}
Let $\W = (W,Z,[\xi])$ be a balanced cobordism from $(M_0,\g_0)$ to~$(M_1,\g_1)$. Choose a relative $\spinc$ structure
$\s \in \spinc(\W)$, and for $i \in \{0,1\}$ set
\[
\ft_i = \s|_{M_i} \in \spinc(M_i,\g_i).
\]

First,  suppose that $Z$ has no isolated components. Then there is a natural restriction map from $\spinc(\W)$ to $\spinc(\W_1)$, we denote the image of $\s$ by $\s_1 = \s|_{\W_1}$. If we view~$\ft_0$ as an element of~$\spinc(-M_0,-\g_0)$, then
\[
\s_1|_{(-N,-\g_1)} = f_{-\xi}(\ft_0) \in \spinc(-N,-\g_1).
\]
Hence the special cobordism $\W_1$, endowed with the $\spinc$ structure $\s_1$, induces a map
\[
F_{\W_1,\s_1} \colon \SFH(N,\g_1,f_{-\xi}(\ft_0)) \to \SFH(M_1,\g_1,\ft_1).
\]
By Corollary~\ref{cor:gluingsum}, the sutured submanifold
$(-M_0,-\g_0)$ of $(-N,-\g_1)$, together with the $\spinc$ structure $\ft_0$, give rise to a map
\[
\Phi_{\xi,\ft_0} \colon \SFH(M_0,\g_0,\ft_0) \to \SFH(N,\g_1,f_{-\xi}(\ft_0)).
\]
So we can define
\[
F_{\W,\s} \colon \SFH(M_0,\g_0,\ft_0) \to \SFH(M_1,\g_1,\ft_1)
\]
by the formula $F_{\W,\s} = F_{\W_1,\s_1} \circ \Phi_{-\xi,\ft_0}$.

When $Z$ does have isolated components, then as before, we take the cobordism~$\W'$ from $(M_0,\g_0)$ to $(M_1',\g_1') = (M_1,\g_1) \sqcup (B,\d)$. Then set $\s' = \s|_{\W'}$, and notice that $\ft_1' = \s'|_{(M_1',\g_1')}$ agrees with $\ft_1$ on $(M_1,\g_1)$, and is the vertical
$\spinc$ structure $\ft_B$ on the product $(B,\d)$. Hence
\[
\SFH(M_1',\g_1',\ft_1') = \SFH(M_1,\g_1,\ft_1) \otimes \SFH(B,\d,\ft_B) \cong \SFH(M_1,\g_1,\ft_1) \otimes \Z_2,
\]
and we can set $F_{\W,\s} = F_{\W',\s'}$.
\end{defn}

\begin{prop} \label{prop:B}
Let~$\W = (W,Z,[\xi])$ be a balanced cobordism and~$\s \in \spinc(\W)$.
Then the map~$F_{\W,\s}$ defined above is independent of the choice of balls in the isolated components of~$Z$.
\end{prop}

\begin{proof}
Let~$\W'$ and~$\W''$ be cobordisms obtained from~$\W$ by removing from~$Z$ the standard contact balls~$B'$ and~$B''$
with dividing sets~$\d'$ and~$\d''$, respectively, and adding them to~$(M_1,\g_1)$.
In particular, $\W'$ is a cobordism from~$(M_0,\g_0)$ to
\[
(M_1',\g_1') = (M_1,\g_1) \sqcup (B',\d'),
\]
and $\W''$ is a cobordism from~$(M_0,\g_0)$ to
\[
(M_1'',\g_1'') = (M_1,\g_1) \sqcup (B'',\d'').
\]

By the homogeneity of~$Z$, there is a diffeomorphism $d \colon \W' \to \W''$
that is the identity on~$M_0 \cup M_1 \subset \partial W$ and is isotopic to~$\Id_W$.
Furthermore, we can choose~$d$ such that~$d(\xi') = d(\xi|_{Z'})$
and~$\xi'' = \xi|_{Z''}$ are equivalent. Indeed, choose a tight contact ball~$B \subset Z$
containing both~$B'$ and~$B''$ with convex boundary and connected dividing set~$\d$.
Let $C' = B \setminus \Int(B')$ and $C'' = B \setminus \Int(B'')$.
Choose properly embedded arcs~$a' \subset C'$
and~$a'' \subset C''$ connecting~$\d$ with~$\d'$ and~$\d$ with~$\d''$,
respectively. Then let~$A' \subset C'$
and~$A'' \subset C''$  be properly embedded annuli
parallel to~$a'$ and~$a''$, respectively, in convex position.
The contact structures~$\xi'|_{C'}$ and~$\xi''|_{C''}$ are both tight
as they are restrictions of the tight contact structure~$\xi|_B$,
and they are determined up to isotopy by their dividing sets on~$A'$ and~~$\partial C'$,
and on~$A''$ and~$\partial C''$, respectively.
By convex surface theory, the dividing sets~$\nu'$ and~$\nu''$ on~$A'$ and~$A''$ consist of two parallel
arcs connecting the two boundary components. Hence, there is a diffeomorphism~$d \colon Z \to Z$
that is the identity outside~$B$, maps~$(B',\d')$ to~$(B'',\d'')$, and~$(A',\nu')$
to~$(A'',\nu'')$. The restriction of~$d$ to~$Z'$ gives the desired map.
Indeed, decomposing~$C'$ along~$A'$ gives two standard contact balls, and each has a
unique contact structure up to isotopy fixing the dividing sets on their boundaries.
These are mapped by~$d$ to the contact balls obtained by decomposing~$C''$ along~$A''$.

Then we can apply Theorem~\ref{thm:finalnat}
for cobordisms without isolated components to obtain the commutative diagram
\[\begin{CD}
\SFH(M_0,\g_0,\ft_0) @>F_{\W',\s'}>> \SFH(M_1',\g_1',\ft_1') \\
@VV(d|_{M_0})_*V   @VV(d|_{M_1'})_*V\\
\SFH(M_0,\g_0,\ft_0) @>F_{\W'',\s''}>> \SFH(M_1'',\g_1'',\ft_1''),
\end{CD}\]
where $\s' = \s|_{\W'}$ and $\s'' = \s|_{\W''}$.
By construction, $(d|_{M_0})_*$ is the identity of the group $\SFH(M_0,\g_0,\ft_0)$.
Furthermore, and~$d|_{M_1'}$ is the identity on~$M_1$ and sends~$B'$ to~$B''$.
Hence, after canonically identifying $\SFH(B',\d',\ft_{B'})$ and $\SFH(B'',\d'',\ft_{B''})$ with~$\Z_2$,
\[
(d|_{M_1'})_* \colon \SFH(M_1,\g_1,\ft_1) \otimes \Z_2 \to  \SFH(M_1,\g_1,\ft_1) \otimes \Z_2
\]
is the identity as well. It follows that $F_{\W',\s'} = F_{\W'',\s''}.$
\end{proof}

The above construction motivates the following definition.

\begin{defn}
A cobordism $\W = (W,Z,[\xi])$ from $(M_0,\g_0)$ to $(N,\g_1)$ is called a \emph{boundary cobordism} if $\W$ is balanced, $N$ is parallel to $M_0 \cup (-Z)$, and we are also given a retraction $r \colon W \to M_0 \cup (-Z)$ such that $r|_N$ is an orientation preserving diffeomorphism from $N$ to $M_0 \cup (-Z)$.
\end{defn}

Given a boundary cobordism, we can view $(-M_0,-\g_0)$ as a sutured submanifold of $(-N,-\g_1)$, and $-\xi$ is a contact structure
such that $\partial M_0 \cup \partial N$ is  a convex surface with dividing set $-\g_0 \cup -\g_1$. Hence a boundary cobordism
gives rise to a map
\[
\Phi_{-\xi} \colon \SFH(M_0,\g_0) \to \SFH(N,\g_1).
\]

\begin{defn}
Let $\W = (W,Z,[\xi])$ and $\W' = (W',Z',[\xi'])$ be boundary cobordisms from $(M_0,\g_0)$ to $(N,\g_1)$, together with retractions $r$ and $r'$, respectively. Then we say that $\W$ and $\W'$ are \emph{equivalent} if there is an equivalence $d \colon W \to W'$ in the sense of Definition~\ref{defn:cobeq} that also respects the retractions; i.e., $d \circ r = r' \circ d$. Such a $d$ is called an \emph{equivalence}.

If $\W$ is a boundary cobordism from $(M_0,\g_0)$ to $(N,\g_1)$, and $\W'$ is a boundary cobordism from $(M_0',\g_0')$ to $(N',\g_1')$, then $\W$ and $\W'$ are said to be \emph{diffeomorphic} if there is a diffeomorphism $d \colon W \to W'$ in the sense
of Definition~\ref{defn:cobeq} that also satisfies $d \circ r = r' \circ d$. We call such a $d$ a \emph{diffeomorphism}.
\end{defn}

The naturality of $\Phi_{-\xi}$ implies the following statement.

\begin{prop} \label{prop:natbound}
Suppose that $d$ is a diffeomorphism from the boundary cobordism $\W = (W,Z,[\xi])$ from $(M_0,\g_0)$
to $(N,\g_1)$ to $\W' = (W',Z',[\xi'])$ from $(M_0',\g_0')$ to~$(N',\g_1')$.
Furthermore, let $\ft \in \spinc(M_0,\g_0)$ and $\ft' = d_*(\ft)$. Then there is a commutative diagram
\[\begin{CD}
\SFH(M_0,\g_0,\ft) @>\Phi_{-\xi,\ft}>> \SFH(N,\g_1,f_{-\xi}(\ft))\\
@VV(d|_{M_0})_*V   @VV(d|_N)_*V\\
\SFH(M_0',\g_0',\ft') @>\Phi_{-\xi',\ft'}>> \SFH(N',\g_1',f_{-\xi'}(\ft')),
\end{CD}\]
and there is an analogous diagram for $\Phi_{\xi}$ and $\Phi_{\xi'}$ that does not refer to $\spinc$ structures.
\end{prop}

\begin{lem} \label{lem:decomp}
Suppose that $\W = (W,Z,[\xi])$ is a balanced cobordism from $(M_0,\g_0)$ to $(M_1,\g_1)$ such that $Z$ has no isolated components.
Then $\W$ can be written as a composition $\W_1 \circ \W_0$, where $\W_0$ is a boundary cobordism from $(M_0,\g_0)$ to some
sutured manifold $(N,\g_1)$, and $\W_1$ is a special cobordism from $(N,\g_1)$ to $(M_1,\g_1)$.

This decomposition is unique in the following sense. Suppose that $d$ is a diffeomorphism from~$\W$ to another cobordism~$\W'$
from $(M_0',\g_0')$ to $(M_1',\g_1')$, and we have a splitting $\W' = \W_1' \circ \W_0'$ along $(N',\g_1')$.
Then there is a diffeomorphism $d'$ such that $d'(N,\g_1) = (N',\g_1')$ and $d(x) = d'(x)$ for every $x \in M_0 \cup M_1$. Furthermore,
the map~$d'|_{W_0}$ is a diffeomorphism from the boundary cobordism $\W_0$ to $\W_0'$, and $d'|_{W_1}$ is a diffeomorphism from the
special cobordism~$\W_1$ to~$\W_1'$. Finally, $d_*(\s) = d'_*(\s)$ for every $\s \in \spinc(\W)$.
\end{lem}

\begin{proof}
First, choose a collar neighborhood $Z_1 = \partial M_1 \times I$ of $\partial M_1$ in $Z$ such that $\partial M_1 = \partial M_1 \times \{0\}$, and for every $t \in I$ the surface $\partial M_1 \times \{t\}$ is convex with dividing set $\g_1 \times \{t\}$ in $\xi_1 = \xi|_{Z_1}$. Set $Z_0$ to be the closure of $Z \setminus Z_1$. Then choose a properly embedded 3-manifold $N \subset W$ such that $\partial N = \partial M_0 \times \{1\}$, and $N$ is parallel
to~$M_0 \cup (-Z_0)$. Cutting $\W$ along $N$ gives the required decomposition.

The uniqueness follows from the uniqueness of the ``product'' collar neighborhood $(Z_1,\xi_1)$, see \cite[Theorem~2.5.23]{contact}. So one can isotope $d$ relative to $M_0 \cup M_1$ thorough diffeomorphisms of cobordisms until we obtain a diffeomorphism~$d'$ such that~$d'(N,\g_1) = (N',\g_1')$.
Then a further isotopy in $\Int(W')$ ensures that
$d' \circ r = r' \circ d'$, where $r$ and $r'$ are the retractions for the boundary cobordisms $\W_0$ and $\W_0'$, respectively.
Hence $d'|_{W_0}$ is an equivalence of boundary cobordisms, and~$d'|_{W_1}$ is an equivalence of special cobordisms.
The last statement follows from the fact that we got~$d'$ from~$d$ by isotoping it through diffeomorphisms of cobordisms.
\end{proof}

Let $\W = (W,Z,[\xi])$ be a balanced cobordism such that $Z$ has no isolated components. An alternative way of thinking about the map $F_{\W}$ is
to write $\W = \W_1 \circ \W_0$ as in Lemma~\ref{lem:decomp}; i.e., $\W_0 = (W_0,Z_0,[\xi_0])$ is a boundary cobordism from $(M_0,\g_0)$ to $(N,\g_1)$,
and~$\W_1$ is a special cobordism from $(N,\g_1)$ to $(M_1,\g_1)$. Then set
\[
F_{\W} = F_{\W_1} \circ F_{\W_0},
\]
where $F_{\W_1}$ is the map defined in Section~\ref{sec:special}, and $F_{\W_0} = \Phi_{-\xi_0}$.

\begin{prop} \label{prop:extends}
Let $\W$ be a cobordism from $(M_0,\g_0)$ to $(M_1,\g_1)$. If $\W$ is a special cobordism, then the
map~$F_{\W}$ defined here agrees with the special cobordism map~$F_{\W}$ defined in
Section~\ref{sec:special}. If $\W = (W,Z,[\xi])$ is a boundary cobordism,
then we have $F_{\W} = \Phi_{-\xi}$.
\end{prop}

\begin{proof}
If $\W$ is a special cobordism, then~$\W_0 = (W_0,Z_0,[\xi_0])$ is the trivial cobordism from $(M_0,\g_0)$ to itself,
so by Theorem~\ref{thm:id}, the map $\Phi_{-\xi_0}$ is the identity of~$\SFH(M_0,\g_0)$, and $F_{\W} = F_{\W_1}$. Furthermore, $\W = \W_1$.

On the other hand, if $\W$ is a boundary cobordism, then $\W_1$ is the trivial cobordism from $(M_1,\g_1)$ to itself, so
$\W = \W_0$. By Proposition~\ref{prop:specialid}, the map $F_{\W_1}$ is the identity, hence
$F_{\W} = \Phi_{-\xi_0}$.
\end{proof}

Consequently, the maps $F_{\W}$ generalize both special cobordism maps and gluing maps.

\begin{defn}
Suppose that $\W = (W,Z,[\xi])$ is a balanced cobordism such that~$Z$
has no isolated components. Write $\W$ as $\W_1 \circ \W_0$, where $\W_0$ is a boundary cobordism and
$\W_1$ is a special cobordism. We say that the $\spinc$ structures
$\s$, $\s' \in \spinc(\W)$ are \emph{equivalent}, in short $\s \sim \s'$, if $\s|_{\W_i} = \s'|_{\W_i}$ for $i \in \{0,1\}$.
We denote the set of equivalence classes by $\spinc(\W)/_{\sim}$.

If $\W$ does have isolated components, then we let $\spinc(\W)/_{\sim} = \spinc(\W')/_{\sim}$, where
$\W'$ is the cobordism from $(M_0,\g_0)$ to $(M_1,\g_1) \sqcup (B,\d)$ introduced in Definition~\ref{defn:phi}.

By construction, if $\s \sim \s'$, then $F_{\W,\s} = F_{\W,\s'}$. For $s \in \spinc(\W)/_{\sim}$, let
\[
F_{\W,s} = F_{\W,\s},
\]
where $\s$ is an arbitrary representative of $s$.
\end{defn}

\begin{rem}
Suppose that $\W_0$ is a boundary cobordism from $(M_0,\g_0)$ to $(N,\g_1)$, and $\W_1$ is a special cobordism from
$(N,\g_1)$ to $(M_1,\g_1)$. Let $\W = \W_1 \circ \W_0$. Using a relative Mayer-Vietoris sequence,
we see that $\spinc(\W)/_{\sim}$ corresponds to the set of
$\d H^1(N,\partial N)$ orbits in $\spinc(\W)$.
Hence $\spinc(\W) = \spinc(\W)/_{\sim}$ if $H_2(N) = 0$.

Also note that $\spinc(\W_0) \cong \spinc(M_0,\g_0)$.
For $\ft_0 \in \spinc(M_0,\g_0)$, let $\s_0$ be the corresponding element of $\spinc(\W_0)$. Then $\s_0|_N = f_{-\xi}(\ft_0)$,
and in general the map $f_{-\xi}$ is neither injective, nor surjective, cf.~Lemma~\ref{lem:affine}.
Furthermore, $\s$, $\s' \in \spinc(\W)$ are equivalent if and only if $\s|_{M_0} = \s'|_{M_0}$ and $\s|_{\W_1} = \s'|_{\W_1}$.
\end{rem}

\begin{prop} \label{prop:spinc}
Given a balanced cobordism $\W = (W,Z,[\xi])$, we have
\[
F_{\W} = \bigoplus_{s \in \spinc(\W)/_{\sim}} F_{\W,s}.
\]
\end{prop}

\begin{proof}
This follows from Proposition~\ref{prop:specialsum} and Corollary~\ref{cor:gluingsum}.
\end{proof}

\section{Properties of the cobordism map $F_{\W}$}

\subsection{Naturality and functoriality}
We start by proving a naturality result for the cobordism map $F_{\W}$.

\begin{thm} \label{thm:finalnat}
Suppose that $\W = (W,Z,[\xi])$ is a balanced cobordism from the sutured manifold $(M_0,\g_0)$ to $(M_1,\g_1)$, and $\W' = (W',Z',[\xi'])$ is a balanced cobordism from $(M_0',\g_0')$ to $(M_1',\g_1')$. Pick an $\s \in \spinc(\W)$, and let $\ft_i = \s|_{M_i}$ for $i \in \{0,1\}$.
If $d$ is a diffeomorphism from
$\W$ to $\W'$, then we have a commutative diagram
\[\begin{CD}
\SFH(M_0,\g_0,\ft_0) @>F_{\W,\s}>> \SFH(M_1,\g_1,\ft_1)\\
@VV(d|_{M_0})_*V   @VV(d|_{M_1})_*V\\
\SFH(M_0',\g_0',\ft_0') @>F_{\W',\s'}>> \SFH(M_1',\g_1',\ft_1'),
\end{CD}\]
where $\s' = d_*(\s)$ and $\ft_i' = \s'|_{M_i'}$ for $i \in \{0,1\}$.
An analogous statement holds for~$F_{\W}$.
\end{thm}

\begin{proof}
This follows from the corresponding naturality result for special cobordisms, Theorem~\ref{thm:welldef}, and the naturality
of the gluing map $\Phi_{-\xi}$. More precisely, first assume that $Z$ and $Z'$ have no isolated components.
Set $N = M_0 \cup (-Z)$ and $N' = M_0' \cup (-Z')$. Then $(M_0,\g_0)$ is a sutured submanifold of $(N,\g_1)$, and
$(M_0',\g_0')$ is a sutured submanifold of $(N',\g_1')$. Furthermore, $d|_N$ is an orientation preserving diffeomorphism
from $N$ to $N'$ such that $d(Z) = Z'$, $d(\g_0) = \g_0'$, $d(\g_1) = \g_1'$, and~$d_*(\xi) = \xi'$. Hence the naturality of
the gluing map gives the commutative diagram
\[\begin{CD}
\SFH(M_0,\g_0,\ft_0) @>\Phi_{-\xi,\ft_0}>> \SFH(N,\g_1,f_{-\xi}(\ft_0))\\
@VV(d|_{M_0})_*V   @VV(d|_N)_*V\\
\SFH(M_0',\g_0',\ft_0') @>\Phi_{-\xi',\ft_0'}>> \SFH(N',\g_1',f_{-\xi'}(\ft_0')).
\end{CD}\]
On the other hand, $d$ gives rise to a diffeomorphism between the special cobordism~$\W_1$ from $(N,\g_1)$ to~$(M_1,\g_1)$
and the special cobordism~$\W_1'$ from $(N',\g_1')$ to~$(M_1',\g_1')$. So Theorem~\ref{thm:welldef} gives the commutative diagram
\[\begin{CD}
\SFH(N,\g_1,f_{-\xi}(\ft_0)) @>F_{\W_1,\s_1}>> \SFH(M_1,\g_1,\ft_1)\\
@VV(d|_N)_*V   @VV(d|_{M_1})_*V\\
\SFH(N',\g_1',f_{-\xi'}(\ft_0')) @>F_{\W_1',\s_1'}>> \SFH(M_1',\g_1',\ft_1').
\end{CD}\]
Putting the above two commutative diagrams together gives the required diagram.
The case when $Z$ and $Z'$ have isolated components follows from this by deleting standard contact balls $(B,\d)$ from $Z$,
and then deleting the corresponding balls $B' = d(B)$ from $Z'$.
\end{proof}

\begin{cor}
If the balanced cobordisms $\W$ and $\W'$ from $(M_0,\g_0)$ to $(M_1,\g_1)$ are equivalent, then $F_{\W} = F_{\W'}$.
Furthermore, if $d$ is an equivalence and $\s \in \spinc(\W)$, then $F_{\W,\s} = F_{\W', d_*(\s)}$.
\end{cor}

\begin{proof}
This follows from Theorem~\ref{thm:finalnat} by observing that for a strong equivalence $d$ and for $i \in \{0,1\}$
the map $d|_{M_i}$ is the identity of $M_i$,
hence $(d|_{M_i})_*$ is the identity of $\SFH(M_i,\g_i)$.
\end{proof}

So $F_{\W}$ can be defined on strong equivalence classes of cobordisms; i.e., on morphisms of the category of $\BSut$. We now show that $F_{\W}$ is in fact functorial.

\begin{thm} \label{thm:funct}
Let $\W_1$ be a balanced cobordism from $(M_0,\g_0)$ to $(M_1,\g_1)$, and let $\W_2$ be a balanced cobordism from $(M_1,\g_1)$ to $(M_2,\g_2)$.
We write $\W$ for the composition $\W_2 \circ \W_1$. Then
\[
F_{\W_2} \circ F_{\W_1} = F_{\W}.
\]
Furthermore, if $s_i \in \spinc(\W_i)/_{\sim}$ for $i \in \{1,2\}$, then
\[
F_{\W_2,s_2} \circ F_{\W_1,s_1} = \sum_{\{s \in \spinc(\W)/_{\sim}\colon s|_{\W_1} = s_1, \, s|_{\W_2} = s_2\}} F_{\W,s}.
\]
\end{thm}

\begin{proof}

\begin{figure}[tb]
\includegraphics{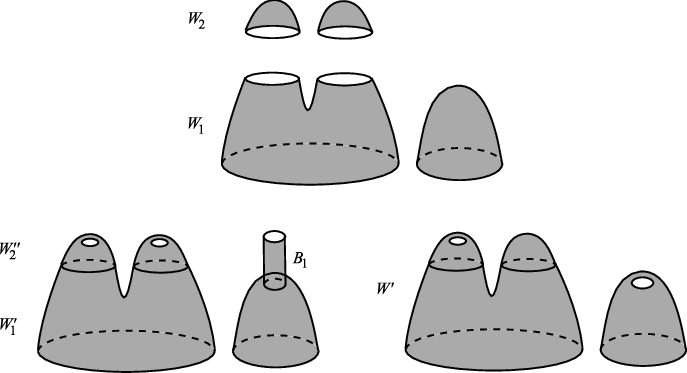}
\caption{The cobordism $\W_1$ has one isolated component, while $\W_2$ has two. In this example, $\W_2'' \circ \W_1'$ and $\W'$ are different.} \label{fig:8}
\end{figure}

We show how to reduce to the case when neither $\W_1$, nor $\W_2$ has isolated components. Recall that in the general case, for
$i \in \{1,2\}$, the map $F_{\W_i}$ was defined to be $F_{\W_i'}$, where $\W_i'$ is a cobordism from $(M_{i-1},\g_{i-1})$ to
\[
(M_i',\g_i') = (M_i,\g_i) \sqcup (B_i,\d_i).
\]
The cobordisms $\W_1'$ and $\W_2'$ do not compose if $B_1 \neq \emptyset$, see Figure \ref{fig:8}. To get around this problem, let $\mathcal{B}_1$ be the trivial cobordism from $(B_1,\d_1)$ to itself, and take $\W_2''$ to be the disjoint union of $\W_2'$ and $\mathcal{B}_1$. Then, assuming the result for no isolated components, we have
\[
F_{\W_2''} \circ F_{\W_1'} = F_{\W_2'' \circ \W_1'}.
\]
Now $F_{\W_2''} = F_{\W_2'}$ by Proposition \ref{prop:specialid},
so the left-hand side is just $F_{\W_2} \circ F_{\W_1}$.

To define the map $F_{\W}$, we use a cobordism $\W'$ with no isolated components. Notice that $\W_2'' \circ \W_1'$ almost agrees with $\W' = (W,Z',\xi')$, except that some standard contact balls might be removed from $(Z',\xi')$ and added to $(M_2',\g_2')$. Such
a situation is depicted on Figure \ref{fig:8}.
However, the following lemma will ensure that
\[
F_{\W_2'' \circ \W_1'} = F_{\W'} = F_{\W}
\]
still holds, and hence we still have
$F_{\W_2} \circ F_{\W_1} = F_{\W}$ in the presence of isolated components.

\begin{lem}
Suppose that $\W = (W,Z,[\xi])$ is a balanced cobordism from $(M_0,\g_0)$ to~$(M_1,\g_1)$. Let $B \subset \Int(Z)$ be a standard contact ball in
$(Z,\xi)$ with convex boundary and connected dividing set $\d$. Furthermore, let $\V = (W,Z_0,\xi_0)$ be the cobordism from
$(M_0,\g_0)$ to $(M_1,\g_1) \sqcup (B,\d)$, where $Z_0 = Z \setminus \Int(B)$ and $\xi_0 = \xi|_{Z_0}$. Then $F_{\W} = F_{\W_0}$,
after we identify $\SFH(M_1,\g_1)$ with $\SFH((M_1,\g_1) \sqcup (B,\d))$.
\end{lem}

\begin{proof}

\begin{figure}[tb]
\includegraphics{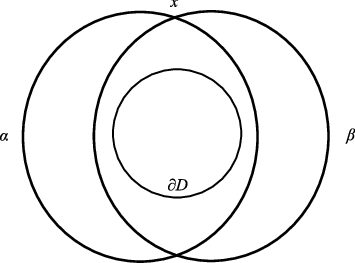}
\caption{The Heegaard diagram after removing a standard contact ball.} \label{fig:9}
\end{figure}

First, assume that $Z$ has no isolated components, then the same holds for~$Z_0$.
By definition, $F_{\W} = F_{\W_1} \circ \Phi_{-\xi}$ and $F_{\V} = F_{\V_1} \circ \Phi_{-\xi_0}$,
where $\W_1$ is a special cobordism from some sutured manifold $(N,\g_1)$ to $(M_1,\g_1)$.
As explained on page~7 of~\cite{TQFT}, the map
\[
\Phi_{-\xi_0} \colon \SFH(M_0,\g_0) \to \SFH(N,\g_1) \otimes V
\]
agrees with $\Phi_{-\xi}$ composed with a map
\[
\SFH(N,\g_1) \to \SFH(N,\g_1) \otimes V,
\]
given by connected summing with the suture manifold $S^3(1)$.
More precisely, take the same balanced diagram that gives $\Phi_{-\xi}$, remove an open ball~$D$
from the Heegaard surface, and add a circle~$\a$ and a circle~$\b$ that are small Hamiltonian
translates of each other, and are parallel to~$\partial D$, see Figure~\ref{fig:9}. We call this diagram~$\H_0$.
If~$x \in \a \cap \b$ is the intersection point with the smaller grading, and if $\phi_{-\xi}(\y) = (\y,\x'')$, then
$\phi_{-\xi'}(\y) = (\y,\x'',x)$. Furthermore, $F_{\V_1}$ agrees with $F_{\W_1}$ composed with the
3-handle map corresponding to the connected sum sphere~$S$,
represented by the periodic domain bounded by~$\a$ and~$\b$, see Definition~\ref{defn:3handle}.
So, on the chain level, the map~$F_{\V_1}$ is given by $f_{\H_0,S}(\y,\x'',x) = (\y,\x'')$.
Hence
\[
F_{\V} \colon \SFH(M_0,\g_0) \to \SFH(M_1,\g_1) \otimes \SFH(B,\d) \cong \SFH(M_1,\g_1)
\]
agrees with $F_{\W}$.

Now consider the general case. If $B$ lies in an isolated component of $Z$, then the claim is obvious.
So suppose that $B$ lies in a non-isolated component of $Z$.
Then one can apply the lemma to $\W'$ and $\W_0'$ to obtain that
\[
F_{\W} = F_{\W'} = F_{\W_0'} = F_{\W_0},
\]
which concludes the proof of the lemma.
\end{proof}

From now on, we can suppose that neither $\W_1$, nor $\W_2$ has isolated components. For $i \in \{1,2\}$, we have a unique decomposition
\[
\W_i = \W_{i1} \circ \W_{i0},
\]
where $\W_{i0}$ is a boundary cobordism and $\W_{i1}$ is a special cobordism. By definition, $F_{\W_i} = F_{\W_{i1}} \circ F_{\W_{i0}}$. So
\[
F_{\W_2} \circ F_{\W_1}  = F_{\W_{21}} \circ F_{\W_{20}} \circ F_{\W_{11}} \circ F_{\W_{10}}.
\]
We can uniquely write the cobordism $\W_{20} \circ \W_{11}$ in the form $\V_1 \circ \V_0$, where $\V_0$ is a boundary cobordism
and $\V_1$ is a special cobordism. Then
\[
\W = \W_{21} \circ \V_1 \circ \V_0 \circ \W_{10},
\]
and $\W_{21} \circ \V_1$ is a special cobordism, whereas $\V_0 \circ \W_{10}$ is a boundary cobordism. So, by the uniqueness part of Lemma \ref{lem:decomp},
\[
F_{\W} = F_{\W_{21} \circ \V_1} \circ F_{\V_0 \circ \W_{10}}.
\]
Using Theorem \ref{thm:specialcomp}, we get
\[
F_{\W_{21} \circ \V_1} = F_{\W_{21}} \circ F_{\V_1},
\]
and Proposition \ref{prop:comp} implies that
\[
F_{\V_0 \circ \W_{10}} = F_{\V_0} \circ F_{\W_{10}}.
\]
So we are done as soon as we show that
\[
F_{\V_1} \circ F_{\V_0} = F_{\W_{20}} \circ F_{\W_{11}}.
\]
This follows from the next proposition.

\begin{prop}
Suppose that $\W = \W_1 \circ \W_0$, where $\W_1$ is a boundary cobordism and $\W_0$ is a special cobordism. Let $\V_1 \circ \V_0$ be the unique decomposition of $\W$ into a boundary cobordism $\V_0$ and a special cobordism $\V_1$. Then
\[
F_{\W_1} \circ F_{\W_0} = F_{\V_1} \circ F_{\V_0}.
\]
\end{prop}

\begin{proof}
Suppose that $(M',\g')$ is a sutured submanifold of $(M,\g)$, and assume $\xi$ is a contact structure on $M \setminus \Int(M')$
that has no isolated components. The result follows once we show that the gluing map $\Phi_{-\xi}$ commutes with 1-, 2-, and 3-handle maps,
corresponding to handle attachments along~$(-M',-\g')$. I.e., we can assume that $\W_0$ and $\V_1$ are both $k$-handle cobordisms for $k \in \{1,2,3\}$.

First, we consider the case $k =2$. Let $\L' \subset (-M',-\g')$ be the framed link along which we attach the 2-handles,
and we write $\L$ when we consider this link in $(-M,-\g)$.
Then construct a triple diagram $(-\S_0',\bolda_0',\boldb_0',\boldd_0')$ subordinate to some bouquet for $\L'$. As explained after
Theorem~\ref{thm:glue}, we can extend $(-\S_0',\bolda_0',\boldb_0')$ to a diagram $(-\S',\bolda',\boldb')$ of $(-M',-\g')$ that is contact compatible near $\partial M'$, which we then further extend to a diagram $(-\S,\bolda,\boldb)$ defining $(-M,-\g)$
in a way compatible with $-\xi$. Next, we extend $\boldd_0'$ to a set of curves $\boldd$ by taking small Hamiltonian translates
of the curves in $\boldb \setminus \boldb_0'$, and denote by $\boldd'$ the translates of the curves in~$\boldb' \setminus \boldb_0'$,
together with $\boldd_0'$. Then the two-handle map $F_{\L'}$ is defined using the triple diagram $\cT' = (-\S',\bolda',\boldb',\boldd')$,
while $F_{\L}$ can be defined using $\cT = (-\S,\bolda,\boldb,\boldd)$.

\begin{figure}[tb]
\includegraphics{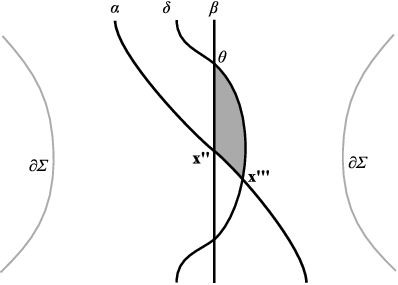}
\caption{An illustration of the proof that the gluing and the 2-handle maps commute.} \label{fig:10}
\end{figure}

Take intersection points $\x \in \T_{\a} \cap \T_{\b}$ and $\y \in \T_{\a'} \cap \T_{\b'}$.
Furthermore, let $\t \in \T_{\b} \cap \T_{\d}$ and $\t' \in \T_{\b'} \cap \T_{\d'}$
be the distinguished intersection points. Then we have $\phi_{\xi}(\y) = (\y,\x'')$ and
$E_{\cT'}(\y) = F_{\cT'}(\y \otimes \t')$, while $E_{\cT}(\x) = F_{\cT}(\x \otimes \t)$.
Denote by~$\x''' \in \T_{\a \setminus \a'} \cap \T_{\d \setminus \d'}$
the point lying closest to $\x''$, which is also $F_{\cT \setminus \cT'}(\x''\otimes (\t \setminus \t'))$,
see Figure~\ref{fig:10}.
So what we need to check is
\[
\left(F_{\cT'}(\y \otimes \t'),\x''' \right) = F_{\cT}\left(\left(\y,\x''\right)\otimes\t \right).
\]
To see this, notice that the only domains in~$\cT$ that
contain both~$\t$ and~$\x''$ as incoming corners consist of the small triangles next to~$\x''$,
plus a domain in $(-\S',\bolda',\boldb')$, as all the curves $\bolda \setminus \bolda'$, $\boldb \setminus \boldb'$, and $\boldd \setminus \boldd'$
are used up in the small triangles. These small triangles have corners $\x''$, $\t \setminus \t'$, and~$\x'''$, and appear in the count
for~$F_{\cT}((\y,\x'') \otimes \t)$. The above mentioned domains in $(-\S',\bolda',\boldb')$
are exactly the ones counted in~$F_{\cT'}(\y \otimes \t')$. The argument is analogous to the proof
of the invariance of the map~$\Phi_{-\xi}$ under handleslides in $(-\S_0',\bolda_0',\boldb_0')$, found in~\cite{TQFT}.

The cases $k=1$ and $k=3$ are analogous to the invariance of $\Phi_{-\xi}$ under stabilization and destabilization
in~$(-\S_0',\bolda_0',\boldb_0')$. For example, if we attach a 1-handle to~$(-M',-\g')$ along a framed pair of points~$\bP$,
then we start out with an arbitrary diagram
$(-\S_0',\bolda_0',\boldb_0')$ that defines $(-M',-\g')$ and is adapted to some bouquet~$B(\bP)$ for~$\bP$,
and extend it to~$\H' = (-\S',\bolda',\boldb')$ that is contact compatible near~$\partial M'$, and make it adapted to~$\bP$.
Then we use the diagram
\[
\H'_{\bP} = (-(\S')^0 \cup A, \bolda' \cup \{\a\},\boldb' \cup \{\b\})
\]
to define the 1-handle map by $f_{\H',\bP}(\y) = \y \times \{\theta\}$, where $\theta \in \a \cap \b$ is the intersection point with
the larger relative grading. Note that we glue~$A$ to the subsurface~$-\S_0'$ of~$-\S'$. Then the result follows from
\[
(\y \times \{\theta\},\x'') = (\y,\x'') \times \{\theta\}.
\]
The case $k = 3$ is very similar, we start with a diagram for~$(-M',-\g')$ that represents the attaching sphere of the 3-handle,
as in Definition~\ref{def:adapted-to-sphere}, then ``turn around'' the argument for 1-handles. This proves the proposition.
\end{proof}

The second part, using the $\spinc$ structures, follows from the first part and Proposition \ref{prop:spinc}.
This concludes the proof of Theorem \ref{thm:funct}.
\end{proof}

\begin{prop} \label{prop:id}
If $\W = (W,Z,[\xi])$ is the trivial cobordism from $(M,\g)$ to $(M,\g)$, then $F_{\W}$ is the identity of $\SFH(M,\g)$.
Furthermore, the restriction map from $\spinc(\W)$ to $\spinc(M,\g)$ is an isomorphism, and for every $\s \in \spinc(\W)$,
the map~$F_{\W,\s}$ is the identity of $\SFH(M,\g,\s|_M)$.
\end{prop}

\begin{proof}
This follows from Propositions~\ref{prop:specialid} and~\ref{prop:extends}.
\end{proof}

\begin{rem}
If $(M',\g')$ is obtained from $(M,\g)$ using a convex decomposition, then we can view $(-M',-\g')$
as a sutured submanifold of $(-M,-\g)$, and there is a natural contact structure $\z$ on $M \setminus \Int(M')$,
with convex boundary and dividing set~$-\g \cup -\g'$, see page~3 of~\cite{TQFT}.
Hence we can view this as a boundary cobordism~$\W$ from~$(M',\g')$ to $(M,\g)$.
For nice sutured manifold decompositions, $F_{\W} = \Phi_{\z}$ is an embedding by \cite{decomposition}.
\end{rem}

\subsection{Duality and blow-up}
Let$(\S,\bolda,\boldb)$ be a diagram of the balanced sutured manifold~$(M,\g)$.
Recall that there is a canonical affine isomorphism between $\spinc(M,\g)$ and~$\spinc(-M,-\g)$.
It follows from \cite[Proposition~2.12]{decateg} that there is a natural bilinear pairing
\[
CF(\S,\bolda,\boldb,\s) \otimes CF(-\S,\bolda,\boldb,\s) \to \Z_2.
\]
On the generators $\x$, $\y \in \T_{\a} \cap \T_{\b}$ with
$\s(\x) = \s(\y) = \s$, it is
given by~$\langle\,\x,\y\,\rangle = 1$ if~$\x = \y$, and~$\langle\,\x,\y\,\rangle = 0$ otherwise. Just as in~\cite[Lemma~5.1]{OSz10},
for every element $a \in CF(\S,\bolda,\boldb,\s)$ and~$b \in CF(-\S,\bolda,\boldb,\s)$, we have
\[
\langle\, a, \partial_{-\S} b\,\rangle = \langle\, \partial_{\S}\, a,b \,\rangle.
\]
So there is an induced pairing $\langle\,\, ,\,\rangle$ on $\SFH(M,\g,\s) \otimes \SFH(-M,-\g,\s)$.
Then $F_{\W}$ satisfies the following duality result.

\begin{thm} \label{thm:dual}
Let $\W$ be a special cobordism from $(M_0,\g_0)$ to $(M_1,\g_1)$.
In Remark~\ref{rem:upside}, we defined the cobordism~$\ol{\W}$ from $(-M_1,-\g_1)$ to $(-M_0,-\g_0)$,
obtained by ``turning~$\W$ upside down.'' Then the map $F_\W$ is dual to~$F_{\ol{\W}}$.
More precisely, for every $\s \in \spinc(\W)$ with~$\s|_{M_i} = \ft_i$ for~$i \in \{1,2\}$,
and every $x \in \SFH(M_0,\g_0,\ft_0)$ and $y \in \SFH(-M_1,-\g_1,\ft_1)$, we have
\[
\langle\, F_{\W,\s}(x), y \,\rangle_1 =
\langle\, x, F_{\ol{\W},\s}(y) \,\rangle_0.
\]
Here $\langle \, , \, \rangle_i$ is the pairing on~$\SFH(M_i,\g_i,\ft_i) \otimes \SFH(-M_i,-\g_i,\ft_i)$ for~$i \in \{0,1\}$.
\end{thm}

\begin{proof}
This is completely analogous to \cite[Theorem 3.5]{OSz10}.
\end{proof}

\begin{qn} \label{qn:dual}
Does Theorem \ref{thm:dual} hold for arbitrary balanced cobordisms?
\end{qn}

We will get back to this question shortly. Next, we generalize the hat version of the blow-up formula \cite[Theorem 3.7]{OSz10} to our situation.

\begin{thm}
Let $\W_0$ be the trivial cobordism from the balanced sutured manifold~$(M,\g)$ to itself. Consider the
blowup $\W = \W_0 \# \overline{\C P}^2$, where we take an internal connected sum, and write~$E$ for the exceptional
divisor. Furthermore, fix a
$\spinc$ structure $\s \in \spinc(\W)$, and let $\ft = \s|_{M \times \{0\}}$.
Then $\s|_{M \times \{1\}}$ is also $\ft$, and the map
\[
F_{\W,\s} \colon \SFH(M,\g,\ft) \to \SFH(M,\g,\ft)
\]
is the identity if $\langle\,c_1(\s),E \,\rangle = \pm 1$,
and is zero otherwise.
 \end{thm}

\begin{proof}
This follows from the same local computation as \cite[Theorem 3.7]{OSz10}.
\end{proof}

\subsection{Sutured Floer homology as a TQFT}
An axiomatic description of topological quantum field theories, in short TQFTs, was first given by Atiyah~\cite{Atiyah}.
As explained by Blanchet~\cite{Blanchet}, an $(n+1)$-dimensional TQFT over a field $\F$ is a symmetric monoidal
functor from the cobordism category of $n$-manifolds to the category of finite dimensional vector spaces over $\F$.

\begin{thm} \label{thm:TQFT}
The functor $\SFH \colon \BSut \to \Vect$ is a $(3+1)$-dimensional TQFT over $\Z_2$ in the sense of Atiyah~\cite{Atiyah} and Blanchet~\cite{Blanchet}.
\end{thm}

\begin{proof}
The axioms of Atiyah~\cite{Atiyah} and Blanchet~\cite{Blanchet} are equivalent. Since Blanchet takes the cobordism
point of view, and we are working with cobordisms, we are going to check the axioms in~\cite{Blanchet}.

In our case, we have a finitely generated $\Z_2$ vector space $\SFH(M,\g)$ assigned to every balanced sutured manifold $(M,\g)$, and
a linear map
\[
F_{\W} \colon \SFH(M_0,\g_0) \to \SFH(M_1,\g_1)
\]
corresponding to every balanced cobordism $\W$ from
$(M_0,\g_0)$ to $(M_1,\g_1)$.

Axiom (1) is naturality. We showed in~\cite[Theorem~1.9]{naturality} that every orientation preserving diffeomorphism
$\phi \colon (M_0,\g_0) \to (M_1,\g_1)$ induces an isomorphism
\[
\phi_* \colon \SFH(M_0,\g_0) \to \SFH(M_1,\g_1),
\]
cf.~Section~\ref{sec:naturality}.
Furthermore, this assignment is functorial; i.e., $(\psi \circ \phi)_* = \psi_* \circ \phi_*$.
The naturality of the cobordism maps $F_{\W}$ was stated in Theorem \ref{thm:finalnat}.

Axiom (2) is functoriality; i.e.,
\[
F_{\W_2 \circ \W_1} = F_{\W_2} \circ F_{\W_1}.
\]
This was shown in Theorem~\ref{thm:funct}.

Axiom (3), normalization, states that if $\W$ is the trivial cobordism from $(M,\g)$ to $(M,\g)$, then
\[
F_{\W} \colon \SFH(M,\g) \to \SFH(M,\g)
\]
is the identity. This is the content of Proposition~\ref{prop:id}.

Axiom (4) is multiplicativity. It states that there are functorial isomorphisms
\[
\SFH((M_1,\g_1) \sqcup (M_2,\g_2)) \cong \SFH(M_1,\g_1) \otimes \SFH(M_2,\g_2),
\]
which easily follows from the definitions.
%If we are working over $\Z$, a K\"unneth-type formula is satisfied, although the base ring $\Z$ can essentially be replaced by $\Q$ and $\Z_p$.
Furthermore, we set $\SFH(\emptyset) = \Z_2$, where~$\emptyset$ is the empty balanced sutured manifold.
It is straightforward to check that these isomorphisms fit into the commutative diagrams in~\cite{Blanchet}
that describe associativity and the unit.
Finally, using the above identifications, we also have
\[
F_{\W_1 \sqcup \W_2} = F_{\W_1} \otimes F_{\W_2}.
\]

Axiom (5), symmetry, states that the isomorphism
\[
\SFH((M_1,\g_1) \sqcup (M_2,\g_2)) \cong \SFH((M_2,\g_2) \sqcup (M_1,\g_1))
\]
induced by the obvious diffeomorphism $(x,i) \mapsto (x,1-i)$ for~$i \in \{1,2\}$
corresponds to the isomorphism of vector spaces
\[
\SFH(M_1,\g_1) \otimes \SFH(M_2,\g_2) \cong \SFH(M_2,\g_2) \otimes \SFH(M_1,\g_1)
\]
mapping~$a \otimes b$ to~$b \otimes a$.
This is also straightforward.
\end{proof}

\begin{rem}
Similar axioms hold if we are dealing with balanced sutured manifolds endowed with $\spinc$ structures, and morphisms being
balanced cobordisms endowed with relative $\spinc$ structures. These do not form a proper category, as the composition of $(\W_1,\s_1)$ and $(\W_2,\s_2)$ does not have a well defined relative $\spinc$ structure on it. But if we replace functoriality with the formula in Theorem~\ref{thm:funct}, we do get a type of TQFT.
\end{rem}

Recall that~\cite[Axiom 2]{Atiyah} states that $\SFH(M,\g)$ and $\SFH(-M,-\g)$ are dual vector spaces.
This was proved in~\cite[Proposition~2.14]{decateg}.
Furthermore, there has to be a functorial isomorphism between $\SFH(-M,-\g)$ and $\SFH(M,\g)^*$.
This follows from the set of axioms in~\cite{Blanchet}.
Indeed, if we view the trivial cobordism $\W$ from $(M,\g)$ to $(M,\g)$
as a cobordism from $(M,\g) \sqcup (-M,-\g)$ to $\emptyset$,
then we get a pairing
\[
\langle\,\,,\,\rangle' \colon \SFH(M,\g) \otimes \SFH(-M,-\g) \to \Z_2,
\]
which is non-degenerate by~\cite[Theorem 2.1.1]{TuTQFT}, see also the right hand-side of Figure~\ref{fig:11}.

\begin{conj} \label{conj:pairing}
The pairing $\langle\,\,,\,\rangle'$ agrees with the pairing $\langle\,\,,\,\rangle$ appearing in Theorem~\ref{thm:dual}.
\end{conj}

\begin{cor}
The above conjecture would give an affirmative answer to Question~\ref{qn:dual}.
\end{cor}

\begin{proof}
\begin{figure}[tbm]
\includegraphics{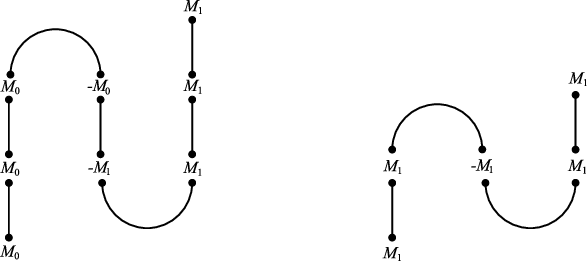}
\caption{On the left, a cobordism is written as a product of its reverse plus other simpler pieces. The maps induced by
these simpler pieces can be related to each other using the decomposition of the trivial cobordism on the right.} \label{fig:11}
\end{figure}

Note that Theorem~\ref{thm:dual}
is true for arbitrary balanced cobordisms if we replace $\langle\,\,,\,\rangle$ with $\langle\,\,,\,\rangle'$. To see this,
note that a cobordism $\W$ from $(M_0,\g_0)$ to $(M_1,\g_1)$ can be decomposed as follows.
Let $\overline{\W}$ be $\W$ viewed as a cobordism from~$(-M_1,-\g_1)$ to~$(-M_0,-\g_0)$,
and for $i \in \{0,1\}$ let $\I_i$ be the trivial cobordism from $(M_i,\g_i)$ to itself.
Furthermore, let $\I_i'$ be $\I_i$ viewed as a cobordisms from $\emptyset$ to $(-M_i,-\g_i) \sqcup (M_i,\g_i)$ and
$\I_i''$ is $\I_i$ viewed as a cobordism from $(M_i,\g_i) \sqcup (-M_i,-\g_i)$ to $\emptyset$. Then
\[
\W = (\I_0'' \sqcup \I_1) \circ (\I_0 \sqcup \overline{\W} \sqcup \I_1) \circ (\I_0 \sqcup \I_1'),
\]
see the left-hand side of Figure~\ref{fig:11}. Note that we can compute $F_{\I_1'}$ using the decomposition
\[
\I_1 = (\I_1 \sqcup \I_1') \circ (\I_1'' \sqcup \I_1),
\]
see the right-hand side of Figure~\ref{fig:11}.
\end{proof}

Using the same trick, one can rearrange the ingoing and outgoing ends of any balanced cobordism $\W$, and relate the induced
map to $F_{\W}$ using $\langle\,\,,\,\rangle'$.

If we were to follow the axioms of Atiyah~\cite{Atiyah}, instead of maps induced by balanced cobordisms, we would need to talk about elements $\EH(\W') \in \SFH(M,\g)$ associated to balanced cobordisms $\W'$ from $\emptyset$ to $(M,\g)$. It was explained in~\cite{Atiyah} how to
translate between the two approaches. If we are given cobordism maps,
to a balanced cobordism~$\W'$ from $\emptyset$ to $(M,\g)$,
we can associate the element~$\EH(\W') = F_{\W'}(1)$, where $1 \in \SFH(\emptyset) \cong \Z_2$.

In the other direction, if $\W$ is a balanced cobordism from $(M_0,\g_0)$ to $(M_1,\g_1)$, then we can also view $\W$ as a cobordism $\W'$ from $\emptyset$ to $(-M_0,-\g_0) \sqcup (M_1,\g_1)$.
Hence
\[
\begin{split}
\EH(\W') &\in \SFH((-M_0,-\g_0) \sqcup (M_1,\g_1)) \cong \\
SFH(M_0,\g_0)^* \otimes \SFH(M_1,\g_1) &\cong \text{Hom}(\SFH(M_0,\g_0),\SFH(M_1,\g_1)).
\end{split}
\]
So we do get a homomorphism $F_{\W}$ from $\SFH(M_0,\g_0)$ to $\SFH(M_1,\g_1)$ induced by the balanced cobordism $\W$.
It is important to note that if we decide to follow Atiyah~\cite{Atiyah},
we need to use the pairing $\langle\,\,,\,\rangle'$ -- and not $\langle\,\,,\,\rangle$ -- to identify $\SFH(-M,-\g)$
with $\SFH(M,\g)^*$.

\begin{rem} \label{rem:gluing}
Another consequence of Conjecture~\ref{conj:pairing} would be a new definition of the maps $F_{\W}$, only using
the $\EH$ class in $\SFH$ and special cobordism maps. Indeed, let $\W = (W,Z,[\xi])$ be a balanced cobordism from $(M_0,\g_0)$ to $(M_1,\g_1)$.
Let~$\g = -\g_0 \cup \g_1$. If we view $\W$ as a cobordism $\W'$ from $\emptyset$ to $(-M_0,-\g_0) \sqcup (M_1,\g_1)$, then, by definition,
\[
\EH(\W') = F_{\W'}(1) = F_{\W_1'}(\EH(Z,\g,\xi)),
\]
where $\W_1'$ is the obvious special cobordism from $(Z,\g)$ to $(-M_0,-\g_0) \sqcup (M_1,\g_1)$.
But we saw above that, using $\langle\,\,,\,\rangle'$,
we can view $\EH(\W')$ as a homomorphism from $\SFH(M_0,\g_0)$ to $\SFH(M_1,\g_1)$, which agrees with $F_{\W}$.
A priori, to compute $\langle\,\,,\,\rangle'$, one already needs to use some gluing map, but Conjecture~\ref{conj:pairing} would
eliminate this problem. In particular, if $\W$ is a boundary cobordism, then we would get a new definition for the gluing map
$F_{\W} = \Phi_{-\xi}$.

However, if one starts out with this definition of~$F_{\W}$, then Axioms~2 and~3 seem to be very difficult to prove without using
the theory of gluing maps~\cite{TQFT}.
\end{rem}

\begin{rem}
A cobordism $\W$ from the empty sutured manifold to itself is a smooth oriented 4-manifold $W$ with contact boundary $(Z,\xi)$.
Here $\xi$ is a positive contact structure when $Z = \partial W$ is given the boundary orientation. Then, for every relative $\spinc$
structure $\s \in \spinc(\W)$, we have a map $F_{\W,\s} \colon \Z_2 \to \Z_2$. This can be computed using just the contact class and the cobordism map in Heegaard Floer homology. More precisely, let $W_0$ be the cobordism from $Y$ to $S^3$ obtained by removing an open ball from the interior of $W$, and set $\s_0 = \s|_W$. Then
\[
F_{\W,\s}(1) = F_{W_0,\s_0}(c(Y,\xi)) \in \HFh(S^3) \cong \Z_2,
\]
where $c(Y,\xi) \in \HFh(Y,\s(\xi))$ is the contact invariant defined by Ozsv\'ath and Sza\-b\'o~\cite{OSz4},
and agrees with $\EH(Y,\xi)$ according to~\cite{EH}.
\end{rem}

Recall that in Section~\ref{sec:links},
we defined a functor $\W \colon \link \to \BSut$, and for a decorated link $(Y,L,P) \in \link$, we have
\[
\SFH(\W(Y,L,P)) \cong \HFLh(Y,L) \otimes V^{\otimes d}.
\]
Then Theorem~\ref{thm:TQFT} and Proposition~\ref{prop:Wfunct} give the following.

\begin{cor}
The functor
\[
\SFH \circ \W \colon \link_0 \to \Vect
\]
is also a TQFT in the sense of Atiyah~\cite{Atiyah}
and Blanchet~\cite{Blanchet}, making link Floer homology functorial.
\end{cor}

\subsection{Weinstein cobordisms}
We conclude with defining Weinstein cobordisms, and showing that they preserve the $\EH$ class.
The following definition extends the corresponding notions of Bourgeois et.~al~\cite{BEE} and
Sidel~\cite{Seidel} to sutured manifolds.

\begin{defn}
Suppose that $(M_0,\g_0,\z_0)$ and $(M_1,\g_1,\z_1)$ are contact manifolds.
A \emph{Liouville cobordism from $(M_0,\g_0,\z_0)$ to $(M_1,\g_1,\z_1)$} is a pair $(\W,\theta)$,
where $\W = (W,Z,[\xi])$ is a balanced cobordism from $(M_0,\g_0)$ to $(M_1,\g_1)$,
and $\theta$ is a 1-form on~$W$ satisfying the following properties:
\begin{enumerate}
\item $\omega= d\theta$ is symplectic,
\item the \emph{Liouville vector field} $X$ defined by $\iota_X \omega = \theta$
is transverse to every face of $\partial W$,  enters $W$ through $M_0$, and exits $W$ through $Z \cup M_1$,
\item $\xi = \ker(\theta|_Z)$, $\z_0 = \ker(\theta|_{M_0})$, and $\z_1 = \ker(\theta|_{M_1})$.
\end{enumerate}

The Liouville cobordisms $(\W,\theta)$ and $(\W',\theta')$ are equivalent (diffeomorphic) if there is an
equivalence (diffeomorphism) $d$ between $\W$ and $\W'$ such that $d^*(\theta') = \theta$.
\end{defn}

\begin{ex} \label{ex:sympl}
An instance of a Liouville cobordism is the \emph{symplectization} of a contact manifold $(M,\g,\z)$.
Let $\W = (W,Z,[\xi])$ be the trivial cobordism from~$(M,\g)$ to itself.
Since $\partial M$ is a convex surface, there is a contact vector field $v$ on $M$ which is transverse to $\partial M$,
points out of $M$, and $v \in \z$ exactly along $\g$. Let $\a$ be a contact 1-form
such that $\z = \ker(\a)$. Then $\mathcal{L}_v\a = \mu \a$ for some function $\mu \colon M \to \R$.
By the Cartan formula, this is equivalent to $\iota_v d\a + d\iota_v\a = \mu \a$.
The dividing set $\g$ coincides with the zero set of $\a(v)$ along $\partial M$.
If we multiply $v$ by a sufficiently small
positive scalar, then we can assume that $\mu < 1/2$. Take
\[
\theta = e^t\a - d(e^t \a(v)).
\]
Then
\[
\omega = d\theta = e^t dt \wedge \a + e^t d\a
\]
is symplectic since
$\omega \wedge \omega = 2e^{2t} dt \wedge \a \wedge d\a$ is nowhere zero on $W$.

We claim that
\[
X = (1-\mu)\partial_t + v
\]
is the Liouville vector field for $\theta$. Indeed,
\[
\iota_X \omega = (1-\mu)\iota_{\partial_t} \omega + \iota_{v} \omega.
\]
Since $\iota_{\partial_t} \omega = e^t \a$ and
\[
\begin{split}
\iota_{v} \omega = - e^t \a(v) dt + e^t \iota_v d\a &= \\
- e^t \a(v) dt + e^t(\mu\a - d\iota_v \a) &=  e^t\mu\a - d(e^t\a(v)),
\end{split}
\]
we can conclude that $\iota_X \omega = \theta$. As $1-\mu >0$, the vector field $X$ points into $W$ along $M_0$ and points out of $W$ along $M_1$.
Furthermore, since $v$ points out of $M$ along $\partial M$, we see that $X$ points out of $W$ along $Z = \partial M \times I$.

For $i \in \{0,1\}$, the contact structure
\[
\ker(\theta|_{M_i}) = \ker(\a - d(\a(v)))
\]
is isotopic to $\z$ through contact structures. To see this, note that, for every $0 \le \varepsilon \le 1$,
the same argument as above shows that
\[
\theta^{\varepsilon} = e^t\a - d(e^t \a(\varepsilon v))
\]
induces a contact
structure $\z^{\varepsilon} = \ker(\theta^{\varepsilon}|_{M_i})$ on $M_i$. For $\varepsilon = 0$, we have
$\z^0 = \ker(\a) = \z$, while $\z^1 = \ker(\theta|_{M_i})$.

Finally, we show that $\ker(\theta|_Z) = \xi$.
Since $\mathcal{L}_{\partial_t}d = d \mathcal{L}_{\partial_t}$,
we have $\mathcal{L}_{\partial_t}\theta = \theta$. Hence $\partial_t|_Z$
is a contact vector field on $(Z,\theta|_Z)$. So, for every $0 \le t \le 1$,
the surface~$\partial M \times \{t\}$ is convex. The dividing set on this surface
is given by the equation $\theta(\partial_t) = 0$. However,
\[
\theta(\partial_t) = - d(e^t \a(v))(\partial_t) =
-e^t \a(v),
\]
and the function $\a(v)|_{\partial M}$ vanishes exactly along $\g$.
\end{ex}

\begin{defn}
Let $(\W,\theta)$ from $(M_0,\g_0,\z_0)$ to $(M_1,\g_1,\z_1)$ be a Liouville cobordism,  and let $X$ be the corresponding
Liouville vector field. A function $H \colon W \to \R$ is a \emph{Lyapunov function for $X$} if $H$ is a smooth Morse function,
and there exists a constant $\d \in \R_+$ and a Riemannian metric on $W$ such that $dH(X) \ge \d |X|^2$.
\end{defn}

\begin{ex}
If $(\W,\theta)$ is a symplectization, then $H(x,t) = t$ is a Lyapunov function for $X = (1-\mu) \partial_t + v$.
\end{ex}

\begin{defn}
We say that the Liouville cobordism $(\W,\theta)$ from~$(M_0,\g_0,\z_0)$ to~$(M_1,\g_1,\z_1)$ is \emph{Weinstein}
if there exists a Lyapunov function $H$ for the Liouville vector field $X$ such that
\begin{enumerate}
\item a collar neighborhood $M_0 \times I$ of~$M_0 = M_0 \times \{0\}$ is a symplectization of~$(M_0,\g_0,\z_0)$,
as in Example~\ref{ex:sympl},
\item $H(x,t) = t$ for $(x,t) \in M_0 \times I$,
\item there is a collar neighborhood $\partial M_0 \times [0,2]$ of $\partial M_0$ in $Z$ that extends $\partial M_0 \times I$,
where $H(x,t) = t$,
\item $H \equiv 2$ on $(Z \cup M_1) \setminus (\partial M_0 \times [0,2])$,
\item $H$ has no critical points on $\partial W$.
\end{enumerate}
\end{defn}

\begin{rem} \label{rem:handle}
Since $X$ points out of $W$ along $Z \cup M_1$, the negative gradient flow lines of $X$ can only exit $W$ along $M_0$.
As explained in~\cite{BEE}, all critical points of the Lyapunov function $H$ have Morse index at most two, and the stable manifolds intersected
with regular levels $H^{-1}(c)$ are isotropic for the induced contact structure $\ker(\theta|_{H^{-1}(c)})$.
Using the work of Weinstein~\cite{Weinstein}, we see that one can build $W$,
viewed as a special cobordism from $M_0$ to $Z \cup M_1$, from the symplectization of $(M_0,\g_0,\z_0)$
by attaching Weinstein 1- and 2-handles. A Weinstein 1-handle attachment changes the
boundary by removing two standard contact balls, and gluing a standard contact~$S^2 \times I$,
or equivalently, by taking a connected sum with the standard contact~$S^2 \times S^1$.
Each Weinstein 2-handle is attached along some Legendrian knot $K$ with framing $tb(K) - 1$.
\end{rem}

\begin{thm} \label{thm:Wein}
Let $(\W, \theta)$ be a Weinstein cobordism from the contact manifold $(M_0,\g_0,\z_0)$ to $(M_1,\g_1,\z_1)$.
Let $\s \in \spinc(\W)$ be the $\spinc$ structure associated with~$\omega= d\theta$.
As explained in Remark~\ref{rem:upside}, we can view~$\W$ as a balanced cobordism~$\ol{\W}$ from $(-M_1,-\g_1)$ to $(-M_0,-\g_0)$.
Then
\[
F_{\ol{\W},\s}(\EH(M_1,\g_1,\z_1)) =  \EH(M_0,\g_0,\z_0).
\]
\end{thm}

\begin{rem}
Recall that, for $i \in \{0,1\}$,
\[
\EH(M_i,\g_i,\z_i) \in \SFH \left(-M_i,-\g_i,\s_{\z_i}\right).
\]
Furthermore, $\s|_{M_i} = \s_{\z_i}$, and
\[
F_{\ol{\W},\s} \colon \SFH \left(-M_1,-\g_1,\s_{\z_1} \right) \to \SFH \left(-M_0,-\g_0,\s_{\z_0} \right).
\]
\end{rem}

\begin{proof}
Consider the cobordism $\ol{\W} = (W,Z,[-\xi])$ from $(-M_1,-\g_1)$ to $(-M_0,-\g_0)$.
First, suppose that $Z$ has no isolated components, and set $N = M_1 \cup Z$.
Then $F_{\ol{\W}} = F_{\ol{\W}_1} \circ \Phi_{\xi}$, where $\ol{\W}_1$ is a special cobordism from $(-N,-\g_0)$ to $(-M_0,-\g_0)$.
By Theorem~\ref{thm:glue}, we have
\[
\Phi_{\xi}(\EH(M_1,\g_1,\z_1)) = \EH(N,\g_0,\z_1 \cup \xi).
\]

By turning it upside down, we can view $\ol{\W}_1$ as a special cobordism $\W_1$ from $(M_0,\g_0)$ to~$(N,\g_0)$.
As explained in Remark~\ref{rem:handle}, we can build $\W_1$ from the symplectization of $(M_0,\g_0,\z_0)$ by first attaching
Weinstein 1-handles, then attaching Weinstein 2-handles along a Legendrian link $L$ with framing $tb(L)-1$.
Hence, we are done if we prove the result when $\W$ is a Weinstein 1- or 2-handle cobordism.

Suppose that $\W$ is a Weinstein 1-handle cobordism from $(M_0,\g_0,\z_0)$ to
\[
(M_0 \# (S^1 \times S^2),\g_0,\z_1).
\]
The contact structure $\z_1$ agrees with $\z_0$ minus a standard contact ball on $M_0 \setminus B^3$, and is the unique tight contact structure minus  a standard contact ball $\xi_{std}$ on~$(S^1 \times S^2) \setminus B^3$. Fix a partial open book decomposition $(S_0,h_0 \colon P_0 \to S_0)$
for~$(M_0,\g_0,\z_0)$, and let $(\S_0,\bolda_0,\boldb_0)$ be the associated balanced diagram. Furthermore, let $(S,h \colon P \to S)$ be the partial
open book decomposition of
\[
((S^1 \times S^2)(1), \xi_{std})
\]
described in~\cite[Example 1]{Ozbagci}. It agrees with the partial open book of
Example~\ref{ex:partial}, except that
the map $h$ is the identity of $P$. We are going to denote by $(\S,\a,\b)$ the corresponding balanced diagram
of $(S^1 \times S^2)(1)$. Then $\S = T^2 \setminus B^2$, the curves~$\a$ and~$\b$ intersect in exactly two points,
and~$\a$ is a small Hamiltonian translate of~$\b$. Let $y \in \a \cap \b$ be the intersection point with the smaller relative grading.
If we take the boundary connected sum of~$S_0$ and~$S$ along $\partial S_0 \setminus P_0$ and $\partial S \setminus P$,
then we get a partial open book decomposition for $(M_0 \# (S^1 \times S^2),\g_0,\z_1)$, which induces the balanced
diagram $(\S_0 \natural \S, \bolda_0 \cup \{\a\},\boldb_0 \cup \{\b\})$. Let $\x \in \T_{\a_0} \cap \T_{\b_0}$ be the
distinguished generator representing $\EH(M_0,\g_0,\z_0)$, and note that~$y$ represents the class $\EH((S^1 \times S^2)(1),\xi_{std})$.
Then $\x \times \{y\}$ represents $\EH(M_0 \# (S^1 \times S^2), \g_0,\z_1)$.
The curves $\a$ and $\b$ bound a periodic domain which represents a sphere $\{p\} \times S^2$ inside $S^1 \times S^2$.
The cobordism $\ol{\W}$ corresponds to a 3-handle attached along $\{p\} \times S^2$,
so by Definition~\ref{defn:3handle}, the map $F_{\ol{\W}}$ takes $\x \times \{y\}$ to $\x$, which proves the claim for Weinstein 1-handles.

Now assume that $\W$ is a Weinstein 2-handle cobordism corresponding to a Legendrian knot $K$ in $(M_0,\g_0,\z_0)$.
Then the proofs of~\cite[Proposition~4.4]{contact} and~\cite[Theorem~3.5]{OSz10} imply that the $\EH$ class is preserved
by $F_{\ol{\W}}$. More concretely, in the proof of~\cite[Proposition~4.4]{contact},
Honda, Kazez, and Mati\'c construct a partial open book decomposition $(S,h\colon P \to S)$ for $(M_0,\g_0,\z_0)$
which contains the Legendrian knot $K$ inside $P$. This gives rise to a triple diagram $(\S,\bolda,\boldb,\boldd)$
which is subordinate to some bouquet for~$K$. As described by Ozsv\'ath and Szab\'o in the proof of~\cite[Theorem~3.5]{OSz10},
if~$(\S,\bolda,\boldb,\boldd)$ corresponds to the 2-handle cobordism~$\W$,
then~$(-\S,\bolda,\boldd,\boldb)$ corresponds to $\ol{\W}$.
We end up with the configuration depicted in Figure~\ref{fig:10},
where there is a distinguished triangle mapping the generator representing $\EH(M_1,\g_1,\z_1)$
to the one representing $\EH(M_0,\g_0,\z_0)$.

If~$Z$ does have isolated components, then $F_{\ol{\W}} = F_{\ol{\W}'}$, where $\ol{\W}'$ is the cobordism from $(-M_1,-\g_1)$
to $(-M_0,-\g_0) \sqcup (B,\d)$ given by Definition~\ref{defn:phi}.
Note that $B \subset Z$, and $\ker(\theta|_B) = \xi|_B$ is a union of standard contact balls.
Let~$\W'$ be~$\ol{\W}'$ viewed as a cobordism from $(M_0,\g_0) \sqcup (-B,-\d)$ to $(M_1,\g_1)$.
Then $(\W',\theta)$ is not a Liouville cobordism, since the Liouville vector field $X$ points out of $W$
along $B$. To fix this, attach a Weinstein 1-handle to each component of $B$ along one of its feet. Then $\theta$ plus
the Liouville 1-forms on the 1-handles give a 1-form $\theta'$ such that the new Liouville vector field~$X'$
points in along the free feet of the 1-handles. The contact structure~$\xi$ on~$Z$ is left unchanged, since $I \times S^2$
has a unique tight contact structure. The cobordism $(\W',\theta')$ is also Weinstein, since a Lyapunov
function~$H$ on~$\W$ extends to the 1-handles with a unique index one critical point in each.
Let~$\xi_0 = \ker(\theta'|_{-B})$, then $(-B,\xi_0)$ is also a union of standard contact balls.
If we apply the previous part to the Weinstein cobordism $(\W',\theta')$, then we get that
\[
F_{\ol{\W}',\s'}(\EH(M_1,\g_1,\z_1)) =  \EH(M_0,\g_0,\z_0) \otimes \EH(-B,-\d,\xi_0),
\]
where $\s' = \s|_{\W'}$. As $\EH(-B,-\d,\xi_0) = 1 \in \SFH(B,\d)$, the right-hand side
maps to $\EH(M_0,\g_0,\z_0)$ under the isomorphism
\[
\SFH(-M_0,-\g_0) \otimes \SFH(B,\d) \to \SFH(-M_0,-\g_0),
\]
which concludes the proof.
\end{proof}

\begin{cor}
Let $(W,\theta)$ from $(M_0,\g_0,\z_0)$ to $(M_1,\g_1,\z_1)$ be a Weinstein cobordism. If $\EH(M_0,\g_0,\z_0) \neq 0$, then
$\EH(M_1,\g_1,\z_1) \neq 0$.
\end{cor}

% ----------------------------------------------------------------
\bibliographystyle{amsplain}
\bibliography{topology}
\end{document}